\documentclass[11pt,reqno]{article} \usepackage{authblk}

\usepackage{graphicx}
\usepackage[top=2.5cm, bottom=2.5cm, left=
3cm, right=3cm] {geometry}
\usepackage{amscd,amsmath,amstext,amsfonts,amsbsy,
amssymb,amsthm,mathrsfs,float,latexsym}
\usepackage{xcolor}
\usepackage{multicol,graphicx,array,multirow,color,lineno}
\usepackage{sidecap,url,fancybox,fancyhdr}
\usepackage{xcolor}
\usepackage[colorlinks=true,linkcolor=blue,citecolor=magenta]{hyperref}

\usepackage{tikz}
\usepackage{amsmath}

\usepackage{subcaption}

\usepackage{enumitem} 
\newcommand{\ue}{u_{\eps}}
\newcommand{\ve}{v_{\eps}}

\newcommand{\di}{\nabla \cdot }
\newcommand{\Tm}{T_{\max}} 
\newcommand{\Tme}{T_{\max,\varepsilon}} 
 
\newcommand{\intQTm}{\iint_{ Q_{\Tm}}} 

\newcommand{\intQT}{\iint_{Q_{T}}}
\newcommand{\intQt}{\iint_{Q_{t}}}

\newcommand{\R}{\mathbb R}

\newcommand{\pa}{\partial}

\newcommand{\eps}{\varepsilon}

\renewcommand{\div}{\mathrm{div}}

\newcommand{\LO}[1]{L^{#1}(\Omega)}
\newcommand{\LQ}[1]{L^{#1}(Q_T)}

\newcommand{\intO}{\int_{\Omega}}

\def\({\left(}
\def\){\right)}

\newtheorem{theorem}{{\bf Theorem}}[section]
\theoremstyle{definition} \newtheorem{definition}{\bf Definition}[section]
\theoremstyle{assumption} \newtheorem{assumption}{\bf Assumption}[section]
\newtheorem{corollary}{{\bf Corollary}}[section]
\theoremstyle{plain}
\newtheorem{lemma}{Lemma}[section]

\newtheorem{remark}{Remark}[section]

\title{\bf \Large Global weak solvability for doubly degenerate nutrient taxis system in physical dimension}

\author{Bao-Ngoc Tran$^{a,b}$\footnote{Email: bao-ngoc.tran@uni-graz.at, tranbaongoc@hcmuaf.edu.vn} \, and Juan Yang$^{c}$\footnote{Email: yangjuan12@sjtu.edu.cn, yangjuan0912@gmail.com}}
 
\date{} 

\begin{document}

\maketitle

\vspace*{-1cm}

\begin{center}
{\small
$^{a}$Department of Mathematics and Scientific Computing, University of Graz,  Heinrichstrasse 36, 8010 Graz, Austria \vspace{0.15cm}\\ 
$^{b}$Department of Mathematics, Faculty of Science, Nong Lam University, Ho Chi Minh City, Vietnam \vspace{0.15cm}\\
$^{c}$School of Mathematical Sciences, Shanghai Jiao Tong University, Shanghai 200240, China
}
\end{center}

\begin{abstract} 
Motivated by the study of bacteria's response to environmental conditions, we consider the doubly degenerate nutrient taxis system
\begin{align*}
\begin{cases}
 u_t = \nabla \cdot ( uv \nabla u) - \chi \nabla\cdot (u^{\alpha} v\nabla v) + \ell uv,\\
 v_t = \Delta v - u v ,
\end{cases}  
\end{align*}
subjected to no-flux boundary conditions and smooth initial data, where 
$\alpha \ge 0$ is the bacterial response parameter. Global solvability of weak solutions to this taxis system appears to be highly challenging due to the difficulty of quantifying the dissipation of the doubly degenerate diffusive flux, the strong chemotactic effect, particularly as $\alpha$ is close to $2$, and the dimensionally dependent limitation, which is the most difficult one. Recent findings on the global weak  solvability for the considered system are summarised as follows 
\begin{itemize}
 \item In [M. Winkler, \textit{Trans. Amer. Math. Soc.}, 2021] for  $\alpha=2$, $N=1$;
 
 \item In [M. Winkler, \textit{J. Differ. Equ.}, 2024] for $1\le\alpha\le 2$, $N=2$with initial data of small size if  $\alpha=2$;

 \item In [Z. Zhang, Y. Li, \textit{Math. Models Methods Appl. Sci.}, 2026] for $\alpha=2$, $N=2$; 

 \item In [G. Li, \textit{J. Differ. Equ.}, 2022] for $\frac{7}{6}<\alpha<\frac{13}{9}$, $N=3$; and

 \item In [X.M. De-Ji, A. Huang, Y. Wang, \textit{Eur. J. Appl. Math.}] for $\frac{3}{2}<\alpha<\frac{19}{12}$, $N=3$.  
 
\end{itemize}
Our work aims to provide a picture of global weak solvability for $0\le \alpha < 2$ in the physically dimensional setting $N=3$. As suggested by the analysis, it is divided into three separable cases, including (i) $0\le \alpha\le 1$:
 Weak chemotaxis effect; (ii)  $1<\alpha\le 3/2$: Moderate chemotaxis effect; and (iii) 
  $3/2<\alpha<2$: Strong chemotaxis effect. 

\medskip
 
\noindent \textbf{Keywords}: Chemotaxis system, Nutrient
taxis system, Doubly degenerate nutrient taxis system, Global existence, Bootstrap arguments. 
\end{abstract}

 
\section{Introduction}

\subsection{Problem formulation and history}

With crucial roles in biology and ecology, chemotaxis is well known for describing the biased movement of species along spatial gradients of stimulus concentrations. The study of chemotaxis goes back to microbiology questions, Keller-Segel \cite{keller1971model} and Berg-Brown \cite{berg1972chemotaxis}, about how bacteria move towards food sources. In the last several decades, with wide-ranging applications, such as in bacterial aggregation \cite{berg1972chemotaxis,erban2004individual}, cell invasion \cite{roussos2011chemotaxis,bellomo2015toward}, food chains \cite{tao2022existence,reisch2024global}, 
etc., chemotaxis systems have been investigated by many mathematicians, resulting in an extensive amount of literature, for example, \cite{horstmann20031970,painter2019mathematical,lankeit2020facing,arumugam2021keller} and references therein. 

\medskip

This work considers the complex responses of bacteria to nutrients in the surrounding environment, leading to a so-called doubly degenerate nutrient taxis system. 
Assume that a species of bacteria and its nutrients are considered in a convex and bounded domain $\Omega \subset \R^N$, which has a sufficiently smooth boundary $\Gamma$. In a nutrient-poor medium of a thin agar plate, and so $N=2$, the Bacillus subtilis strain forms densely branched, highly structured, stable patterns \cite{ohgiwari1992morphological}. To capture the experimental observation on the bacterial colony envelope that bacterial cells around the front are highly active but immotile either deeply inside the envelope due to starvation or in the outermost region due to quorum sensing, K. Kawasaki et al.  \cite{Kawasaki1997Modeling}
allowed the bacteria's diffusivity to be low when either the nutrient concentration or the bacterial density is getting low, resulting in the diffusive flux
\begin{align*}
- D(u,v) \nabla u \quad \text{with} \quad D(u,v) = \vartheta uv,
\end{align*}
where $ \vartheta >0$ depends on the media properties. It is useful to note that the original
model from Kawasaki et al. did not include the transport
terms due to taxis. As motivated by the experimental observations of the spatio-temporal patterns from Ben-Jacob and his collaborators \cite{Benjacob2000Cooperative,Cohen1996Chemotactic,Golding1998Studies}, the model in \cite{Kawasaki1997Modeling} was afterwards extended by Leyva-M\'alaga-Plaza \cite{leyva2013effects}, which included the chemotactic flux  
\begin{align*}
S(u,v) \chi(v) \nabla v \quad \text{with} \quad |S| \propto u D,
\end{align*}
where $S(u,v)$ represents the bacterial response function to the nutrient taxis. The notation $\chi$ represents the receptor law that models how cell-surface receptors detect nutrient signals and convert them into movement behaviours, which can be simplified to be a positive constant. The expression $|S| \propto u D$ means that the bacterial response function should
be proportional to the product of the bacterial density and its diffusivity. 

\medskip

For a more general setting with $N=1,2,3$, the derivation of the above diffusive and chemotactic fluxes, and particularly the incorporation of microscopic responses of individual cells into the chemotactic sensitivity, had been considered in \cite{plaza2019derivation} via a parabolic limit of velocity-jump processes, which was postulated based on empirical considerations. The movement of bacterial cells was modelled by the so-called run-and-tumble motion—a sequence of runs separated by tumbles, which are turns around mostly the same location for orientation. By experimental observation, the velocity-jump process is relevantly a Poisson process with a turning rate $\lambda$, and therefore the probability density function $p=p(t,x,\boldsymbol{v})$, for $t>0$, $x\in \Omega$ and\footnote{Here, $V$ is the set of allowed velocities, usually regarded to be symmetric and compact} $\boldsymbol{v}\in V \subset \mathbb{R}^N$, describing the population of agents can be governed using the forward Kolmogorov equation
\begin{align*}
    \partial_t p + \nabla_{ {x}} \cdot (\boldsymbol{v}p) = \lambda \int_{V} \mathcal T (\boldsymbol{v},\boldsymbol{v}') p d\boldsymbol{v}' - \lambda p + \mathcal G, \quad u= \int_V p(t,x,\boldsymbol{v}) d\boldsymbol{v},  
\end{align*}
where the turning kernel $\mathcal T$ is the probability density of turning from velocity $\boldsymbol{v}'$ to velocity $\boldsymbol{v}$, and $\mathcal G$ takes into account the interaction between the bacterial and nutrient cells. The turning rate $\lambda =\lambda(t,x,\boldsymbol{v})$ depends on both the bacterial density and the nutrient concentration, and particularly the nutrient gradient. By  a parabolic scaling with a sufficiently small relaxation parameter $0<\tau\ll1$, the authors of \cite{plaza2019derivation} performed an asymptotic analysis for the parabolic limit as $\tau \to 0$ using the Ansatz $p_\tau=p_0 + \tau p_1 + O(\tau^2)$, where the corresponding perturbed turning rate are chosen such that  $\lambda_0 \propto 1/(u_0v)$ and $\lambda_1 \propto \kappa(v)(v \cdot \nabla v) $, see \cite{mendez2012density,
othmer2002diffusion}, to derive the system    
\begin{align} 
\begin{cases}
 u_t = \nabla \cdot ( uv \nabla u) - \chi \nabla\cdot (u^2 v\nabla v) + \ell uv & \text{in } \Omega \times (0,\infty),\\
 v_t = \Delta v - u v & \text{in } \Omega \times (0,\infty), 
\end{cases} 
\label{OriginModel}\tag{0}
\end{align}
in combination with the nutrient
consumption mechanism, where $\chi>0$. $\ell\ge 0$ and the diffusion coefficients have been renormalised. This system is supplemented with no-flux boundary conditions and regular initial conditions.

\medskip

 In the last few years, the nutrient taxis system \eqref{OriginModel} has gained much attention from the community with a more generic chemotactic flux $ u^\alpha v \nabla v$ for a suitable value of the bacterial response parameter $\alpha$, see \cite{winkler2021does,winkler2024bounds,li2022large,zhang2026boundedness,DeJi2026Stabilization}. Taking into account this generic chemotactic flux,  this work is devoted to studying the system
\begin{align}\label{System:Main}
\begin{cases}
  \displaystyle \pa_t u = \di (u v \nabla u ) - \chi \di ( u^\alpha v \nabla v) + \ell uv, & \text{in } \Omega \times (0,\infty) , \vspace*{0.15cm} \\
  \displaystyle  \pa_t v = \Delta v - uv, & \text{in } \Omega \times (0,\infty) , \vspace*{0.15cm}\\
  \displaystyle  \frac{\partial u}{\pa \nu}=\frac{\partial v}{\pa \nu}=0, & \text{on } \Gamma \times (0,\infty) , \vspace*{0.15cm}\\
u (x,0)=u_{0}(x) ,\quad v (x,0)=v_{0}(x), & \text{on } \Omega,
  \end{cases}
\end{align}
with different effects of the bacterial response, ranging in the interval 
\begin{align}
0\le\alpha < 2,
\end{align}
where $u_0,v_0$ are given nonnegative initial data. System \eqref{System:Main} is called a  
\textit{doubly degenerate} nutrient taxis system since the diffusive flux is doubly nonlinear, and it degenerates at least a small bacterial density or signal concentration. 

\medskip

 The study of the system \eqref{System:Main} primarily focuses on global weak solvability. In \cite{winkler2021does}, the author demonstrated the global existence of a weak solution to \eqref{System:Main} for $\alpha=2$ in a one-dimensional setting, while also obtaining some qualitative behaviour of solutions, where the main idea is to utilise the gradient-like structure
\begin{align*}
 \frac{d}{dt}\left( - \intO \log u + \frac{1}{2} \intO v_x^2  \right) = - \intO \frac{v}{u} u_x^2 - \intO v_{xx}^2 - \intO uv_x^2 - \intO v.
\end{align*}
In \cite{winkler2022small} and later on in \cite{winkler2024bounds}, the global weak solvability of \eqref{System:Main} for $1\le \alpha\le 2$, $N=2$, had been proved by balancing the evolution of the energy-like quantity  
\begin{align}
  \intO u^p + \intO  \frac{|\nabla v|^q}{v^{q-1}} \quad \text{for any } p>2 \text{ and } q=2  
  \label{Struc:1}
\end{align}
based on the functional inequality, for smooth functions $\varphi,\psi$ and $p\ge 1$,
\begin{align}
 \intO \varphi^{p+1} \psi \le C(p) \intO \varphi^{p-1} \psi |\nabla \varphi|^2 + C(p) \left( \intO \varphi \right)^p  \left( \intO \frac{\varphi}{\psi} |\nabla \psi|^2 + \intO \varphi \psi \right) ,
 \label{FuncIneqn:1}
\end{align} 
see \cite[Lemma 3.2]{winkler2022small}, 
where we note that the first term on the right-hand side is a part of the dissipation of the functional $\intO u^p$, and an appropriate smallness assumption on the initial data was imposed in case  $\alpha=2$ to control the taxis-driven contribution. In the same spatial dimensional  setting, i.e. $N=2$, this smallness condition had been removed in\cite{zhang2026boundedness}, where the authors investigated the structure 
\begin{align}
 \frac{d}{dt} \left(  C_1 \intO  u \log u  +  \intO  \frac{|\nabla v|^4}{v^3} \right) + C_2 \intO  u\frac{|\nabla v|^4}{v^3}  \le C  \left( \intO u^2v|\nabla v|^2 + \intO u^2 v + \intO uv \right)
 \label{Struc:2}
\end{align}
for some suitable constants $C_1,C_2 >0$, and therefore, obtained the global solvability by constructing a bootstrap argument based on the $L^{p_k}$-energy with $p_k \to \infty$ as $k\to \infty$. The taxis-driven contribution in \eqref{Struc:2} was  controlled by 
refining the inequality \eqref{FuncIneqn:1} for $p\ge 1$  as 
\begin{align}
 \intO \varphi^{p+1} \psi \le C(p) \left( \intO \frac{\varphi}{\psi} |\nabla \psi|^2 +  \intO \frac{\psi}{\varphi} |\nabla \varphi|^2 + \intO \varphi \psi \right) \intO \varphi^p ,
 \label{FuncIneqn:2}
\end{align} 
and proposing the new functional inequality, for any $p\ge 1$ and $\eta>0$,
\begin{align}
\begin{split}
  \intO \varphi^{p+1} \psi |\nabla \psi|^2  \le \eta \left( \intO \frac{\psi}{\varphi^{1-p}}
|\nabla \varphi|^2 +\intO \varphi \psi \right)  + C  \left( \intO \varphi^{p+1} \psi + \left( \intO \varphi \right)^{2p+1} \right) \intO \frac{|\nabla \psi|^4}{\psi^3},
\end{split}
\label{FuncIneqn:3}
\end{align}
where the constant $C=C(p,\eta)$ contains also the $L^\infty(\Omega\times(0,T))$-norm of $\psi$, see \cite[Lemmas 3.4 and 3.5]{zhang2026boundedness}. 

\medskip

In these works, with $N=1,2$, the Sobolev embeddings and the proposed functional inequalities are highly effective in establishing necessary a priori estimates, which are sufficient to control the taxis-driven contribution and to carry out a bootstrap argument leading to global weak solvability. However, the inequalities \eqref{FuncIneqn:1}, \eqref{FuncIneqn:2}-\eqref{FuncIneqn:3} typically depend on the spatial dimension $N$. In \cite{li2022large} with $N=3$, by considering the same structures \eqref{Struc:1} and \eqref{Struc:2}, and matching the inequalities \eqref{FuncIneqn:1}, \eqref{FuncIneqn:2}-\eqref{FuncIneqn:3} to their three-dimensional version 
\begin{align}
 \|\varphi^{p+1} \psi\|_{L^r(\Omega)} \le \eta \left( \intO  \frac{\varphi^{p+1}}{\psi} |\nabla \psi|^2 + \intO\frac{\psi}{\varphi^{1-p}} |\nabla \varphi|^2  \right) + C(p,r,\eta) \left( \intO \varphi \right)^p \intO \varphi \psi ,
  \label{FuncIneqn:4}
\end{align} 
the author 
obtained the global existence of a weak solution to \eqref{System:Main} for $7/6<\alpha<13/9$. Recently, in the same spatial dimensional setting, the authors of \cite{DeJi2026Stabilization} have extended the result for another interval of the responsible parameter $3/2<\alpha<19/12$.

\medskip

The question of the global weak solvability for the considered problem \eqref{System:Main} has become highly challenging, particularly as $\alpha$ is closer to $2$.

\subsection{Main goal and challenges} The main goal of this paper is to establish the global existence of a weak solution to the problem \eqref{System:Main} in the physical dimension, i.e. $N=3$, and for the wide range of bacterial responses  $0\leq \alpha < 2$. Let us analyse the main challenges of our problem, as presented below. 
\begin{itemize}
 \item \textbf{Quantification of doubly degenerate diffusion}: The first core challenge stems primarily from the doubly nonlinear degenerate diffusion. As straightforward approaches, for instance, in \cite{wang2015global,ishida2012global,lankeit2017locally}, nonlinear diffusion processes with degeneracies that explicitly depend only on the species density have been addressed reasonably and effectively. However, these approaches are not extendible in the context of doubly nonlinear degenerate diffusion, since energy dissipation depends on both the species density and the nutrient concentration, with the possibility of degeneracy arising from the diffusive flux. 
Therefore, quantifying its dissipation is crucial to our analysis.

\item \textbf{Strong taxis-driven contribution}: The second core challenge also stems from the strong chemotactic effect, which is particularly pronounced at high species density when $\alpha$ is close to $2$. To control this effect, as appearing in energy-dissipation inequalities, a direct absorption into the quantified dissipation of the structures \eqref{Struc:1} and \eqref{Struc:2}, in light of the inequality \eqref{FuncIneqn:3}, is certainly inadequate. In other words, finding new energy structures and their possible combinations is needed. Moreover, the utilisation of both the new and the well-known structures is suggested not only from their dissipation but also from the boundedness itself, since this boundedness is in $L^\infty$ in time. 

\item \textbf{Dimensionally dependent limitation}: Finally, many analysis tools from prior results are not applicable because they strictly rely on the dimension, which does not hold in the physical-dimension setting $N=3$.
\end{itemize}

Our new ideas to overcome the above challenges will be explained in Section \ref{Sec:Newideas}. Let us present our main results in the following subsection.

\subsection{Main results} The main result is presented in Theorem \ref{MainTheo}, which is derived from our analysis across different ranges of the chemotactic effects. For this result, we impose the technical assumption \ref{Ass:Initial} on the initial data. Our study does not focus on weakening this assumption, since the question of global solvability of the system \eqref{System:Main} is highly challenging, even under sufficiently strong assumptions on the regularity of the initial data. Moreover, it is emphasised that our results require no smallness condition imposed on the initial data, and from this point onward we set $\chi=1$ for convenience, exactly as stated in Theorem~\ref{MainTheo}. 
 
\begin{assumption}\label{Ass:Initial} 
Assume that the initial datum $0 \not \equiv u_0 \in W^{1,\infty}(\Omega)$ is nonnegative, and satisfies $1/u_0\in L^1(\Omega)$. Furthermore, let $v_0 \in W^{2,\infty}(\Omega)$ be strictly positive on  $\overline{\Omega}$, subjected to the no-flux boundary condition $  \frac{\partial v_0}{\pa \nu}|_\Gamma=0$, and suppose that $ \log (1/v_0) \in W^{2-\frac{2}{q_*},q_*}(\Omega)$ for some $q_*>5/2$. 
\end{assumption}

\begin{theorem}  \label{MainTheo} 
Let $N=3$,  $\ell\ge0$, and $0\leq \alpha<2$.  Under Assumption
\ref{Ass:Initial}, there exists a
global weak solution to the system \eqref{System:Main} in the sense of
Definition \ref{Def:Weak}.  Moreover, for every $T\in (0,\infty)$ and every
$p\in(1,\infty)$ there is a constant $C_{T,p}>0$ such that
\begin{align*}
 \|u\|_{L^\infty(Q_T)} + \|u\|_{L^2((0,T);H^1(\Omega))}  + \|v\|_{L^\infty((0,T);W^{1,\infty}(\Omega))}  + \|v\|_{L^p((0,T);W^{2,p}(\Omega))} \le C_{T,p}.  
\end{align*} 
\end{theorem}

In the following, we introduce the definition of a weak solution as mentioned in Theorem \ref{MainTheo}, which was also considered in the existing results, e.g. see \cite{winkler2024bounds,zhang2026boundedness,li2022large}. 

\begin{definition}\label{Def:Weak}
Suppose that Assumption
\ref{Ass:Initial} holds. A pair $(u,v)$ of nonnegative functions
 \begin{align*}
u\in L_{loc}^1(\overline{\Omega}\times[0,\infty)) \quad \text{and} \quad 
v\in L_{loc}^\infty(\overline{\Omega}\times[0,\infty))\cap L_{loc}^1([0,\infty);W^{1,1}(\Omega)) 
 \end{align*}
 will be called a global weak solution of \eqref{System:Main} if 
 \begin{align*}
  u^2 \in L_{loc}^1([0,\infty);W^{1,1}(\Omega)) \quad\mathrm{and}\quad 
  u^\alpha \nabla v\in L_{loc}^1(\overline{\Omega}\times[0,\infty);\mathbb{R}^3), 
 \end{align*}
 and 
 \begin{gather*} 
  -\iint_{(0,T)\times\Omega} u\varphi_{t}-\int_{\Omega}u_{0}\varphi(0) = \iint_{(0,T)\times\Omega} \left(- \frac{1}{2} v \nabla u^2 \cdot \nabla \varphi + \chi u^\alpha v \nabla v \cdot \nabla \varphi + \ell uv \varphi \right), \\  \iint_{(0,T)\times\Omega} v\varphi_t+\intO v_0\varphi(0)=\iint_{(0,T)\times\Omega} ( \nabla v\cdot\nabla\varphi + u v \varphi )  
 \end{gather*}
 for any $0<T<\infty$ and all $\varphi\in C_0^\infty(\overline{\Omega}\times[0,T))$ such that $\frac{\partial\varphi}{\partial\nu}=0$ on $\Gamma\times(0,T)$.
\end{definition}

{\bf Organisation}: Based on o methodology, we divide our analysis into three separable cases: (i) $0\leq \alpha\le 1$ --
 Weak chemotaxis effect; (ii)  $1<\alpha\le 3/2$ -- Moderate chemotaxis effect; and (iii) 
  $3/2<\alpha<2$ -- Strong chemotaxis effect.  In the next section, we present the regularised system and basic a priori estimates for the regularised solution.  The 
proof of Theorem \ref{MainTheo} will be divided into Sections \ref{Sec:Weak}-\ref{Sec:Stro}, respectively, for the cases (i)--(iii) of the chemotactic effect. 

\medskip

{\bf Notations}: We denote $Q_T:=\Omega \times(0,T)$ for $0<T\le \infty$.  The notations $L^q$, $W^{l,q}$, for $1\le q\le \infty$ and $l\ge 0$, are used for the standard Lebesgue and Sobolev spaces. Throughout the paper, we denote by $C$ a general positive constant that does not depend on $x,t$, the unknowns, and the approximate parameter $\varepsilon$. It can vary from line to line, or sometimes, in the same line. As dependencies are important, for example, the dependence on $T$, we will explicitly write $C_T$, and so on.

\section{New energy structures and key ideas}
\label{Sec:Newideas}

Since we have not introduced the regularised solution, all mathematical expressions in this section will be presented for the formal solution $(u,v)$, which will be made rigorous later on. 

\medskip

\textbf{Observation of new energy structures}: As mentioned in Subsection 1.1, the prior results are based on the evolution of the energy structures  
\begin{align}
 \intO u(t)\log u(t) \quad \text{and} \quad \intO \frac{|\nabla v|^4}{v^3},
 \label{well-knownEnergy}
\end{align}
as well as a combination of them, which was used as an a priori estimate to control the evolution of the $L^p$-energy $\intO u^p$ for $p>1$. Unfortunately, these are not sufficient for our problem due to the challenges analysed earlier; therefore, new energy structures are needed. At this point, we found that new energy structures for the problem can be formed due to
the following important observations.
\begin{itemize}
 \item \textit{Competitive diffusive-chemotactic flux}: The boundedness of solutions is mainly ``decided" by the competition between the signal-dependent diffusion degeneracy and chemotaxis, which can be rewritten as
\begin{align}
 u v \nabla u - u^\alpha v \nabla v 
 = u^ \alpha v \nabla \left( \frac{1}{2-\alpha} u^{2-\alpha} - v \right) .
 \label{Flux:Compe}
\end{align} 
This flux suggests testing the equation for $u$ by $\frac{1}{2-\alpha} u^{2-\alpha} - v$, namely, suggests to consider the energy $\intO u\left( \frac{1}{(2-\alpha)(3-\alpha)} u^{2-\alpha} - v\right)$. In more detail, it shows 
\begin{gather}
\begin{gathered}
  \frac{d}{dt} \intO \left( \frac{1}{(2-\alpha)(3-\alpha)} u^{3-\alpha} - u v \right)  
 + \intO u^\alpha v \left| \nabla \left( \frac{1}{2-\alpha} u^{2-\alpha} - v \right)\right|^2 \\
= \ell \intO \left( \frac{1}{2-\alpha} u^{3-\alpha}v - u v^2 \right) - \intO  u \Delta v + \intO u^2 v ,
\end{gathered}
\label{Struc:3}
 \end{gather}
see Lemma \ref{L:FirstEnergy}. Possible combinations of this energy with the well-known energies \eqref{well-knownEnergy} will give a \textit{control} of the $L^\infty((0,T);L^{3-\alpha}(\Omega))$-norm of $u$, which is significantly better than the $L^\infty((0,T);L^{1}(\Omega))$-norm as $\alpha<2$. 

 \item \textit{The structure of $u^pv^q$ has the necessary dissipation}: It is already found that the consumption mechanism $-uv$ is helpful in estimating $\frac{|\nabla v|^4}{v^3}$, a better type of estimate compared to $|\nabla v|^4$ since it may contains more singularity. Consequently, the integrability of $u^pv^q$ for some $p,q>0$ on $\Omega \times (0,T)$ is easier to obtain than that of $u^p$, see e.g. \cite{winkler2024bounds,zhang2026boundedness,li2022large,DeJi2026Stabilization}.  
 However, aside from these points (in the existing results), we found that the evolution of this product exhibits additional dissipation, that is, a compensation for the well-known energies \eqref{well-knownEnergy}; see Lemma \ref{L:St:upvq}. 

 \item \textit{The $L^p$-energy with $p \in (0,1)$ or even with negative exponent}: As usual expectation, the evolution of $u^p$ has been considered for $p>1$. Since this is one of the critical estimates of the problem, it is difficult to obtain at the very beginning of the analysis. Our approach is to consider this energy for some $p \in (0,1)$ (e.g., Lemmas \ref{L:Basic:MixUV}, \ref{L:Mod:Key}), and in particular for $p<0$ (e.g., Lemmas \ref{L:Mod:Key}, \ref{L:Stro:Key-1}), thereby yielding some dissipation, albeit at a lower level. We then develop estimates for solutions from this lower regularity.

\end{itemize}
The structures observed above are known in the general analysis of PDEs; However, their integration in this paper constitutes a completely novel and uniquely powerful framework. They have not been utilised in the existing results for doubly degenerate nutrient taxis systems. These new structures, as well as possible combinations with the well-known structures \eqref{well-knownEnergy}, provide more comprehensive insight into the problem. In the physical dimension $N=3$, the rigorous conclusion obtained here extends the previously covered interval
\begin{align*}
 1.166...=\frac{7}{6}<\alpha<\frac{13}{9}=1.444... \quad \text{and} \quad \frac{3}{2}<\alpha<\frac{19}{12}=1.583...
\end{align*}
to the range $0\leq \alpha <2$.

\medskip

\textbf{Key ideas and contribution}:  
 For $0\leq \alpha\le1$, the modified competitive energy in
Lemma \ref{L:SublinearEnergy}, the essential boundedness of $v$ on $\Omega \times (0,T)$, for any $0<T\le \Tm$, and the maximal regularity allow us to balance the structure \eqref{Struc:3} when $0\leq \alpha \le 1$, resulting as 
\begin{align}
\sup_{0\le t\le T} \intO u^{3-\alpha}(t) +  \intQT  |\Delta \ve|^{3-\alpha} \le C_T,
\label{Bound:1}
\end{align}
see Lemma \ref{L:Weak:L2Est}.This allows us to perform a short bootstrap in Lemma \ref{L:Weak:BS} using the $L^p$-energy to obtain the boundedness of $u$ in $L^p((0,T)\times\Omega )$ for any $1<p<\infty$. It suffices to show that $v$ is positively lower-bounded if the initial data is well-prepared, thereby reducing the system from the doubly diffusive degeneracy to a degeneracy only in the species density. Hence, the essential boundedness and the H\"older continuity of solutions are ensured, consequently implying the global solvability, see Lemma \ref{L:Weak:Global}.

\medskip

\textit{The case $1<\alpha\le \frac{3}{2}$}: Balancing the structure \eqref{Struc:3} is now difficult since $3-\alpha<2$. In fact, due to the structures \eqref{well-knownEnergy}, it requires controlling the taxis-driven contribution $\iint_{\Omega \times (0,T)} u^{2\alpha-2} v |\nabla v|^2$, which can be achieved as in Lemma \ref{L:Mod:FirstEst} by establishing short yet subtle estimates Lemmas \ref{L:BasicEst:v}-\ref{L:Basic:MixUV} and using the $L^p$-energy estimate for $p\in(0,1)$. Consequently, 
\begin{align}
  \sup_{0\le t\le T} \left( \intO  u(t) \log u(t) +   \frac{|\nabla v(t)|^4}{v^3(t)} \right) +\intQT v|\nabla u|^2 +\intQT  u\frac{|\nabla v|^4}{v^3} \le C_T.
  \label{Diss:1}
\end{align} 
The dissipation obtained in \eqref{Diss:1} will be quantified thanks to the functional inequality \eqref{L:SOBOineqn:State1} in Lemma \ref{L:SOBOineqn} and the use of the $L^p$-energy with a suitable negative exponent, resulting in the integrability of $u^2v$, see Subsection \ref{Sec:Mod:m=2}. Then, to obtain the boundedness of $u$ in $L^\infty(0,T;L^p(\Omega))$ for any $1<p<\infty$, we need to control the corresponding $L^p$-energy with the same exponent, where the dissipation and the taxis-driven contribution read  \begin{align*}
 \intQT  u^{p-1} v |\nabla u|^2 \quad \text{and} \quad \intQT  u^{p-3+2\alpha} v |\nabla v|^2,  
\end{align*} 
respectively. Since the dissipation in \eqref{Diss:1} yields very low regularity,  extending the bacterial response parameter range $7/6<\alpha<13/9$ in \cite{li2022large} to $1<\alpha\le 3/2$ seems highly challenging. Our idea is to perform the so-called \textbf{two-phase bootstrap argument}, starting from that low regularity and progressing to a certain intermediate regularity in the first phase, and then improving the system such that it is sufficient to perform the second phase. 
Indeed, defining the component-wise increasing sequence $(m_k,p_k,r_k)_{k=0,1,\dots}$ by $m_0 := 2$ and
\begin{gather*}
  p_{k}:= \frac{m_{k}}{2} + \frac{7}{2} - 2\alpha, \quad 
  r_{k} : = \min\left\{ \frac{4}{3}p_{k} -2 ;\, p_{k}-1 \right\} , \quad 
m_{k+1} := \frac{2}{3}p_k + r_k +2 ,  
\end{gather*}
we show that, see Subsections \ref{Sec:M:FB}-\ref{Sec:M:2BS},  
\begin{itemize}
 \item[(M1)]  $\displaystyle 
\intQT  u^{m_k} v \leq C_T \quad \text{implies} \quad 
\sup_{0<t<T}\intO u^{p_k}(t) + \intQT  u^{r_k} v|\nabla u|^2 \leq C_{T,k}$; 

\item[(M2)] $ \displaystyle
  \sup_{0<t<T} \intO u^{p_k}(t) + \intQT  u^{r_k} v|\nabla u|^2 \leq C_T \quad \text{implies}$ 
$$ \left\{ \begin{array}{lll} 
  \displaystyle \intQT u^{m_{k+1}} v \leq C \intQT  u^{r_k} v|\nabla u|^2+ C \intQT  u^{r_k+2}  \frac{|\nabla v|^2}{v}, \vspace{0.15cm}\\
  \displaystyle \intQT  u^{r_k+2} \frac{|\nabla v|^2}{v} \leq C_{T,\eta,k} +  \eta \intQT  
 u^{m_{k+1}} v  \hfill \quad \text{for any } \eta>0,& 
  \end{array} \right. $$
\end{itemize}
which yields the integrability of $u^{m_{k+1}} v$. This phase provides a control on the $L^p$-energy for some $p>3$ and, subsequently, the integrability of $u^mv$ for some $m>6$. This certain level of regularity gives the boundedness of $\nabla v$ in $L^\infty(Q_T)$ using the heat semigroup, and allows us to perform the second phase by defining another component-wise increasing sequence 
\begin{align*} 
\widehat m_0> 6, \quad \widehat p_k:= \widehat m_k+3-2\alpha,  \quad 
  \widehat r_k : = \widehat p_k-1 ,  \quad \widehat m_{k+1} := \frac{2}{3}\widehat p_k + \widehat r_k +2,
\end{align*}
resulting in the solution regularity in  $L^\infty(0,T;L^p(\Omega))$. Note that the case $\alpha>3/2$ falls outside the control of this approach since the sequence $(p_k)$ is no longer increasing. 

\medskip 

\textit{The case $\alpha>\frac{3}{2}$}: The taxis-driven contribution is much stronger, particularly as $\alpha$ is closer to $2$, where achieving the estimate \eqref{Diss:1} seems very difficult. This is shown only in    \cite{DeJi2026Stabilization} for $3/2<\alpha<19/12$. In our analysis, we propose  an unusual control of the form
\begin{align*}
 \iint_{\Omega \times (0,T)} \frac{v}{u} |\nabla u|^2 \le C_{T} + \intO u^{-\kappa}  
\end{align*}
for some small constant $0<\kappa<1$ (see Lemma \ref{L:Stro:Key-1}), which in combination with the dissipation of the $L^q$-energy for $1-\alpha < q < \min(-1/3, 4/3-\alpha)$ gives the boundedness of $u^mv$ in $L^1(\Omega \times (0, T))$ for $m=4/3$ and some more useful regularity in Lemma \ref{L:Stro:Key}. This is obviously too weak to bootstrap the system. 
To obtain the estimate \eqref{Diss:1} and improve the solution regularity serving for the global solvability proof, the control of the terms 
\begin{align}
 \iint_{\Omega \times (0,T)} u^{p-1} v^{2} |\nabla u|^2 \quad \text{and} \quad \iint_{\Omega \times (0,T)} u^{p-1+\alpha} v|\nabla v|^2
\label{MixTerms}
\end{align}
for $0<p<1$, but close to $1$, will be the key. We observe that the dissipation of the structure $u^pv^q$ includes exactly the mixed terms in  \eqref{MixTerms}, namely, \textbf{the structure $u^pv^q$ has the needed dissipation}. Indeed,  using the evolution, see Lemma \ref{L:St:upvq},
\begin{align*}
 \begin{split}
  \frac{d}{dt} \intO u^p v^q& = p(1-p)\intO u^{p-1} v^{q+1} |\nabla u|^2 + p q\intO u^{p-1+\alpha} v^{q} |\nabla v|^2 + p \ell \intO u^{p} v^{q+1} \\
  & - p (1-p)\intO u^{p-2+ \alpha} v^{q+1} \nabla u \cdot \nabla v - p q\intO u^{p} v^{q} \nabla v \cdot \nabla u \\
  & - pq \intO u^{p-1} v^{q-1} \nabla v \cdot \nabla u -q (q-1) \intO u^p v^{q-2} |\nabla v|^2 - q \intO u^{p+1} v^q ,
 \end{split} 
\end{align*}
we obtain a possible coupling of the well-known structures \eqref{well-knownEnergy} with the new structures $u^pv^q$ and the $L^p$-energy for $0<p<1$ in Lemma \ref{L:St:Comb}. Then, utilising the functional inequality in  Lemma \ref{L:SOBOineqn}, we derive necessary estimates
for the terms \eqref{MixTerms} in Lemmas \ref{L:St:B:Key0a}-\ref{L:St:B:Key0a2}, which can be quantified to have the integrability of $u^{m}v$ with $m=7/3$ (Lemma \ref{L:St:B:Key-1}), and of $u^{q_0+2}v|\nabla v|^2$ for some $q_0>2\alpha-4$ (Lemma \ref{L:St:B:Key0c}). This regularity is the most important ingredient for our bootstrap argument, for which the needed regularity as $1<\alpha\le3/2$ is only $m=2$.    

 \medskip

At this stage, a natural approach would be to proceed with the bootstrap argument, as in the previous case (i.e., $1<\alpha\le 3/2$). However, because of the strong chemotactic effect with the possibility of $\alpha$ close to $2$, an argument of singularity-enhancing feedback (denoted by (S2) below) is required in our method. From the prior results, after successfully controlling the structures \eqref{well-knownEnergy}, or more precisely, after having the estimate
\begin{align*}
  & {\small  \underbrace{\sup_{0<t<T} \int_{\Omega}  u \log u +  \iint_{Q_T}  v|\nabla u|^2}_{\text{has fully been utilised}}+  \underbrace{\sup_{0<t<T} \int_{\Omega}\frac{|\nabla c|^{4}}{c^{3}}}_{\text{has not been utilised}}+  \underbrace{\iint_{Q_T} \left( \frac{|\nabla c|^{6}}{c^5} + u \frac{|\nabla c|^{4}}{c^3}  \right)}_{\text{has been utilised}} }  \le C_T,
 \end{align*} 
the boundedness of $\sup_{0<t<T} \int_{\Omega}\frac{|\nabla c|^{4}}{c^{3}}$ remains unused and elusive.  
Our idea is to unlock the use of this boundedness, or in other words, to exploit the ``whole" energy-dissipation estimate. Indeed, defining the bootstrap sequence, see Lemmas \ref{L:Stro:Sequence}-\ref{L:SeqProp}, 
\begin{align*}
p_{k}:= q_{k} + 5 - 2\alpha, \quad 
  r_{k} : =  p_{k}-1,    \quad q_{k+1} := \min\left\{2p_k-2;\, \frac{7}{6}p_k-1 \right\}  
\end{align*}
with $q_0 > 2\alpha-4$, we show in Lemmas \ref{L:Feedback3}, \ref{L:Feedback1}, and \ref{L:Feedback2} that 
\begin{itemize}
 \item[(S1)\hspace{0.15cm}] \hspace{-0.15cm} $\displaystyle \intQT  u^{q_k+2} v |\nabla v|^2 \le C_T \quad \text{implies} \quad\sup_{0<t<T}\intO u^{p_k}(t) + \intQT u^{r_k} v  |\nabla u|^2 \le C_{T,k}$; 
 \item[(S2)\hspace{0.15cm}] \hspace{-0.15cm} $\displaystyle  \sup_{0<t<T}\intO u^{p_k}(t) + \intQT u^{r_k} v  |\nabla u|^2 \le C_{T}
  \; \text{ implies} \; \left\{ \begin{array}{llll}
\displaystyle \sum_{i=1}^4 \intQT u^{q_{k+1}-\frac{i}{6}p_k} v^{\frac{6-i}{2}}  |\nabla u|^2 \le C_{T,k}, \vspace{0.15cm} \\
 \displaystyle  \intQT u^{q_{k+1}-\frac{2}{3}p_k+2}  \frac{|\nabla v|^2}{v}    \le C_{T,k};
\end{array} \right. $  
 \item[(S3)\hspace{0.15cm}] \hspace{-0.15cm} $\left\{ \begin{array}{llll}
\displaystyle \sum_{i=1}^4 \intQT u^{q_{k+1}-\frac{i}{6}p_k} v^{\frac{6-i}{2}}  |\nabla u|^2 \le C_T, \vspace{0.15cm} \\
 \displaystyle  \intQT u^{q_{k+1}-\frac{2}{3}p_k+2}  \frac{|\nabla v|^2}{v}    \le C_{T}
\end{array} \right.
 \quad \text{implies} \quad \displaystyle  \intQT  u^{q_{k+1}+2} v |\nabla v|^2 \le C_{T,k}, $
\end{itemize}
where the boundedness of $\sup_{0<t<T} \int_{\Omega}\frac{|\nabla c|^{4}}{c^{3}}$ is crucial to prove  (S2)-(S3).  The term \textit{singularity enhancement} is understood in the sense that the feedback in (S2) includes $v^{-1}|\nabla v|^2$ with the exponent of $v$ decreased by $2$ relative to (S1), yielding that the feedback in (S2) may contain stronger singularity than in (S1) and the likelihood of $v$ being bounded below by a positive constant increases. These feedbacks constitute our bootstrap arguments in Lemma \ref{L:Stro:Boot}.

\medskip

Finally, it is worth remarking that the global weak solvability of \eqref{System:Main} as $\alpha=2$ and $N=3$ has become a highly challenging question, since the proposed methods, not only in this work but also in the related works mentioned earlier, seem inextendible. On the other hand, finding an optimal condition of the form $\alpha  - (m-1) < f(N) $ for having the global solvability of the more generalised system  
\begin{align*} 
\begin{cases}
  \displaystyle \pa_t u = \div (u^{m-1}v \nabla u ) -   \div ( u^\alpha v \nabla v)  , & \text{in } \Omega \times (0,\infty) , \vspace*{0.15cm} \\
  \displaystyle  \pa_t v = \Delta v - uv, & \text{in } \Omega \times (0,\infty) ,  
  \end{cases}
\end{align*}
is also another challenging question, where the quantity $f(N)$ depends on the dimension that \textit{quantifies} how strong doubly degenerate nonlinear diffusion and the chemotactic effect are. From our results in Theorem \ref{MainTheo}, as $m=1$, $f(3)<2$ can be arbitrarily close to $2$. These questions will be in our subsequent investigation. 
 
\section{Regularised system and a priori estimates}

\subsection{Regularised system}  
\label{Sec:Regularise}

In this part, we introduce a regularisation of the system \eqref{System:Main} and recall the local existence of a classical solution. For this purpose, we regularise $u^\alpha$ by $S_{\alpha,\varepsilon}(u^\alpha)$ for a sufficiently small regularised parameter $\eps\in(0,1)$, where the function $S_{\alpha,\varepsilon}$ is defined as   
\begin{align}
 S_{\alpha,\varepsilon}(s):=
\begin{cases}
 s(s+\varepsilon)^{\alpha-1},&0\leq \alpha\le1,\\
 s^\alpha,&1<\alpha<2,
 \end{cases} 
 \label{Def:SublinearSensitivity}
\end{align}
for all $s\in (0,\infty)$. 
Here, we note that $S_{\alpha,\varepsilon}$ is locally Lipschitz. More precisely, the regularised problem of finding $(u_\varepsilon,\ve )$ is given by 
\begin{align}\label{System:RegPro0}
  \begin{cases}
  \displaystyle \pa_t \ue = \di (  \ue \ve \nabla \ue) - \di (  S_{\alpha,\varepsilon}(\ue) \ve \nabla \ve) + \ell \ue  \ve, & \text{in } \Omega\times (0,\infty), \vspace*{0.15cm} \\
  \displaystyle  \pa_t \ve = \Delta \ve - \ue \ve, &\text{in } \Omega\times (0,\infty), \vspace*{0.15cm}\\
  \displaystyle  \frac{\partial \ue}{\pa \nu}=\frac{\partial \ve}{\pa \nu}=0, & \text{in } \Gamma \times (0,\infty), \vspace*{0.15cm}\\
  (\ue(0), \ve(0) )= (u_{0}+\eps,v_{0}), & \text{in } \Omega.
  \end{cases}
 \end{align} 
By the same argument as \cite{winkler2022small,li2022large}, we directly obtain the local existence of the classical solution to the system, presented in the following lemma.

 \begin{lemma} 
 \label{L:LocalExis}
 Assume that $0\leq \alpha\le 2$ and the initial data is given by Theorem \ref{MainTheo}. For each $\varepsilon \in (0, 1)$, the regularised problem \eqref{System:RegPro0} has a unique classical solution $(\ue,\ve)$ up to the interval $[0,\Tme)$, $0<\Tme \le \infty$, in the sense that
\begin{align}
 \begin{gathered}
u_\eps\in C^0(\overline{\Omega}\times[0,T_{max,\eps}))\cap C^{2,1}(\overline{\Omega}\times(0,T_{max,\eps})),\\
v_\eps\in\cap_{q>2}C^0([0,T_{max,\eps});W^{1,q}(\Omega))\cap C^{2,1}(\overline{\Omega}\times(0,T_{max,\eps})),  
\end{gathered}
\label{L:LocalExis:S0}
\end{align} 
and $\ue,\ve$ are strictly positive in $\overline{\Omega}\times[0,\infty)$. Moreover,
\begin{align}
\text{If } T_{max,\eps}<\infty \quad \text{then} \quad \lim_{t\nearrow T_{max,\eps}}\| u_{\varepsilon}(\cdot,t)\|_{L^{\infty}(\Omega)}=\infty.
\label{L:LocalExis:State1}
\end{align} 
Additionally, the solution enjoys the following $\varepsilon$-independent regularity 
\begin{align}
\| v_\eps(\cdot,t)\|_{L^\infty(\Omega)}\leq\|v_0\|_{L^\infty(\Omega)}\quad \text{for all } t\in(0,T_{max,\eps}),
\label{L:LocalExis:State2}
\end{align} 
  and 
\begin{align}
\intQTm \ue v_\eps\leq\intO  v_0 .
\label{L:LocalExis:State3}
\end{align} 
 \end{lemma}

\subsection{Technical lemmas}

We present here some technical lemmas. First, in Lemma \ref{L:SOBOineqn}, an estimate for the product $\phi^{p+1}\psi$ via mixed terms including the gradients of $\phi$ and $\psi$. In Lemma \ref{L:Mod:Phi-NablaPhi}, the gradient and the Laplacian of the logarithm of a function will be presented. 
 
 \begin{lemma} \label{L:SOBOineqn}
Assume generally that $\Omega\subset \mathbb{R}^N$, $N\ge 1$, is a bounded domain with sufficiently smooth boundary. Then, it holds 
 \begin{align}
\|\phi^{p+1}\psi \|_{L^{\frac{N}{N-2}}(\Omega)} &\leq C_p \intO \phi^{p-1} \psi|\nabla \phi|^2+ C_p \intO \phi^{p+1} \psi^{-1}|\nabla \psi|^2 + C_p \left( \intO \phi \psi^{\frac{1}{p+1}} \right)^{p+1},
  \label{L:SOBOineqn:State1}
 \end{align}
for any $p>-1$, all $\phi,\psi$ in $C^1(\overline{\Omega})$ such that $\phi>0,\psi>0$ in $\overline{\Omega}$. Here, the fraction $\frac{N}{N-2}$ is set to be infinity if $N=1$ and be arbitrarily less than infinity if $N=2$. 
\end{lemma}

\begin{proof} The Sobolev embedding $H^1(\Omega)\subset L^{\frac{2N}{N-2}}(\Omega)$  yields that 
 \begin{align*}
  \| \phi^{\frac{p+1}{2}} \psi^{\frac{1}{2}} \|^2_{\LO{\frac{2N}{N-2}}} & \leq C \| \nabla (\phi^{\frac{p+1}{2}} \psi^{\frac{1}{2}}) \|^{2 }_{\LO{2}} + C \left( \intO \phi \psi^{\frac{1}{p+1}} \right)^{p+1}, 
 \end{align*}
where we have used the interpolation as follows
\begin{align*}
C \| \phi^{\frac{p+1}{2}} \psi^{\frac{1}{2}} \|^{2 }_{\LO{2}} & = C \| \phi \psi^{\frac{1}{p+1}} \|^{p+1 }_{\LO{p+1}} \le C  \left( \|\phi \psi^{\frac{1}{p+1}}\|_{L^{\frac{N(p+1)}{N-2}}(\Omega)}^{\frac{Np}{Np+2}} \|\phi \psi^{\frac{1}{p+1}}\|_{L^{1}(\Omega)}^{\frac{2}{Np+2}} \right)^{p+1} \\
& \le \frac{1}{2} \| \phi^{\frac{p+1}{2}} \psi^{\frac{1}{2}} \|^2_{\LO{\frac{2N}{N-2}}} +  C_p \left( \intO \phi \psi^{\frac{1}{p+1}} \right)^{p+1}. 
 \end{align*} 
 The inequality \eqref{L:SOBOineqn:State1} is then obtained by directly expanding the gradient term.  
\end{proof}

\begin{lemma}[{\cite[Lemma 3.4]{winkler2022approaching}}]
\label{L:Mod:Phi-NablaPhi} 
Assume generally that $\Omega\subset \mathbb{R}^N$, $N\ge 1$, is a bounded domain with sufficiently smooth boundary. 
 If $\varphi\in C^2(\overline{\Omega})$ is such that $\varphi>0$ in $\overline{\Omega}$ and $\frac{\partial\varphi}{\partial\nu}=0$ on $\Gamma$, then 
 \begin{align*}
  \intO\varphi^{-q-1}|\nabla\varphi|^{q+2}\leq (q+\sqrt{N})^{2}\intO\varphi^{-q+3}|\nabla\varphi|^{q-2}|D^{2}\log \varphi|^{2} ,
 \end{align*}
 and
 \begin{align*}
  \intO\varphi^{-q+1}|\nabla\varphi|^{q-2}|D^{2}\varphi|^{2} \leq (q+\sqrt{N}+1)^{2}\intO\varphi^{-q+3}|\nabla\varphi|^{q-2}|D^{2}\log \varphi|^{2}.
 \end{align*} 
\end{lemma}


\subsection{A priori estimates}
\label{Sec:AprioriEst}
 
To obtain the global existence of Problem \eqref{System:Main}, we need to establish the compactness of the regularised solution, where $\varepsilon$-independent estimates for the $\varepsilon$-dependent solution are necessary. We present such basic estimates in this part. Let us first find a bound for the total mass. By integrating the equations for $\ue ,\ve $ and their boundary conditions over the domain $\Omega$, 

\begin{align*}
\intO \ue (t) & = \intO (u_0 + \eps) + \intO \nabla \cdot \bigl(\ue \ve \nabla \ue - S_{\alpha,\varepsilon}(\ue) \ve \nabla \ve \bigr) + \ell \intQt \ue \ve \\
& = \intO (u_0 + \eps) + \ell \intQt (\Delta \ve - \partial_t \ve )   = \intO (u_0 + \eps) + \ell \int_\Omega (v_0 - \ve(t) ) ,
\end{align*} 
where we recall from the regularised problem that the initial data of $u_0^\varepsilon$ is $u_0+\varepsilon$. 
Consequently, we get the mass conservation of the form 
\begin{align} 
\intO ( \ue (t) + \ell \ve (t) ) \le M \quad \text{with} \quad M:= \intO ( (u_0+1) + \ell v_0 ),
\label{Note:TotalMass}
\end{align}
 for all $t \ge 0$. Next, we present some fundamental computations based on the equation for $\ve $. 

\begin{lemma} 
\label{L:BasicEst:v}
For all $t>0$, 
  \begin{align}
\frac{1}{2} \frac{d}{dt} \intO |\nabla \ve|^2 + \intO |\Delta \ve|^2 + \intO \ue|\nabla \ve|^2
= - \intO \ve \nabla \ue \cdot \nabla \ve .
\label{L:BasicEst:v:State1}
 \end{align}
 On the other hand,
 \begin{equation} 
  q \intQT \ue \ve^q + \intQT \ve^{q-2} \left|\nabla \ve \right|^2 \leq \intO v_0^q, 
  \label{L:BasicEst:v:State2}
 \end{equation} 
 for all $1<q<\infty$.
\end{lemma}
 \begin{proof} The equality can be obtained directly by testing the equation for $\ve $ by $-\Delta \ve $. By testing this equation again by $\ve^{q-1}$, we see that
\begin{align*}
  \frac{1}{q} \frac{d}{d t} \intO \ve^q=-(q-1) \intO \ve^{q-2} \left|\nabla \ve\right|^2-\intO \ue \ve^q ,
\end{align*}
which shows the inequality after integrating over time.
  \end{proof}

We will see in the following sections that estimating the product $\ue^{r} \ve|\nabla \ue|^2$ in $L^1(Q_T)$, for $r<0$, is a crucial step in the analysis, which will be established in the lemma below.

\begin{lemma} \label{L:Basic:MixUV}
 Assume $1<\alpha<2$.  
It holds that  
 \begin{align}
  \intQT |\Delta \ve|^2 + \intQT \ue^{1-\alpha} \ve|\nabla \ue|^2 + \intQT \ue|\nabla \ve|^2 \leq C_T,
  \label{L:Basic:MixUV:State1}
 \end{align}
 and consequently,
 \begin{align}
 \intQT \ve \nabla \ue \cdot \nabla \ve +  \intQT \ue^{r} \ve|\nabla \ue|^2 \leq C_{T,q},
 \label{L:Basic:MixUV:State2}
 \end{align}
 for $$\max(-1;2-2\alpha) < r < \min(0;3-2\alpha).$$ 
\end{lemma}

\begin{proof} The key idea is to choose a suitable choice of $0<p<1$ in the $L^p$-energy estimate to eliminate the mixed gradient term in the equality \eqref{L:BasicEst:v:State1}. Indeed, by choosing $p=2-\alpha$, direct computations show that 
 \begin{align*}
  & - \frac{1}{(2-\alpha)(\alpha-1)} \frac{d}{dt} \intO \ue^{2-\alpha} + \intO \ue^{1-\alpha} \ve|\nabla \ue|^2 = \intO \ve \nabla \ue \cdot \nabla \ve - \frac{\ell}{\alpha-1} \intO \ue^{2-\alpha} \ve . 
 \end{align*}
On the other hand, we rewrite the equality \eqref{L:BasicEst:v:State1} 
\begin{align}
\frac{1}{2} \frac{d}{dt} \intO |\nabla \ve|^2 + \intO |\Delta \ve|^2 + \intO \ue|\nabla \ve|^2 = - \intO \ve \nabla \ue \cdot \nabla \ve .
\label{L:Basic:MixUV:Proof0}
 \end{align}
Therefore, a combination of the above equalities gives 
 \begin{align*}
  & \frac{d}{dt} \intO \left(\frac{1}{2}|\nabla \ve|^2 - \frac{1}{(2-\alpha)(\alpha-1)} \ue^{2-\alpha} \right) + \intO \ue^{1-\alpha} \ve|\nabla \ue|^2 \\
  & \hspace*{3cm} + \intO |\Delta \ve|^2 + \intO \ue|\nabla \ve|^2 + \frac{\ell}{\alpha-1} \intO \ue^{2-\alpha} \ve = 0 .
 \end{align*}
 Integrating over time, we get 
 \begin{align*}
  \intQT \Big(|\Delta \ve|^2 + \ue^{1-\alpha} \ve|\nabla \ue|^2 + \ue|\nabla \ve|^2 + \ue^{2-\alpha} \ve\Big) \leq C_{\alpha} \left( \intO \ue^{2-\alpha} + \intO |\nabla v_{0}|^2 \right) .
 \end{align*}
We obtain the estimate \eqref{L:Basic:MixUV:State1} by noting that the last right-hand side is finite since $2-\alpha <1$. 
Let us show \eqref{L:Basic:MixUV:State2}. 
Thanks to the above estimate, we can deal with the product $\ve \nabla \ue \cdot \nabla \ve $ as 
\begin{align*}
\intQT \ve \nabla \ue \cdot \nabla \ve & \le \intQT \ue^{1-\alpha} \ve|\nabla \ue|^2 + \intQT \ue^{\alpha -1} \ve|\nabla \ve|^2 \\
& \le \intQT \ue^{1-\alpha} \ve|\nabla \ue|^2 + \|v_0\|_{L^\infty(\Omega)} \intQT (1+ \ue) |\nabla \ve|^2 \le C_\alpha . 
\end{align*}
Let us consider $0<q<1$. Using the computation 
 \begin{align}
  \intO \ue^{q-1} \ve|\nabla \ue|^2 &= \frac{1}{q(1-q)} \frac{d}{dt} \intO \ue^q + \intO 
  \ue^{q-2+\alpha} \ve \nabla \ue \cdot \nabla \ve - \frac{\ell}{1-q} \intO \ue^q \ve , 
  \label{L:Basic:MixUV:P1}
 \end{align}
we have 
  \begin{align}
  \intQT \ue^{q-1} \ve|\nabla \ue|^2 &\le \frac{2} {q(1- q)} \intO \ue^q + \intQT 
  \ue^{q-3+2\alpha} \ve|\nabla \ve|^2 .
  \label{L:Basic:MixUV:P2}
 \end{align}
It follows from \eqref{L:BasicEst:v:State2} that $\nabla \ve $ is uniformly bounded in $L^2(Q_T)$. Moreover, we have the same conclusion for $\intQT \ue|\nabla \ve|^2$ by noting \eqref{L:Basic:MixUV:State1}. Therefore, the latter right-hand side is finite if both conditions $0< q < 1$ (for the boundedness of $\intO \ue^q$) and $0\le q-3+2\alpha \le 1$, which correspond to 
\begin{align*}
 \max(0;3-2\alpha) < q < \min(1;4-2\alpha),
\end{align*}
i.e. the inequality \eqref{L:Basic:MixUV:State2} is proved.
\end{proof}

\section{Weak chemotactic effect}
\label{Sec:Weak}

In this section, we will establish the global existence of a classical solution to the problem, where the cross-diffusion effect is weak in the sense that $$0\leq \alpha \le 1.$$ 
 From now on, we always take $T$ such that $0<T < T_{\max,\eps}$, where $T_{\max,\eps}$ is the maximal time that the solution $(\ue,\ve)$ exists up to, see Lemma \ref{L:LocalExis}.

\subsection{Structure of the competitive diffusive-chemotactic flux}

To find what exactly the competitive diffusive-chemotactic flux introduced at \eqref{Flux:Compe} is for the regularised solution, we note that the case $0\le \alpha \le 1$ will be calculated differently due to the regularisation  \eqref{Def:SublinearSensitivity}.  
For any $s\ge0$, let us define
\begin{align}
 F_\varepsilon(s) := \int_0^s \frac{r}{S_{\alpha,\eps}(r)}dr
 =\frac{(s+\varepsilon)^{2-\alpha}-\varepsilon^{2-\alpha}}
 {2-\alpha}, 
 \label{Def:SublinearEnergy}
\end{align}
and 
\begin{align*}
 H_\varepsilon(s) :=\int_0^sF_\varepsilon(r)\,dr 
 =\frac{(s+\varepsilon)^{3-\alpha}-\varepsilon^{3-\alpha}}
 {(2-\alpha)(3-\alpha)}
 -\frac{\varepsilon^{2-\alpha}s}{2-\alpha},
\end{align*}
and derive their properties in the following lemma.

\begin{lemma} 
\label{L:SublinearApproximation}
For $0\leq \alpha\le1$, $0<\varepsilon<1$, and $s>0$,
\begin{align}
 0\leq  S_{\alpha,\eps}(s)\le s^\alpha, \quad \text{and}
 \quad
 0\leq  s^\alpha-S_{\alpha,\eps}(s)\le\eps^\alpha.
 \label{L:SublinearApproximation:S1}
\end{align}
Moreover, there are constants $c_\alpha,C_\alpha>0$, independently of $\eps$, such that
\begin{align}
 c_\alpha s^{3-\alpha}\le H_\eps(s)
 \le C_\alpha(1+s^{3-\alpha}),\qquad
 0< F_\eps(s)\le C_\alpha(1+s^{2-\alpha}).
\label{L:SublinearApproximation:S3}
\end{align}
\end{lemma}
\begin{proof} Since $(s+\varepsilon)^{\alpha-1}\le s^{\alpha-1}$, the first property of \eqref{L:SublinearApproximation:S1} is obvious. For the second one,   $0\le s^\alpha-S_{\alpha,\varepsilon}(s)\le s^\alpha\le\varepsilon^\alpha$ if $0\le s\le\varepsilon$ due to  the nonnegativity of $S_{\alpha,\varepsilon}$.  If $s>\varepsilon$, then
\begin{align*}
 s^\alpha-S_{\alpha,\eps}(s)
 =s^\alpha\left[1-\left(1+\frac\eps s\right)^{\alpha-1}\right]
 \le\eps s^{\alpha-1}\le\eps^\alpha.
\end{align*}
by applying the fundamental inequality $1-(1+x)^{-(1-\alpha)}\le(1-\alpha)x$ for $x>0$. 
Hence, \eqref{L:SublinearApproximation:S1} is proved. To show \eqref{L:SublinearApproximation:S3}, we can see that $F_\varepsilon(s)\ge s^{2-\alpha}/(2-\alpha)$ and then use the definition of $H_\varepsilon$ to obtain the lower estimate for $H_\varepsilon$ directly. The upper estimate for $H_\varepsilon$, and similarly for $F_\varepsilon$, can be shown easily.  
\end{proof}
 
Let us combine the diffusive and chemotactic fluxes given in the regularised system \eqref{System:RegPro0}, where, thanks to the definition of $H_\varepsilon$ and $F_\varepsilon$,  
\begin{align}
 \ue \ve\nabla \ue
 -S_{\alpha,\eps}(\ue)\ve\nabla \ve
 =S_{\alpha,\eps}(\ue)\ve
 \nabla\bigl(F_\eps(\ue)-\ve\bigr).
 \label{SublinearFluxIdentity}
\end{align}
This suggests testing the equation for $\ue$ against $F_\eps(\ue)-\ve$ to form an energy structure for the system, which yields that $\ue $ can be estimated in  $L^\infty(0,T;L^{3-\alpha}(\Omega))$ via the integrals of $\ue^ 2 \ve $ and $\nabla \ue \cdot \nabla \ve $, as in the following lemma.
\begin{lemma} 
\label{L:SublinearEnergy}
For $0\leq \alpha\le1$,
\begin{align}
\begin{split}
 &\frac{d}{dt}\intO\bigl(H_\eps(\ue)
              -\ue \ve\bigr)
 +\intO S_{\alpha,\eps}(\ue)\ve
\left|\nabla\bigl(F_\eps(\ue)-\ve\bigr)\right|^2\\
 &\quad=\ell\intO \ue \ve
       \bigl(F_\eps(\ue)-\ve\bigr)
 +\intO\nabla \ue\cdot\nabla \ve
 +\intO \ue^2\ve.
\end{split}
\label{L:SublinearEnergy:S}
\end{align}
\end{lemma}
\begin{proof}
As suggested earlier, we test the first equation of \eqref{System:RegPro0} by $F_\eps(\ue)-\ve$  and use the identity \eqref{SublinearFluxIdentity}, which gives 
\begin{align*}
 \frac d{dt}\intO(H_\eps(\ue)-\ue \ve) = \intO  (F_\eps(\ue)-\ve) \partial_t\ue
  - \intO \ue\partial_t\ve.
\end{align*}
Substituting $\partial_t\ve=\Delta \ve-\ue \ve$ and integrating directly shows the desired identity.
\end{proof}

\begin{remark}  Let us give short comments on \eqref{L:SublinearEnergy:S}. Although we consider $0\leq \alpha\le1$ in this section, Lemma \ref{L:SublinearEnergy} holds up to $0\leq \alpha<2$. For $0\leq \alpha\le 1$, due to the regularity \eqref{L:LocalExis:State2}-\eqref{L:LocalExis:State3}, we see that balancing the energy \eqref{L:SublinearEnergy:S} turns to estimating the terms including $\nabla \ue \cdot \nabla \ve$, see Lemma \ref{L:Weak:L2Est}.  For $1<\alpha<2$, one can directly check that 
\begin{align*}
\frac{d}{dt} \intO \left( \frac{1}{(2-\alpha)(3-\alpha)} \ue^{3-\alpha} - \ue \ve \right) \le \frac{\ell}{2-\alpha} \intO \ue^{3-\alpha} \ve + \intO \nabla \ue \cdot \nabla \ve + \intO \ue^ 2 \ve ,
\end{align*}
which can be combined with an $L^p$-energy estimate to eliminate the cross term including $\nabla \ue \cdot \nabla \ve $ by choosing a suitable $p$. 
\end{remark}

Based on the maximal regularity and the smoothing effect of the Neumann heat semigroup,  the energy structure in the previous lemma allows us to obtain the necessary uniform-in-$\eps$ regularity of the regularised solution. 
 
\begin{lemma}\label{L:Weak:L2Est}
Let $0\leq \alpha\leq 1$.
There exists $C_T>0$ such that 
	\begin{align*}
	\|u_\eps\|_{L^\infty(0,T;L^{3-\alpha}(\Omega))} + \|\nabla \ve \|_{L^\infty(0,T;L^p(\Omega))} + \|\Delta \ve \|_{L^{3-\alpha}(Q_T)} \le C_T,
	\end{align*}
for $1 \le p \le p_0$ with 
\begin{align}
p_0 \left\{ \begin{array}{llcl}
< \frac{3(3-\alpha)}{\alpha} & \text{if}& 0\leq \alpha \le 1, \\
< \infty & \text{if}& \alpha=0.
\end{array} \right.
\label{Def:p0}
\end{align} 
\end{lemma}

\begin{proof} 
 First, testing the equation of $\ve $ by $-\Delta \ve $ can give us an estimate for the $L^2(Q_T)$-norm of $\Delta \ve $ in term of the $L^2(Q_T)$-norm of $\ue $. More precisely, since $\ve\le v_0$ as \eqref{L:LocalExis:State2}, 
\begin{align}
\intQT |\Delta \ve|^2 \le C\left( \intO |\nabla v_0|^2 + \intQT \ue^ 2 \ve^2 \right) \le C\left( 1 + \intQT \ue^ 2 \right), 
\label{Le:Weak:L2Est:P1}
\end{align} 
where the general constant $C$ includes the $L^\infty(\Omega)$-norm of $v_0$. By Lemma \ref{L:SublinearApproximation}, and the upper bound for $\ve$ as well as the mass boundedness, we use the Young inequality to have that 
\begin{align}
 \intO\bigl(H_\eps(\ue(t))
 -\ue(t) \ve(t)\bigr)
 \ge c\intO \ue^{3-\alpha}(t)-C,
 \quad
 \ue F_\eps(\ue)
 \le C(\ue+\ue^{3-\alpha}).
 \label{L:Weak:L2Est:Co}
\end{align}
Note that the initial value of the energy in \eqref{L:SublinearEnergy:S} is uniformly bounded. Therefore, after integrating the equality \eqref{L:SublinearEnergy:S}, and using the fact that 
\begin{align*}
 \int_{Q_t}\nabla \ue\cdot\nabla \ve
 =-\int_{Q_t}\ue\Delta \ve
 \le\eta\int_{Q_t}|\Delta \ve|^2
   +C_\eta\int_{Q_t}\ue^2,
\end{align*}
we get
\begin{align*}
 \intO \ue^{3-\alpha}(t)
 \le C_T+C_T\int_0^t\left(1+\intO
 \ue^{3-\alpha}(s)\right)ds.
\end{align*}
Consequently, the Gr\"onwall inequality implies the boundedness of $\ue $ in $L^\infty(0,T;L^{3-\alpha}(\Omega))$. 

\medskip

The boundedness of $\nabla \ve $ in $L^\infty(0,T;L^p(\Omega))$, for $1\le p<3(3-\alpha)/\alpha$ if $\alpha\not =0$ and for arbitrarily $1\le p<\infty$ if $\alpha=0$, comes from the smoothing effect of the heat semigroup as  
\begin{align*}
\|\nabla \ve \|_{L^p(\Omega)} &= \left\|e^{t\Delta} \nabla v_0 \right\|_{L^p(\Omega)} + \left\|\int_0^t \nabla (e^{(t-s)\Delta}  u_\eps (s) \ve  (s)) ds \right\|_{L^p(\Omega)} \\
&\le \| \nabla v_0\|_{L^p(\Omega)} + \left( \int_0^t (t-s)^{-\frac{3}{2}(\frac{1}{3-\alpha}-\frac{1}{p})-\frac{1}{2}} ds \right) \|u_\eps\|_{L^\infty(0,T;L^{3-\alpha}(\Omega))},  
\end{align*}
and so $\nabla \ve  \in L^\infty(0,T;L^p(\Omega))$. 
Moreover, the boundedness of $\Delta \ve $ in $L^{3-\alpha}(Q_T)$ is concluded via the maximal regularity applying to the equation of $\ve $.
\end{proof}

\subsection{The uniform boundedness of $\ue$ and its H\"older continuity}
\label{Sec:W:Holder}
 
The uniform-in-$\eps$ regularity presented in the last part allows us to obtain the global existence of a unique classical solution to the regularised system \eqref{System:RegPro0}, i.e., $T_{max,\eps}=\infty$. This will based on the criteria \eqref{L:LocalExis:State1} and the $L^p$-energy functional.

\begin{lemma}[Feedback argument] 
\label{L:Weak:Feedback} Let $0\leq \alpha\leq 1$ and $r\ge 2$. 
If 
\begin{align*}
 \sup_{0<t<T} \intO \ue^{r}(t)  \le C,
\end{align*}
then 
 \begin{align}
\sup_{0<t<T} \intO u_{\eps}^{r+\frac{1}{4}}(t) + \intQT u_\eps^{r-\frac{3}{4}}\ve  |\nabla u_\eps|^2  \le C_{T,r} .
\label{Lem:Weak:Feedback:State1}
\end{align}
\end{lemma}

\begin{proof}  
Testing by $p u_\eps^{p-1}$ and using $S_{\alpha,\eps}(s)\le s^\alpha$ gives
\begin{align}
\begin{aligned}
 \frac{d}{dt} \intO \ue^p(t)  
& \le - \frac{p(p-1)}{2} \intO \ue^{p-1} \ve|\nabla \ue|^2 + \frac{p(p-1)}{2} \intO \ue^{p-3+2\alpha} \ve|\nabla \ve|^2 + \ell p \intO \ue^p \ve , 
\end{aligned} 
\label{Lem:Weak:Feedback:P1}
\end{align}
which holds for all $p>1$. 
Under the assumption $\ue \in L^\infty(0,T;L^{r}(\Omega))$ (that is uniform in $\eps$), due to the smoothing effect of the heat semigroup, we see that  
\begin{align*}
\|\nabla \ve (t)\|_{L^s(\Omega)}  
&\le \| \nabla v_0\|_{L^s(\Omega)} + \left( \int_0^t (t-s)^{-\frac{3}{2}(\frac{1}{r}-\frac{1}{s})-\frac{1}{2}} ds \right) \|u_\eps\|_{L^\infty(0,T;L^{r}(\Omega))} ,  
\end{align*} 
for all $0<t<T$, where $1\le s <\frac{3r}{3-r}$ if $r<3$, $s$ is arbitrarily large if $r=3$, and $s=\infty$ if $r>3$. 
By the H\"older inequality, the term including $\nabla \ve$ in  \eqref{Lem:Weak:Feedback:P1} can be estimated as
\begin{align*}
\intO u_{\eps}^{p-3+2\alpha} v_{\eps} |\nabla v_{\eps}|^2 \le \|v_0\|_{L^\infty(\Omega)}  \| u_{\eps} \|_{L^\infty(0,T;L^{\frac{(p-3+2\alpha)s}{s-2}}(\Omega))} ^{p-3+2\alpha} 
  \| \nabla v_{\eps} \|_{L^\infty(0,T;L^{s}(\Omega))}^2. 
\end{align*}	
Under the assumption $u\in L^\infty(0,T;L^{r}(\Omega))$, the latter right hand side is finite if  
\begin{align*}
\frac{(p-3+2\alpha)s}{s-2} \le r,
\end{align*} 
which is equivalent to
\begin{align*}
p\le 3-2\alpha + \left(1-\frac{2}{s}\right)r. 
\end{align*}
According to the constrain $1/r-1/s<1/3$, we can choose $s=(3r/(3-r))^-$ (i.e., a number that is strictly less than $3r/(3-r)$ but sufficiently close to $3r/(3-r)$), i.e., the choice 
\begin{align}
p= \left( 3-2\alpha + \left(1-\frac{2(3-r)}{3r}\right) r \right)^{-} = \left( r + \left( \frac{2}{3}r-2\alpha+1 \right) \right)^{-} 
\label{Lem:Weak:Feedback:Proof1}
\end{align}
is possible. 
Since $r\ge 2$ and $\alpha\le 1$, we have $2r/3-2\alpha+1\ge 1/3>1/4$, this choice of $p$ satisfies that $p> r+1/4$. Subsequently, we get  
\begin{align*}
\intO u_{\eps}^p + \frac{p(p-1)}{2}  \intQt u_{\eps}^{p-1} v_{\eps} |\nabla u_{\eps}|^2 \le \intO u_0^p + C_pt + \ell p \|v_0\|_{L^\infty(\Omega)} \intQt u_{\eps}^p
\end{align*}  
under the choice \eqref{Lem:Weak:Feedback:Proof1}. We obtain the boundedness of $u_\eps$ in $L^\infty(0,T;L^{r+1/4}(\Omega))$ by using the Gr\"onwall inequality and then, by choosing $p=r+1/4$, the estimate for $u_\eps^{r-3/4}\ve |\nabla u_\eps|^2$ in $L^1(Q_T)$. The estimate  \eqref{Lem:Weak:Feedback:State1} is proved. 
\end{proof}

\begin{lemma} 
\label{L:Weak:BS} Let $0\leq \alpha\leq 1$. 
For any $1< p<\infty$, it holds that 
\begin{align}
\sup_{0<t<T} \intO \ue^{p}(t) + \intQT \ue^{p-1}\ve|\nabla \ue|^2 \le C_{T,p} . 
\label{L:Weak:BS:State1}
\end{align} 
\end{lemma}

\begin{proof}
The proof will be done by performing a bootstrap argument. 
Since $u_\eps$ is uniformly-in-$\eps$ bounded in $L^\infty(0,T;L^{p_1}(\Omega))$ with $p_1:=3-\alpha$, we can apply Lemma \ref{L:Weak:Feedback} that 
\begin{align*}
\|u_{\eps}\|_{L^\infty(0,T;L^{p_1+\frac{1}{4}}(\Omega))} \le C_{T,p_1}.
\end{align*}
Let $p_{n+1}=p_n+\frac{1}{4}$, for $n\in \mathbb{N}$, $n\ge 1$. Then, an iteration of the above argument gives 
\begin{align*}
\|u_{\eps}\|_{L^\infty(0,T;L^{p_n+\frac{1}{4}}(\Omega))} \le C_{T,p_n},
\end{align*}
which shows \eqref{L:Weak:BS:State1} by letting $n$ sufficiently large such that $n\ge 4p + 4\alpha-12$, i.e., $p_{n+1} \ge p$. 
\end{proof}

\begin{lemma}\label{L:Weak:Global} Let $0\leq \alpha\leq 1$.
For each $\eps>0$, the classical solution $(\ue,\ve)$ to the regularised problem \eqref{System:RegPro0}, obtained by Lemma \ref{L:LocalExis}, exists globally in time, i.e., $T_{\max,\eps}=\infty$ such that 
\begin{align}
 \|v_\eps\|_{L^\infty(Q_T)} \ge C_T,
\label{L:Weak:Global:S2}
\end{align}
where $q_*>5/2$ given in Theorem \ref{MainTheo}, and consequently,  
\begin{align}
 \|\ue\|_{\LQ{\infty}} \leq C_T .
 \label{L:Weak:Global:S1}
\end{align} 
\end{lemma}

\begin{proof}
  Let $w_{\eps} = \log \frac{\|v_0\|_{\LO{\infty}}}{v_\eps}$, we note that $w_{\eps}$ is non-negative by \eqref{L:LocalExis:State2}, and satisfies 
  \begin{align}
 \begin{split}
\partial_t w_{\eps} = \Delta w_{\eps} - |\nabla w_{\eps}|^2 + u_\eps \le \Delta w_{\eps} + u_\eps.
 \end{split}
  \end{align} 
By applying the heat regularisation, see e.g. \cite[Lemma B.1]{reisch2024global}, and the comparison principle, it follows from $q_*>5/2=(N+2)/2$ that  
\begin{align*}
 \left\| \log \frac{\|v_0\|_{\LO{\infty}}}{v_\eps}  \right\|_{L^\infty(Q_T)} & \le C \left( \left\| \log \frac{\|v_0\|_{\LO{\infty}}}{v_0(x)} \right\|_{W^{2-\frac{2}{q_*},q_*}(\Omega)} + \|u_\eps\|_{L^{q_*}(Q_T)} \right)  \le C_{T} ,
\end{align*}
where the norm of $\ue$ in $L^{q_*}(Q_T)$ is uniformly bounded due to the estimate \eqref{L:Weak:BS:State1}.  
Then, for a.e. $(x,t)\in Q_T$, we obtain
\begin{align*}
  v_\eps(x,t) \ge C_T. 
\end{align*}
Hence, the diffusivity $\ue \ve$ is now not doubly nonlinear degenerate since 
\begin{align*}
 \ue v_\eps \ge C_T \ue . 
\end{align*}
Moreover, by Lemmas \ref{L:Weak:L2Est} and \ref{L:Weak:BS}, the quantities $\ue^\alpha v_\eps \nabla v_\eps$ and $\ell \ue v_\eps$ are uniformly-in-$\eps$ bounded in $L^\infty(0,T;L^p(\Omega))$ for any $1\le p<\infty$. Hence, we can apply the Moser-Alikakos iteration argument in \cite[Lemma A.1]{tao2012boundedness} to the equation 
\begin{align*} 
 \pa_t u_{\eps } = \di (\ue v_\eps \nabla u_\eps ) - \di (S_{\alpha,\varepsilon}(\ue) v_\eps \nabla v_\eps) + \ell \ue v_\eps
\end{align*}
to obtain the boundedness of $\ue $ in $L^\infty(Q_T)$, where it is necessary to emphasize that this boundedness is uniform in $\eps$. In particular, we imply from  \eqref{L:LocalExis:State1} that $T_{\max,\eps}=\infty$. 
\end{proof}

\begin{lemma}
 \label{L:Weak:HolReg} Let $0\leq \alpha\leq 1$.  There exist $0<\theta_1, \theta_2<1$ such that 
 \begin{align}
\|\ue\|_{C^{\theta_1
,\theta_1/2} (\overline{\Omega}\times[0,T]) } + \|\ve\|_{C^{2+\theta_2
,1+\theta_2/2} (\overline{\Omega}\times(0,T]) } \leq C_T .
 \label{L:Weak:HolReg:S}
 \end{align}
\end{lemma} 

\begin{proof} Thanks to the boundedness of $\ve$ in Lemma \ref{L:LocalExis} as well as its positively lower boundedness given by Lemma \ref{L:Weak:Global}, we can argue similarly to Lemmas 3.9-3.10 in \cite{winkler2024bounds}, where the equation for $\ue$ can be particularly viewed as a scalar parabolic with a degenerate diffusion of porous medium type, to obtain \eqref{L:Weak:HolReg:S}. 
\end{proof}

\subsection{Global existence of a weak solution} \label{Sec:passToLimit}

\begin{proof}[\underline{Proof of Theorem \ref{MainTheo} with $0\leq \alpha\le 1$}] Based on the uniform-in-$\eps$ regularity in Lemma \ref{L:Weak:HolReg}, we can apply the Arzelà–Ascoli theorem that, up to a subsequence, 
\begin{align*}
 \begin{array}{llclll}
 \ue \to u & \text{in} & C^0_{loc}(\overline{\Omega}\times [0,\infty)), \\
 \ve \to v & \text{in} & C^0_{loc}(\overline{\Omega}\times [0,\infty)) \cap C^{2,1}_{loc}(\overline{\Omega}\times (0,\infty)),
 \end{array}
\end{align*}
which also yields that $S_{\alpha,\varepsilon}(\ue) \to u^\alpha$ in $C^0_{loc}(\overline{\Omega}\times [0,\infty))$. 
Moreover, taking $p=3$ in Lemma \ref{L:Weak:BS} and using the positively lower boundedness of $\ve$ in Lemma \ref{L:Weak:Global}, we get 
\begin{align*}
 \intQT |\nabla \ue^2|^2 \le C_T \intQT \ue^{2}\ve|\nabla \ue|^2 \le C_T, 
\end{align*}
and therefore, 
\begin{align*}
 \nabla \ue^2 \rightharpoonup \nabla u^2 \quad \text{(weakly) in } L^2(Q_T)^3. 
\end{align*}
Now, it follows from the regularised problem 
\eqref{System:RegPro0} that 
\begin{gather*} 
 -\iint_{Q_T} \ue\varphi_{t}-\int_{\Omega} (u_{0}+\eps)\varphi(0) = \iint_{Q_T} \left(- \frac{1}{2} v_\eps \nabla u_\eps^2 \cdot \nabla \varphi + S_{\alpha,\varepsilon}(u_\eps) v_\eps \nabla v_\eps \cdot \nabla \varphi + \ell u_\eps v_\eps \varphi \right), \\  \iint_{Q_T} v_\eps\varphi_t+\intO v_0\varphi (0)=\iint_{Q_T} ( \nabla v_\eps\cdot\nabla\varphi + u_\eps v_\eps \varphi ), 
 \end{gather*} 
for all $\phi\in C^{\infty}_0(\overline{\Omega}\times[0,T))$ such that $\frac{\pa\phi}{\pa\nu} =0$ on $\Gamma \times (0,T)$. With the above convergence, 
  passing $\eps\rightarrow 0$ gives directly the weak formulation given in Definition \ref{Def:Weak}, that is, $(u,v)$ is a global weak solution to \eqref{System:Main}. 
\end{proof}

\section{Moderate chemotactic effect}
\label{Sec:Moderate}

This section establishes the global existence of a weak solution to the problem, where the cross-diffusion effect is moderate in the sense that $$1 < \alpha \le \frac{3}{2}.$$
We note that the result \cite[Theorem 1.1]{li2022large} showed the global existence of a weak solution to \eqref{System:Main} for $7/6 < \alpha < 13/9$.

\begin{lemma}\label{L:st:vNablav}
 Let $1<\alpha\le 2$. Then there exits $C_0>0$ such that 
 \begin{align}
 \begin{aligned}
  & \frac{d}{dt} \intO \left( 4C_0(\ue \log \ue -\ue ) + \frac{|\nabla \ve|^{4}}{\ve^{3}} \right) + C_0 \intO \ve|\nabla \ue|^2 \\
  & + 2 \intO \frac{|\nabla \ve|^{2}}{\ve}|D^{2}\log \ve|^{2} + \intO \ue \ve^{-3}|\nabla \ve|^{4} \\
  & \hspace{2cm} \leq C \left( \intO \ue^{2\alpha -2} \ve|\nabla \ve|^2 + \intO \ue \ve \log \ue \right), 
 \end{aligned}
  \label{L:Basic:vNablav:State1}
 \end{align} 
 whenever the right-hand side is finite.
\end{lemma}

\begin{proof} This lemma will be proved by considering the evolution of the quantities $\ue \log \ue-\ue$ and $|\nabla \ve|^4/\ve^3$. For the first one, it is easy to check that 
\begin{align*}
 \frac{d}{dt} \intO (\ue \log \ue-\ue) &= - \intO \ve|\nabla \ue|^2 + \intO \ue^{\alpha-1} \ve \nabla \ue \cdot \nabla \ve + \ell \intO \ue \ve \log \ue \\
 &\le - \frac{1}{2} \intO \ve|\nabla \ue|^2 + \frac{1}{2} \intO \ue^{2\alpha -2} \ve|\nabla \ve|^2 + \ell \intO \ue \ve \log \ue.
\end{align*}
On the other hand, thanks to \cite[Lemma 3.2]{li2022large}, we have 
 \begin{align*}
  \frac{d}{dt} \intO \frac{|\nabla \ve|^{4}}{\ve^{3}} \leq -2 \intO \ve^{-1}|\nabla \ve|^{2}|D^{2}\log \ve|^{2} - \intO \ue \ve^{-3}|\nabla \ve|^{4} 
  + C \intO \ve|\nabla \ue|^2. 
 \end{align*}
 Combining the above estimates directly gives \eqref{L:Basic:vNablav:State1}. 
\end{proof}

\subsection{The uniform boundedness of $\ue^m\ve$ with $m=2$} 
\label{Sec:Mod:m=2}

Let us begin with an estimate for the logarithmic term $\ue \log \ue$ presented in Lemma \ref{L:Mod:FirstEst}, which will be based on Lemma \ref{L:st:vNablav} and estimates obtained in Lemmas \ref{L:BasicEst:v} and \ref{L:Basic:MixUV}. Then, we will obtain the following estimates 
\begin{align}
  \intQT \ue^{r} \ve|\nabla \ue|^2 \quad \text{for } 2-2\alpha <r\le 0 
  \quad \text{and} \quad \intQT \frac{|\nabla \ve|^6}{\ve^5}  
  \label{L:Mod:Intro:1}
 \end{align} 
 in Lemma \ref{L:Mod:vNablav}, 
as well as estimates for the product $\ue^{m} \ve $, for some $m\in \mathbb{R}$, 
in Lemma \ref{L:Mod:Key0}, which allows us to improve the estimate for the first term in \eqref{L:Mod:Intro:1} up to the wider range $-\alpha < r\le 0$ in Lemma \ref{L:Mod:Key}.

\begin{lemma}\label{L:Mod:FirstEst}
Let $1<\alpha\le 3/2$. 
 It holds that 
 \begin{align}
\begin{aligned}
 & \sup_{0<t<T} \left( \intO \ue(t) \log \ue(t) + \intO \frac{|\nabla \ve(t) |^{4}}{\ve^{3}(t)} \right) + \intQT \ve|\nabla \ue|^2 \\
 & \qquad + \intQT \ue \frac{|\nabla \ve|^{4}}{\ve^3} + \intQT \ve^{-1}|\nabla \ve|^{2}|D^{2}\log \ve|^{2} \leq C_T.
\end{aligned}
 \label{L:Mod:FirstEst:State1}
 \end{align}
\end{lemma}

\begin{proof} By Lemma \ref{L:st:vNablav}, we have 
\begin{align*}
  & \frac{d}{dt} \intO \left( 4C_0(\ue \log \ue -\ue ) + \frac{|\nabla \ve|^{4}}{\ve^{3}} \right) + C_0 \intO \ve|\nabla \ue|^2 \\
  & + 2 \intO \frac{|\nabla \ve|^{2}}{\ve}|D^{2}\log \ve|^{2} + \intO \ue \frac{|\nabla \ve|^{4}}{\ve^3} \\
  & \hspace{2cm} \leq C \left( \intO \ue^{2\alpha -2} \ve|\nabla \ve|^2 + \intO \ue \ve \log \ue \right). 
 \end{align*} 
Let us estimate the term including $\ue^{2\alpha-2}\ve|\nabla \ve|^2$, and note that we keep the last term to apply the Gr\"onwall inequality at the end of this proof. Since the condition $1<\alpha\le 3/2$ guarantees that $2\alpha - 2 \in (0,1]$, an interpolation between $\ue^ 0=1$ and $\ue^ 1=\ue$ yields 
 \begin{align*}
\intQT \ue^{2\alpha -2} \ve|\nabla \ve|^2 \leq C_\alpha \|\ve\|_{L^\infty(Q_T)} \left( \intQT |\nabla \ve|^2 + \intQT \ue|\nabla \ve|^2 \right) 
 \end{align*}
in which the right-hand side finitely exists due to estimates \eqref{L:BasicEst:v:State2} and \eqref{L:Basic:MixUV:State1} in Lemmas \ref{L:BasicEst:v} and \ref{L:Basic:MixUV}. Thus, integrating over time gives 
 \begin{align*}
  & \intO \left( 4C_0(\ue(t) \log \ue(t) -\ue(t) ) + \frac{|\nabla \ve(t) |^{4}}{\ve^{3}(t)} \right) + C_0 \intQT \ve|\nabla \ue|^2 \\
  & + 2 \intQT \frac{|\nabla \ve|^{2}}{\ve}|D^{2}\log \ve|^{2} + \intQT \ue \frac{|\nabla \ve|^{4}}{\ve^3} \\
  & \leq \intO \left( 4C_0((u_0 +\eps) \log(u_0 +\eps) -(u_0+\eps)) + \frac{|\nabla v_{0} |^{4}}{v_{0}^{3}} \right) \\
  & \quad + C_{T,\alpha,\|v_{0}\|_{L^\infty(\Omega)}} + C\|v_{0}\|_{L^\infty(\Omega)} \int_0^t\left(\intO\ue\log\ue+\frac{|\Omega|}{e}\right)ds,
\end{align*}
where we used
$\intO\ue\ve\log\ue\le\|v_0\|_\infty
\intO(\ue\log\ue)_+$ and
$\intO(\ue\log\ue)_+\le\intO\ue\log\ue+|\Omega|/e$.  After adding a
fixed constant to make the energy nonnegative, Gronwall's inequality applies.
Therefore, we obtain an estimate for $\ue \log \ue$ in $L^\infty(0,T;L^1(\Omega))$, 
and so \eqref{L:Mod:FirstEst:State1}. 
\end{proof}

\begin{lemma}
\label{L:Mod:vNablav}
Let $1<\alpha\le 3/2$.  For any $2-2\alpha<r\le 0$,
\begin{align}
  \intQT \ue^{r} \ve|\nabla \ue|^2 + \intQT \frac{|\nabla \ve|^6}{\ve^5} \leq C_{T}.
 \label{L:Mod:vNablav:State1}
 \end{align} 
\end{lemma}

\begin{proof} 
For $r=0$ the assertion follows from Lemma
\ref{L:Mod:FirstEst}.  If $2-2\alpha<r<0$, choose
$\rho\in(2-2\alpha,r]$.  Lemma~\ref{L:Basic:MixUV} gives
\begin{align*}
\intQT\ue^\rho\ve|\nabla\ue|^2\le C_{T,\rho}.
\end{align*}
Since $s^r\le 1+s^\rho$ for $s>0$, Lemma
\ref{L:Mod:FirstEst} yields
$\intQT\ue^r\ve|\nabla\ue|^2\le C_{T,r}$.
On the other hand, the boundedness of the second term in \eqref{L:Mod:vNablav:State1} is obtained using Lemmas \ref{L:Mod:Phi-NablaPhi} and \ref{L:Mod:FirstEst}, which reads 
\begin{align*}
 \intQT \frac{|\nabla \ve|^6}{\ve^5} \le (4+\sqrt{3})^2\intQT \ve^{-1}|\nabla \ve|^{2}|D^{2}\log \ve|^{2} \le C_T, 
\end{align*}
where we notice that, for each $\eps>0$, $\ve$ is smooth enough to apply Lemma \ref{L:Mod:Phi-NablaPhi}.
\end{proof}

\begin{lemma} 
\label{L:Mod:Key0}
Let $1<\alpha\le 3/2$ and $m>-2$.  Then, whenever the terms on theright-hand side are finite,
\begin{align}
 \intQT \ue^{m+\frac{8}{3}} \ve \le C_{T}+ C_{m} \intQT \ue^{m} \ve|\nabla \ue|^2 + C_{m} \intQT \ue^{2m+3} \ve . 
\label{L:Mod:Key0:State}
\end{align}
\end{lemma}

\begin{proof} By applying Lemma \ref{L:SOBOineqn}, we get 
 \begin{align*}
\intQT (\ue^{m+2} \ve ) \ue^{\frac{2}{3}} 
  & \leq \int_0^T \|\ue^{m+2}\ve\|_{\LO{3}} \|\ue^{\frac{2}{3}}\|_{\LO{\frac{3}{2}}} 
\leq M^{\frac{2}{3}} \int_0^T \|\ue^{m+2} \ve\|_{\LO{3}}
  \\
  &\leq C_{m} \intQT \ue^{m} \ve|\nabla \ue|^2+ C_{m} \intQT \ue^{m+2} \ve^{-1}|\nabla \ve|^2 + C_{m}.
 \end{align*}
 Here Lemma~\ref{L:SOBOineqn} is used with parameter $p=m+1>-1$.
Its zero-order term is bounded because
$\intO\ue\ve^{1/(m+2)}\le C_{m,v_0}\intO\ue\le C_m$; moreover,
$\|\ue^{2/3}\|_{L^{3/2}}=(\intO\ue)^{2/3}\le M^{2/3}$.
Then, using the Young's inequality,  
\begin{align*}
  \intQT \ue^{m+2} \ve^{-1}|\nabla \ve|^2 &\leq C \intQT \ue \frac{|\nabla \ve|^4}{\ve^3} + C \intQT \ue^{2m+3} \ve \\
  &\leq C_T + C \intQT \ue^{2m+3} \ve , 
 \end{align*}
where we have used Lemma \ref{L:Mod:FirstEst} for the uniform boundedness of $\ue|\nabla \ve|^4 / \ve^3 $ in $L^1(Q_T)$. Finally, estimate \eqref{L:Mod:Key0:State} is obtained by combining the above estimates.
\end{proof}

Taking into account the boundedness of $\ue |\nabla \ve|^4 / \ve^3 $ in $L^1(Q_T)$, as established in Lemma \ref{L:Mod:FirstEst}, we will improve estimate \eqref{L:Basic:MixUV:P2} provided in Lemma \ref{L:Basic:MixUV}. Then, by Lemma \ref{L:SOBOineqn}, we will obtain the boundedness of $\intQT \ue^ m \ve $ for a suitable $m>1$. 
 
\begin{lemma} \label{L:Mod:Key} 
Let $1<\alpha\le 3/2$.  For any $-\alpha <r \le 0$,
 \begin{align}
  \intQT \ue^{r} \ve|\nabla \ue|^2 + \intQT \ue^{\frac{5}{3}} \ve \le C_{T,r}.
\label{L:Mod:Key:State1}
 \end{align}
\end{lemma}

\begin{proof} For $q\in \mathbb{R}\setminus\{0;1\}$, straightforward computations show 
\begin{align*}
  \intO \ue^{q-1} \ve|\nabla \ue|^2 
  &= \frac{-1}{q(q-1)} \frac{d}{dt} \intO \ue^q + \intO 
  \ue^{q-2+\alpha} \ve \nabla \ue \cdot \nabla \ve + \frac{\ell}{q-1} \intO \ue^q \ve \\
  &= \frac{-1}{q(q-1)} \frac{d}{dt} \intO \ue^q + \frac{1}{2(q-1+\alpha)} \intO 
  \nabla \ue^{q-1+\alpha} \cdot \nabla \ve^2 + \frac{\ell}{q-1} \intO \ue^q \ve \\
  &= \frac{-1}{q(q-1)} \frac{d}{dt} \intO \ue^q - \frac{1}{2(q-1+\alpha)} \intO 
  \ue^{q-1+\alpha} \Delta \ve^2 + \frac{\ell}{q-1} \intO \ue^q \ve , 
 \end{align*} 
where $\Delta \ve^2 = 2|\nabla \ve|^2 + 2 \ve \Delta \ve $. Integrating over time implies
\begin{align*}
  \intQT \ue^{q-1} \ve|\nabla \ue|^2 
  &=\frac{1}{q(q-1)} \intO (u_0 +\eps)^q - \frac{1}{q(q-1)} \intO \ue^q + \frac{\ell}{q-1} \intQT \ue^q \ve \\
  & - \frac{1}{q-1+\alpha} \intQT 
  \ue^{q-1+\alpha} \left( |\nabla \ve|^2 + \ve \Delta \ve \right). 
\end{align*}
Let us choose $0<\delta\ll1$ such that $\delta<\alpha-1$ and set
\begin{align}
 q:=1 -\alpha +\delta,
\end{align}
which is strictly negative. 
Then, we see that
\begin{align*}
 q(q-1) >0 \; \text{ since } \; q<0, \quad \text{and} \quad q-1+\alpha >0,
\end{align*}
and therefore, we can skip the non-negative terms on the right-hand side to obtain 
\begin{align}
  \intQT \ue^{-\alpha+\delta} \ve|\nabla \ue|^2 
  &\le \frac{1}{(\alpha-1-\delta)(\alpha-\delta)} \intO (u_0 +\eps)^{1 -\alpha +\delta} - \frac{1}{\delta} \intQT 
  \ue^{\delta} \ve \Delta \ve . 
  \label{L:Mod:Key:P1}
\end{align}
Here the negative-power initial value is bounded uniformly in
$\eps$: since $-1<1-\alpha+\delta<0$,
\begin{align*}
\intO(u_0+\eps)^{1-\alpha+\delta}
 \le \intO u_0^{1-\alpha+\delta}
 \le |\Omega|+\intO u_0^{-1}<\infty.
\end{align*}
Thanks to the heat regularisation \cite[Lemma B.1]{reisch2024global} applied to the equation for $\ve $, we get
\begin{align}
 \|\Delta \ve\|_{L^{1+\delta}(Q_T)} \le C_{T} \left(  \|v_0\|_{W^{\frac{2\delta}{\delta+1},\delta+1}(\Omega)}  + \|\ue \ve\|_{L^{\delta+1}(Q_T)} \right), 
\end{align}
recalling that $v_0\in W^{2,\infty}(\Omega)$.. 
Hence, we get
\begin{align*}
 - \frac{1}{\delta} \intQT 
  \ue^{\delta} \ve \Delta \ve & \le \frac{1}{\delta}\|\ue^ \delta \ve\|_{L^{\frac{\delta+1}{\delta}}(Q_T)} \|\Delta \ve\|_{L^{\delta+1}(Q_T)} \\
  & \le \frac{C_{T}}{\delta} \|\ue^ \delta \ve\|_{L^{\frac{\delta+1}{\delta}}(Q_T)} \left( \|v_0\|_{W^{\frac{2\delta}{\delta+1},\delta+1}(\Omega)} + \|\ue \ve\|_{L^{\delta+1}(Q_T)} \right) \\
  & = \frac{C_{T}}{\delta} \left( \intQT \ue^{\delta+1} \ve^{\frac{\delta+1}{\delta}} \right)^{\frac{\delta}{\delta+1}} \left( C + \left( \intQT (\ue \ve )^{\delta+1} \right)^{\frac{1}{\delta+1}} \right) .
\end{align*}
We note that $\ue^{\delta+1} \ve^{(\delta+1)/\delta} \le C_\delta \ue^{\delta+1} \ve $ and $(\ue \ve )^{\delta+1} \le C_\delta \ue^{\delta+1} \ve $ since $\ve \in L^\infty(Q_T)$. Consequently, we obtain
\begin{align*}
 - \frac{1}{\delta} \intQT 
  \ue^{\delta} \ve \Delta \ve 
  & \le C_{T,\delta} \left( \intQT \ue^{\delta+1} \ve \right)^{\frac{\delta}{\delta+1}} \left( C + \left( \intQT \ue^{\delta+1} \ve \right)^{\frac{1}{\delta+1}} \right),
\end{align*}
and by the Young's inequality, 
\begin{align}
 - \frac{1}{\delta} \intQT 
  \ue^{\delta} \ve \Delta \ve 
  & \le C_{T,\delta} \left( 1 + \intQT \ue^{\delta+1} \ve \right) .
  \label{L:Mod:Key:P2}
\end{align}
It follows from \eqref{L:Mod:Key:P1} and \eqref{L:Mod:Key:P2} that 
\begin{align}
  \intQT \ue^{-\alpha+\delta} \ve|\nabla \ue|^2 
  &\le C_{T,\delta,\alpha} \left( 1 + \intQT \ue^{\delta+1} \ve \right) . 
  \label{L:Mod:Key:P3}
\end{align}

Now, let us use Lemma \ref{L:Mod:Key0}. By taking $\widetilde{\delta}$ such that $\delta<\widetilde{\delta}\le \alpha-1$, and then choosing $m=-\alpha +\widetilde{\delta}$ in estimate \eqref{L:Mod:Key0:State} to have that 
\begin{align}
 \intQT \ue^{\frac{8}{3}-\alpha +\widetilde{\delta}} \ve \le C_{T}+ C_{\widetilde{\delta},\alpha} \left( \intQT \ue^{-\alpha +\widetilde{\delta}} \ve|\nabla \ue|^2 + \intQT \ue^{3-2\alpha +2\widetilde{\delta}} \ve \right).
 \label{L:Mod:Key:P4}
\end{align}
On the other hand, by the choice of $\widetilde{\delta}$, one has $1<\delta+1 < \frac{8}{3}-\alpha +\widetilde{\delta}$. Therefore, the Young's inequality yields that $\ue^{\delta+1} \le C_\delta (\ue + \ue^{8/3-\alpha+\widetilde{\delta}})$. Since $\ue \ve$ is uniformly bounded in $L^1(Q_T)$, 
\begin{align}
\begin{aligned}
 \intQT \ue^{\delta+1} \ve &\le C_\delta \left( \intQT \ue \ve + \intQT \ue^{\frac{8}{3}-\alpha +\widetilde{\delta}} \ve \right) \\
 & \le C_\delta \left(1 + \intQT \ue^{\frac{8}{3}-\alpha +\widetilde{\delta}} \ve \right) .
\end{aligned} 
 \label{L:Mod:Key:P5}
\end{align}
A combination of \eqref{L:Mod:Key:P3}-\eqref{L:Mod:Key:P5}, we get 
\begin{align}
  \intQT \ue^{-\alpha+\delta} \ve|\nabla \ue|^2 
  &\le C_{T} \left( 1 + \intQT \ue^{-\alpha +\widetilde{\delta}} \ve|\nabla \ue|^2 + \intQT \ue^{3-2\alpha +2\widetilde{\delta}} \ve \right) ,
  \label{L:Mod:Key:P6}
\end{align}
where the constant $C_T$ depends on $\delta,\widetilde{\delta}$ and $\alpha$. Since $-\alpha+\delta<-\alpha + \widetilde{\delta} < 0$, we can use again the Young's inequality to see that $\ue^{-\alpha +\widetilde{\delta}} \le \kappa \ue^{-\alpha + \delta} +  C_{\kappa}$, where $\alpha,\delta,\widetilde{\delta}$ are implicit in the constant $C$. Hence, for any constant $\kappa>0$,
\begin{align}
\begin{aligned}
 \intQT \ue^{-\alpha +\widetilde{\delta}} \ve|\nabla \ue|^2 & \le C_\kappa \intQT \ve|\nabla \ue|^2 + \kappa \intQT \ue^{-\alpha +\delta} \ve|\nabla \ue|^2 \\
 & \le C_{T,\kappa} + \kappa \intQT \ue^{-\alpha +\delta} \ve|\nabla \ue|^2,
\end{aligned}
 \label{L:Mod:Key:P7}
\end{align}
where $\ve|\nabla \ue|^2$ is uniformly bounded in $L^1(Q_T)$ due to Lemma \ref{L:Mod:vNablav}. Then, a combination of estimates \eqref{L:Mod:Key:P6}-\eqref{L:Mod:Key:P7} shows 
\begin{align}
  \intQT \ue^{-\alpha+\delta} \ve|\nabla \ue|^2 
  &\le C_{T,\kappa} + \kappa C_T \intQT \ue^{-\alpha +\delta} \ve|\nabla \ue|^2 + C_T \intQT \ue^{3-2\alpha +2\widetilde{\delta}} \ve .
  \label{L:Mod:Key:P8}
\end{align}
Since $\widetilde{\delta}\le \alpha-1$, the exponent $3-2\alpha +2\widetilde{\delta}$ belongs to the close interval $[0,1]$, i.e., the last term of \eqref{L:Mod:Key:P6} is uniformly bounded due to Lemma \ref{L:LocalExis}. Finally, the estimate for the first term in \eqref{L:Mod:Key:State1} is obtained by choosing $\kappa$ small enough, where the range of $-\alpha<r\le 0$ is obtained by an interpolation between the uniform boundedness of $ \ue^{-\alpha+\delta} \ve|\nabla \ue|^2$ and $\ve|\nabla \ue|^2$ (given in Lemma \ref{L:Mod:vNablav}) in $L^1(Q_T)$. Finally, the estimate for $\ue^{5/3} \ve $ in $L^1(Q_T)$ is derived by choosing $\widetilde{\delta}=\alpha-1$ in \eqref{L:Mod:Key:P4}.
\end{proof}

\begin{lemma}\label{L:Mod:m=2}
Let $1<\alpha\le 3/2$.  It holds that 
 \begin{align*}
  \intQT u_{\eps}^2 \ve \leq C_T.
 \end{align*}
\end{lemma}
\begin{proof} By choosing $m=-2/3$ in Lemma \ref{L:Mod:Key0}, we get 
\begin{align*}
 \intQT \ue^{2} \ve \le C_{T}+ C \intQT \ue^{-\frac{2}{3}} \ve|\nabla \ue|^2 + C \intQT \ue^{\frac{5}{3}} \ve , 
\end{align*}
where the right-hand side is uniformly bounded according to Lemma \ref{L:Mod:Key}.
\end{proof}

\subsection{Mutual feedback up to the boundedness of $\ue^m\ve$ for $m>6$}
\label{Sec:M:FB}

We aim to perform feedback arguments of the forms (M1) and (M2) as introduced in the first section, which will serve our bootstrap argument in the following subsection.

\begin{lemma}
\label{L:Mod:SeqDef}
Let $1<\alpha\le 3/2$.  Assume that the sequence $\{(m_k,p_k,r_k)\}_{k=0,1,\dots} \subset \mathbb{R}^3$ is defined by $m_0\ge 2$, and for $k\ge 0$,
\begin{gather*}
  p_{k}:= \frac{m_{k}}{2} + \frac{7}{2} - 2\alpha, \quad 
  r_{k} : = \min\left\{ \frac{4}{3}p_{k} -2 ;\, p_{k}-1 \right\} , \quad 
m_{k+1} := \frac{2}{3}p_k + r_k +2 .  
\end{gather*} 
Then, the following conclusions hold.
\begin{itemize}
 \item[a)] For all $k\ge 0$, we have $p_k>1$ and $\frac{3}{2}(r_k+2)\le m_{k+1}$; 
 \item[b)] There exists $k_0$ such that $p_{k_0}>3$ and $m_{k_0+1}>6$; and 
 \item[c)] The sequence $\{m_k\}$ is not increasing as $p_k>24-12\alpha$.
\end{itemize}

\end{lemma}

\begin{proof} Part a is obvious since $r_k \le 4p_k/3-2$. Let us prove Part b. Assume that $p_k\le 3$ for all $k\ge 1$. Then, we have $r_k=4p_k/3-2$ and so $m_{k+1}=2p_k$. This implies 
\begin{align*}
 p_{k+1} = p_k + \frac{7}{2}-2\alpha = \dots = p_0 + (k+1)\left(\frac{7}{2}-2\alpha\right) \to \infty \quad \text{as } k \to \infty. 
\end{align*} 
By contradiction, there exists $k_0\ge 0$ such that $p_{k_0}>3$. 
Consequently, 
\begin{align*}
 m_{k_0+1} = \frac{2}{3}p_{k_0} + r_{k_0} +2 = \frac{5}{3}p_{k_0} + 1 > 6. 
\end{align*}
For Part c, since $p_k>24-12\alpha$, we have $p_k>3$ and $r_k=p_k-1$. Thus, we get 
\begin{align*}
 m_{k+1}-m_k= \left( \frac{2}{3}p_k + \left(p_k -1 \right) +2 \right) - \left( 2p_k -7 + 4\alpha \right) = - \frac{1}{3}p_k+ 8-4\alpha <0, 
\end{align*}
i.e., $\{m_k\}$ is not increasing for $k\in \{j\in \mathbb{N}: p_j>24-12\alpha\}$.  
\end{proof}

\begin{lemma}\label{L:Mod:BS} 
Let $1<\alpha\le 3/2$ and $\{(m_k,p_k,r_k)\}_{k=0,1,\dots}$ be the sequence defined by Lemma \ref{L:Mod:SeqDef}. If
 \begin{align}
\intQT \ue^{m_k} \ve \leq C_T \quad \text{for some } m_k\ge 2, 
  \label{L:Mod:BS:S}
 \end{align} 
 then it holds 
 \begin{align}
\sup_{0<t<T}\intO \ue^{p_k}(t) + \intQT \ue^{r} \ve|\nabla \ue|^2 \leq C_{T,k}, \label{L:Mod:BS:S2}
 \end{align} 
 for all $r\in V_k:= \{r \in \mathbb{R}: -\alpha < r \leq p_k-1\} .$ 
\end{lemma}

\begin{proof} For this proof, we will use the following computation of the $L^p$-energy 
\begin{align}
\begin{aligned} 
 & \intO \ue^{p} + p(p-1) \intQt \ue^{p-1} \ve|\nabla \ue|^2 \\
 & = \intO (u_0 + \eps)^{p} + p(p-1)\intQt \ue^{p-2+\alpha} \ve \nabla \ue \cdot \nabla \ve + \ell p \intQt \ue^{p} \ve 
 \end{aligned}
 \label{L:Mod:BS:P1}
\end{align}
for $p=p_k$. Since $m_k\ge2$ and $\alpha\le3/2$, we have
\begin{align*}
 p_k=\frac{m_k}{2} + \frac{7}{2}-2\alpha
 \ge \frac{9}{2}-2\alpha\ge\frac32>1.
\end{align*}
Therefore, all terms on the left-hand side of \eqref{L:Mod:BS:P1} are nonnegative. Let us estimate the term including $\ue^{p-2+\alpha} \ve \nabla \ue \cdot \nabla \ve$. Using the Young's inequality, we will split this term into two other terms. The first one can be absorbed into the left-hand side, while the second one can be estimated due to Lemma \ref{L:Mod:FirstEst} and the assumption \eqref{L:Mod:BS:S}. 
Indeed, 
\begin{align}
 \intQt \ue^{p_k-2+\alpha} \ve \nabla \ue \cdot \nabla \ve \le \frac{1}{2} \intQt \ue^{p_k-1} \ve|\nabla \ue|^2 + \frac{1}{2} \intQt \ue^{p_k-3+ 2\alpha} \ve|\nabla \ve|^2,
 \label{L:Mod:BS:P1a}
\end{align}
where, by Lemma \ref{L:Mod:FirstEst}  and the assumption \eqref{L:Mod:BS:S}, 
\begin{align}
  \begin{split}
  \intQt \ue^{p_k-3+ 2\alpha} \ve|\nabla \ve|^2 & \le \frac{1}{4} \intQt \ue \frac{|\nabla \ve|^{4}}{\ve^3} + \intQt \ue^{2p_k + 4\alpha - 7} \ve^5
 \\
 & \le C_T + \|v_0\|_{L^\infty(Q_T)}^4 \intQt \ue^{m_k} \ve  \leq C_T.
  \end{split}
  \label{L:Mod:BS:P1b}
\end{align} 
It follows from \eqref{L:Mod:BS:P1} and the above estimates that 
\begin{align}
\begin{aligned}
 \intO \ue^{p_k}(t) + \frac{p_k(p_k-1)}{2} \intQt \ue^{p_k-1} \ve|\nabla \ue|^2 \leq C_{p_k,T} + \ell p_k \|v_0\|_{\LO{\infty}} \intQt \ue^{p_k} .
\end{aligned}
  \label{L:Mod:BS:P2}
 \end{align}
  We get the estimate for the first term in \eqref{L:Mod:BS:S2}, i.e. the spatial integral of $ \ue^{p_k}$, by using the Gr\"onwall inequality. For a prescribed
$r\in(-\alpha,p_k-1]$, choose $\rho\in(-\alpha,\min\{r,0\}]$.  Since $s^r\le s^\rho+s^{p_k-1}$ for $s>0$, Lemma \ref{L:Mod:Key} and
\eqref{L:Mod:BS:P2} give
\begin{align*}
 \intQT \ue^{r} \ve|\nabla \ue|^2 \le C \left( \intQT \ue^{\tilde{r}} \ve|\nabla \ue|^2 + \intQT \ue^{p_k-1} \ve|\nabla \ue|^2 \right) \le C_T, 
 \end{align*}
 where we have used Lemma \ref{L:Mod:Key}. 
\end{proof}

\begin{lemma} \label{L:Mod:BS2} 
Let $1<\alpha\le 3/2$ and $\{(m_k,p_k,r_k)\}_{k=0,1,\dots}$ be the sequence defined by Lemma \ref{L:Mod:SeqDef}. If 
\begin{align}\label{L:Mod:BS2:State1}
  \sup_{0<t<T} \intO \ue^{p_k}(t) \leq C_T \quad \text{for } k\ge 0,
 \end{align} 
then 
 \begin{align}
  \intQT \ue^{m_{k+1}} \ve \leq C \intQT \ue^{r_k} \ve|\nabla \ue|^2+ C \intQT \ue^{r_k+2} \ve^{-1}|\nabla \ve|^2 + C
  \label{L:Mod:BS2:S2}
 \end{align}
 and 
\begin{align}\label{L:Mod:BS3:State2}
  \intQT \ue^{r_k+2} \ve^{-1}|\nabla \ve|^2 \leq C_{T,\eta,k} +  \eta \intQT 
 \ue^{m_{k+1}} \ve , 
 \end{align} 
whenever the above right-hand sides are finite.
\end{lemma}

\begin{proof} Let us first prove \eqref{L:Mod:BS2:S2}. We can observe the relation $ (3/2)( m_{k+1} - (r_k+2)) = p_k$. Consequently, for each $t\in (0,T)$, the assumption \eqref{L:Mod:BS2:State1} ensures that 
\begin{align}
 \|\ue^{m_{k+1} - (r_k+2)}(t) \|_{\LO{\frac{3}{2}}} = \|\ue(t) \|_{\LO{p_k}}^{m_{k+1} - (r_k+2)} \le \|\ue\|_{L^\infty(0,T;\LO{p_k})}^{m_{k+1} - (r_k+2)} \le C_T.
 \label{L:Mod:BS2:P1}
\end{align}
Therefore, by the H\"older inequality, 
 \begin{align*} 
 \intO \ue^{m_{k+1}}(t) \ve(t) & = \intO \big(\ue^{r_k+2}(t) \ve(t) \big) \ue^{m_{k+1} - (r_k+2)}(t) \\
  & \leq \|\ue^{r_k+2}(t) \ve(t) \|_{\LO{3}} \|\ue^{m_{k+1} - (r_k+2)}(t) \|_{\LO{\frac{3}{2}}} \\
  & \leq \|\ue^{r_k+2}(t) \ve(t) \|_{\LO{3}} \|\ue\|_{L^\infty(0,T;\LO{p_k})}^{m_{k+1} - (r_k+2)},
 \end{align*}
and consequently,
\begin{align*} 
 \intO \ue^{m_{k+1}}(t) \ve(t) \leq C_{T,k} \|\ue^{r_k+2}(t) \ve(t) \|_{\LO{3}} .
 \end{align*}
By the boundedness \eqref{L:Mod:BS2:P1}, we only have to estimate the term including $\ue^{r_k+2}(t) \ve(t)$, which can be based on Lemma \ref{L:SOBOineqn} as
\begin{align*} 
  \|\ue^{r_k+2}(t) \ve(t) \|_{\LO{3}} 
\leq C_T \left( \intO \ue^{r_k} \ve|\nabla \ue|^2+ \intO \ue^{r_k+2} \ve^{-1}|\nabla \ve|^2 +1 \right),
 \end{align*}
 where we used the boundedness of $\intO \ue$ and $\intO \ue\ve$.
 
Taking the above estimates and integrating them over time, we directly obtain \eqref{L:Mod:BS2:S2}.

\medskip

We proceed to prove \eqref{L:Mod:BS3:State2}. Since $r_k \le 4p_k/3-2$, we have $3(r_k+2)/2\le m_{k+1}$. Then, there exists $C_{r_k}>0$ such that $ 
 \ue^{3(r_k+2)/2} \ve \le C_{r_k} 
 \ue^{m_{k+1}} \ve + C_{r_k} . $ Therefore,
for any positive constant $\eta>0$, the Young's inequality yields
 \begin{align*} 
  \intQT \ue^{r_k+2} \ve^{-1}|\nabla \ve|^2 &\leq C_{\eta,r_k} \intQT \frac{ |\nabla \ve|^{6}}{\ve^{5}} + \frac{\eta}{C_{r_k}} \intQT \ue^{\frac{3}{2}(r_k+2)} \ve \\
  &\leq C_{\eta,r_k} \intQT \frac{ |\nabla \ve|^{6}}{\ve^{5}} + \eta \intQT 
 \ue^{m_{k+1}} \ve + C_{\eta,r_k} , 
 \end{align*}
which implies \eqref{L:Mod:BS3:State2} due to the boundedness of $|\nabla \ve|^6/\ve^5$ in $L^1(Q_T)$, given in Lemma \ref{L:Mod:vNablav}. 
\end{proof}

\begin{lemma}
\label{L:Mod:m>6}
Let $1<\alpha\le 3/2$.  There exists $m_*>6$ such that 
 \begin{align*}
  \intQT u_{\eps}^{m} \ve \leq C_T, 
 \end{align*}
for all $2\le m \le m_*$.
\end{lemma}

\begin{proof} Let us consider a sequence $(m_k,p_k,r_k)$ defined by Lemma \ref{L:Mod:SeqDef} with $m_0:=2$. Then, by Lemma  \ref{L:Mod:SeqDef}, the condition \eqref{L:Mod:BS:S} holds for $k=0$ (see Lemma \ref{L:Mod:m=2}). Hence, using Lemma \ref{L:Mod:BS},  
\begin{align*}
\sup_{0<t<T}\intO \ue^{p_0}(t) + \intQT \ue^{r} \ve|\nabla \ue|^2 \leq C_{T,0}, 
 \end{align*} 
for all $r\in V_0$. Then, an application of Lemma \ref{L:Mod:BS2} deduces
 \begin{align*}
  \intQT \ue^{m_1} \ve & \leq C \intQT \ue^{r_0} \ve|\nabla \ue|^2+ C \intQT \ue^{r_0+2} \ve^{-1}|\nabla \ve|^2 + C_T\\
  & \leq C \intQT \ue^{r_0} \ve|\nabla \ue|^2+ C_{T} +  \frac{1}{2} \intQT 
 \ue^{m_{1}} \ve, 
 \end{align*}
which consequently shows that the condition \eqref{L:Mod:BS:S} holds for $k=1$. We complete this proof by iterating this argument and noting part b) of Lemma \ref{L:Mod:SeqDef}. 
\end{proof}

\subsection{Uniform boundedness via a two-phase bootstrap argument}
\label{Sec:M:2BS}

We will improve the previous section's arguments that help us prove the boundedness of $\ue$ in $L^\infty(Q_T)$. The weak point of those arguments is the use of $\ue^{2p_k + 4\alpha - 7} \ve^5$ to control $\ue^{p_k-3+ 2\alpha} \ve|\nabla \ve|^2$, see \eqref{L:Mod:BS:P1a}-\eqref{L:Mod:BS:P1b}, where $2p_k + 4\alpha - 7$ is quite large. This causes the non-increasing property of the sequence $\{m_k\}$ as $p_k$ exceeds $3$. However, thanks to Lemma \ref{L:Mod:m>6}, we can enhance estimate \eqref{L:Mod:BS:P1b} by proving that $\nabla \ve$ is uniformly bounded. 

\begin{lemma} 
\label{L:Mod:BS3}
Let $1<\alpha\le 3/2$.  It holds that 
 \begin{align*}
  \|\nabla \ve\|_{L^\infty(Q_T)} \le C_T . 
 \end{align*}
\end{lemma} 

\begin{proof} Let $m_*$ be given by Lemma \ref{L:Mod:m>6}. Thanks to the heat regularisation, see \cite[Theorem 9.1, Chapter IV]{ladyvzenskaja1988linear}, for a sufficiently small $\delta>0$ such that $5+\delta < m_*$, we follow from the equation for $\ve$ that 
\begin{align*}
 \|\nabla \ve\|_{L^\infty(Q_T)} &\le C_T \left( \|\Delta v_0\|_{W^{2,5+\delta}(\Omega)} + \|\ue\ve\|_{L^{5+\delta}(Q_T)} \right) \\
 &\le C_T \left( \|\Delta v_0\|_{W^{2,5+\delta}(\Omega)} + \left( \|\ve\|_{L^\infty(Q_T)}^{4+\delta} \intQT \ue^{5+\delta}\ve \right)^{\frac{1}{5+\delta}} \right).
\end{align*}
Here, the latter integral is uniformly bounded due to Lemma \ref{L:Mod:m>6}. 
\end{proof}

\begin{lemma}
\label{L:Mod:SeqDef2}
Let $1<\alpha\le 3/2$.  Assume that the sequence $\{(\widehat m_k,\widehat p_k,\widehat r_k)\}_{k=0,1,\dots} \subset \mathbb{R}^3$ is defined by $\widehat m_0> 6$, and for $k\ge 0$, 
\begin{align*} 
\widehat p_k:= \widehat m_k+3-2\alpha,  \quad 
  \widehat r_k : = \widehat p_k-1 ,  \quad \widehat m_{k+1} := \frac{2}{3}\widehat p_k + \widehat r_k +2 .
\end{align*} 
Then, the following conclusions hold.
\begin{itemize}
 \item[a)] For all $k\ge 0$, we have $\widehat p_k>3$ and $\frac{3}{2}(\widehat r_k+2)\le m_{k+1}$; 
 \item[b)] All estimates given in the statements of Lemmas \ref{L:Mod:BS}-\ref{L:Mod:BS2} are held if $(m_k,p_k,r_k)$ is replaced by $(\widehat m_k,\widehat p_k,\widehat r_k)$; and
 \item[c)] The components $\widehat p_k, \widehat m_k$ are increasing to infinity, i.e.,
 $$\lim_{k\to \infty} \widehat p_k = \lim_{k\to \infty} \widehat m_k = \infty .$$ 
\end{itemize}
\end{lemma}

\begin{proof} Part a of this lemma is obvious. For Part b, we only have to improve estimates \eqref{L:Mod:BS:P1a}-\eqref{L:Mod:BS:P1b}. Indeed, thanks to Lemma \ref{L:Mod:BS3}, can be enhanced as 
\begin{align*}
 \intQt \ue^{\widehat p_k-2+\alpha} \ve \nabla \ue \cdot \nabla \ve & \le \frac{1}{2} \intQt \ue^{\widehat p_k-1} \ve|\nabla \ue|^2 + \frac{1}{2} \intQt \ue^{\widehat p_k-3+ 2\alpha} \ve|\nabla \ve|^2 \\
 & \le \frac{1}{2} \intQt \ue^{\widehat p_k-1} \ve|\nabla \ue|^2 + \frac{1}{2} \|\nabla \ve\|_{L^\infty(Q_T)}^2 \intQt \ue^{\widehat m_k} \ve , 
\end{align*}
and therefore, one can follow the rest of the proof of Lemma \ref{L:Mod:BS2}. For Part c, we note that 
\begin{align*}
 \widehat m_{k+1} - \widehat m_{k} = \left( \frac{2}{3}\widehat p_k + \widehat r_k +2 \right) - (\widehat p_k-3+2\alpha) = \frac{2}{3}\widehat p_k + 4-2\alpha > 6 - 2\alpha 
\end{align*}
since $\widehat p_k > 3$ for all $k\ge 0$. Equivalently,
$\widehat m_{k+1}=\frac53\widehat m_k+6-\frac{10}{3}\alpha
\ge\frac53\widehat m_k+1$; hence both $\widehat m_k$ and
$\widehat p_k$ tend to infinity.  
\end{proof}

\begin{corollary} 
\label{Coro:Mod:LInf}
Let $1<\alpha\le 3/2$. For any $1\le p<\infty$, it holds that 
 \begin{align} 
  \sup_{0<t<T} \intO \ue^{p}(t)  \leq C_{T,p},
  \label{Coro:Mod:LInf:S1}
 \end{align}
where $C_{T,p}$ may tend to infinity as $p\to \infty$. Consequently, 
\begin{align}
 0< C_{T,v_0} \le \|v_\eps\|_{L^\infty(Q_T)} \quad \text{and} \quad \|\ue\|_{L^\infty(Q_T)} \le C_T ,
 \label{Coro:Mod:LInf:S2}
\end{align}
and $T_{\max,\varepsilon} = \infty$ for all $\varepsilon>0$, i.e. the global solvability of the regularised problem is claimed.
\end{corollary} 

\begin{proof} To prove the estimate \eqref{Coro:Mod:LInf:S1}, we can repeat the same argument as Lemma \ref{L:Mod:m>6} based on Lemma \ref{L:Mod:SeqDef2}, where we begin with $\widehat m_0=m_*$. The estimate \eqref{Coro:Mod:LInf:S2} can be proved using the same arguments as Lemma \ref{L:Weak:Global}.  
\end{proof}

 \subsection{Global existence of a weak solution}

 \begin{proof}[\underline{Proof of Theorem \ref{MainTheo} with $1<\alpha\le \frac{3}{2}$}] Using the same argument as Lemma \ref{L:Weak:HolReg}, one can obtain the H\"older regularity 
\begin{align}
\|\ue\|_{C^{\theta_1
,\theta_1/2} (\overline{\Omega}\times[0,T]) } + \|\ve\|_{C^{2+\theta_2
,1+\theta_2/2} (\overline{\Omega}\times(0,T]) } \leq C_T ,
 \end{align}
for some $0<\theta_1,\theta_2<1$. Therefore, the proof of this case can be performed similarly to the case $0\leq \alpha\le 1$ in Section \ref{Sec:Weak}.
\end{proof}
 
\section{Strong chemotactic effect}
\label{Sec:Stro}

This section establishes the global existence of a weak solution to the problem, where the cross-diffusion effect is moderate in the sense that 
\begin{align}
 \frac{3}{2} <\alpha < 2.
 \label{Case3}
\end{align}

\subsection{The uniform boundedness of $\ue^m\ve$ with $m = 4/3$}
\label{Sec:St:m43}

We observe from Lemma \ref{L:Mod:Key0} that estimate \eqref{L:Mod:Key0:State} only make sense if $m+8/3 \le 2m+3$, i.e., $m\ge -1/3$. Under this restriction of $m$, the boundedness of $ \ue^{m} \ve|\nabla \ue|^2$ in $L^1(Q_T)$ cannot be guaranteed by Lemma \ref{L:Basic:MixUV} unless $3/2<\alpha\le 5/3$. 
In addition, to obtain the uniform boundedness of  $\ve^{-5}|\nabla \ve|^6$ in $L^1(Q_T)$, it requires that of $\ve^{-1}|\nabla \ve|^2 |D^2 \log \ve|^2$. This means that it is necessary to estimate the right-hand side of the estimate \eqref{L:st:vNablav}, given in
Lemma \ref{L:st:vNablav}, including $\ue^{2\alpha -2} \ve|\nabla \ve|^2$. Moreover, the estimate 
\begin{align*}
\intQT \ue^{2\alpha -2} \ve|\nabla \ve|^2 \le \left( \kappa \intQT \frac{|\nabla \ve|^6}{\ve^5} + C_\kappa \intQT \ue^{3\alpha -3} \ve \right), \quad 0<\kappa\ll 1, 
\end{align*}
arises the ``big" exponent $3\alpha-3 \in (3/2,3)$, for which the uniform-regularity of the solution is low and not sufficient to control. This is also the main difficulty in \cite{li2022large}, which explains why the case \eqref{Case3} remains highly challenging. 

\medskip

\begin{lemma}
\label{L:Stro:Key-1} Let  $3/2<\alpha<2$. It holds that 
\begin{align}
 \intQT \frac{\ve}{\ue} |\nabla \ue|^2 \le C_{T} , 
 \label{L:Stro:Key-1:State0}
\end{align}
 and consequently,
 \begin{align}
 \sup_{0<t<T} \intO \frac{|\nabla \ve(t)|^2}{\ve(t)} + \intQT \left( \frac{|\nabla \ve|^4}{\ve^3} + \frac{\ue}{\ve} |\nabla \ve|^2\right) \le C_{T}.
 \label{L:Stro:Key-1:State}
 \end{align} 
\end{lemma}
\begin{proof} Let us first prove \eqref{L:Stro:Key-1:State0}. We prove this lemma by using the logarithmic energy (for the component $\ue$) and Lemma \ref{L:st:vNablav}. Indeed, the following computation is obvious  
\begin{align*}
\intQT \frac{\ve}{\ue} |\nabla \ue|^2 &= \intO \log \ue - \intO \log (u_0 +\eps) - \ell \intQT \ve + \intQT \ue^{\alpha-2}\ve \nabla \ue \cdot \nabla \ve \\
  & \le C + \frac{1}{2}\intQT \frac{\ve}{\ue} |\nabla \ue|^2 + \frac{1}{2} \intQT \ue^{2\alpha-3}\ve|\nabla \ve|^2 , 
\end{align*}
Here Jensen's inequality and the mass bound give
$\intO\log u_\varepsilon(t)\le C$, while
\begin{align*}
 -\intO\log(u_0+\varepsilon)
 \le |\Omega|+\intO u_0^{-1}<\infty,
\end{align*}
uniformly in $\varepsilon$, 
and therefore,
\begin{align*}
 \intQT \frac{\ve}{\ue} |\nabla \ue|^2
  &\le C + \intQT \ue^{2\alpha-3}\ve|\nabla \ve|^2  . 
 \end{align*} 
Since $0<2\alpha-3<1$, an interpolation can estimate $\ue^{2\alpha-3}$ due to $\ue^ 0=1$ and $\ue $. Hence, 
 \begin{align*}
  \intQT \ue^{2\alpha-3}\ve|\nabla \ve|^2 \le C_\alpha \intQT (1 + \ue ) |\nabla \ve|^2 \le C_T, 
 \end{align*}
where the estimate \eqref{L:Basic:MixUV:State1} in Lemma \ref{L:Basic:MixUV} has been used.  
Combining the above estimates, we directly get the estimate \eqref{L:Stro:Key-1:State0}. 

\medskip

Next, we will show \eqref{L:Stro:Key-1:State}. Let us consider the evolution of $\ve^{-1}|\nabla \ve|^2$, which corresponds to 
 \begin{align}
 \begin{aligned}
 \frac{d}{dt } \intO \frac{|\nabla \ve|^2}{\ve} &+ 2\intO \ve|D^2\log \ve|^2 + \intO \frac{\ue}{\ve} |\nabla \ve|^2 = - 2 \intO \nabla \ue \cdot \nabla \ve + \int_{\Gamma} \frac{1}{\ve} \frac{\pa |\nabla \ve|^2}{\pa \nu} ,
 \end{aligned}
  \label{L:Stro:Key-1:P2}
 \end{align}
see the proof of Lemma 3.5 \cite{winkler2022small}. Here, by the convexity of the boundary $\Gamma$, the last term is nonpositive. For the mixed gradient term, 
\begin{align}
 - 2 \intO \nabla \ue \cdot \nabla \ve \leq \frac{1}{2} \intO \frac{\ue}{\ve} |\nabla \ve|^2 + 2 \intO \frac{\ve}{\ue} |\nabla \ue|^2. 
 \label{L:Stro:Key-1:P3}
\end{align}
Taking estimates \eqref{L:Stro:Key-1:State0} and \eqref{L:Stro:Key-1:P2}-\eqref{L:Stro:Key-1:P3}, 
 we obtain
 \begin{align*}
  \intO \frac{|\nabla \ve|^2}{\ve} + \intQT \left( \ve|D^2\log \ve|^2 + \frac{\ue}{\ve} |\nabla \ve|^2 + \frac{\ve}{\ue} |\nabla \ue|^2 \right) \le C_{T} . 
 \end{align*}
Finally, the desired estimate \eqref{L:Stro:Key-1:State} is derived by noting $$\intQT \frac{|\nabla \ve|^4}{\ve^3} \leq C_T \intQT \ve|D^2 \log \ve|^2 $$ due to Lemma \ref{L:Mod:Phi-NablaPhi}. 
\end{proof}

\begin{lemma} 
\label{L:Stro:Key0} Let  $3/2<\alpha<2$ and let $m\in \mathbb{R}$. For all $0< \eta\ll1$, it holds 
\begin{align}
 \intQT \ue^{m+\frac{8}{3}} \ve \le C_{T,m,\eta}+ C \intQT \ue^{m} \ve|\nabla \ue|^2  + 
 \eta \intQT \ue^{2m+4} \ve ,
\label{L:Stro:Key0:State}
\end{align}
where all the constants depend on $m$.
\end{lemma}

\begin{proof} We prove this lemma by modifying the proof of Lemma \ref{L:Mod:Key0} as follows, where using the Young inequality gives 
\begin{align*}
 C_{m} \intQT \ue^{m+2} \ve^{-1}|\nabla \ve|^2 &\leq C_{m,\eta} \intQT \frac{|\nabla \ve|^4}{\ve^3} + \eta \intQT \ue^{2m+4} \ve \\
  &\leq C_{T,m,\eta} + \eta \intQT \ue^{2m+4} \ve , 
 \end{align*}
where we applied Lemma \ref{L:Stro:Key-1} to estimate $\ve^{-3} |\nabla \ve|^4$. 
\end{proof}

If we can control the second and third terms on the right-hand side of \eqref{L:Stro:Key0:State}, then we can choose $m$ such that $m+8/3 = 2m+4$, i.e. $m= -4/3$ to get an estimate for $4/3$. A control of the terms including $\ue^m \ve|\nabla \ue|^2 $  will be obtained in the lemma below.

\begin{lemma} \label{L:Stro:Key} 
Let $3/2<\alpha<2$. For any $-\alpha <r <3-2\alpha$, 
 \begin{align}
  \sup_{0<t<T}\intO \ue^{r+1}(t) + \intQT \ue^{r} \ve|\nabla \ue|^2+ \intQT \ue^{\frac{4}{3}} \ve 
  \le C_{T,r} ,
  \label{L:Stro:Key:State1}
 \end{align}
 and consequently, 
 \begin{align}
 \sup_{0<t<T} \intO \frac{|\nabla \ve(t)|^2}{\ve(t)} + \intQT \left( \frac{|\nabla \ve|^4}{\ve^3} + \frac{\ue}{\ve} |\nabla \ve|^2 + \frac{\ve}{\ue} |\nabla \ue|^2 \right) \le C_T.
  \label{L:Stro:Key:State2}
 \end{align}
\end{lemma}

\begin{proof} We recall from Lemma \ref{L:Mod:Key} the following computation
\begin{align}
\label{L:Stro:Key:Proof0}
 \begin{aligned}
  \intQT \ue^{q-1} \ve|\nabla \ue|^2 
  &=\frac{1}{q(q-1)} \intO (u_0 +\eps)^q - \frac{1}{q(q-1)} \intO \ue^q + \frac{\ell}{q-1} \intQT \ue^q \ve \\
  & - \frac{1}{q-1+\alpha} \intQT 
  \ue^{q-1+\alpha} \left( |\nabla \ve|^2 + \ve \Delta \ve \right), 
\end{aligned}
\end{align}
for $q\in \mathbb{R}\setminus\{0;1\}$. Let us choose $r$ such that 
\begin{align*}
 -\alpha <r < \min\left( -\frac{4}{3};\,-\alpha+\frac{1}{3} \right)
\end{align*}
and take $q=r+1$. We note that $r\not \in \{-1;0\}$. Since $r+1<-1/3$ and $r+\alpha>0$, 
\begin{align*}
  \intQT \ue^{r} \ve|\nabla \ue|^2 
  &=\frac{1}{r(r+1)} \intO (u_0 +\eps)^{r+1} - \frac{1}{r(r+1)} \intO \ue^{r+1} + \frac{\ell}{r} \intQT \ue^{r+1} \ve \\
  & - \frac{1}{r+\alpha} \intQT 
  \ue^{r+\alpha} \left( |\nabla \ve|^2 + \ve \Delta \ve \right) \\
  &\le \frac{1}{r(r+1)} \intO (u_0 +\eps)^{r+1} - \frac{1}{r(r+1)} \intO \ue^{r+1} - \frac{1}{r+\alpha} \intQT 
  \ue^{r+\alpha} \ve \Delta \ve . 
\end{align*} 
The initial term is bounded independently of $\varepsilon$: indeed,
$-1<r+1<0$, $(u_0+\varepsilon)^{r+1}\le u_0^{r+1}$ a.e., and
$u_0^{r+1}\le1+u_0^{-1}$, which is integrable by
Assumption \ref{Ass:Initial}.
Analogously to the proof of Lemma \ref{L:Mod:Key}, one can use the heat regularisation, applied to the equation for $\ve $, to estimate the latter term and the Young inequality to have that  
\begin{align}
  \intO \ue^{r+1} + \intQT \ue^{r} \ve|\nabla \ue|^2 
  &\le C_{T,r,\alpha} \left( 1 + \intQT \ue^{r+\alpha+1} \ve \right) . 
\label{L:Stro:Key:P1}
\end{align}
Now, let us choose $m=-4/3$ and $\eta=1/2$ in Lemma \ref{L:Stro:Key0}, which gives,  
\begin{align*}
 \intQT \ue^{\frac{4}{3}} \ve \le C_{T}+ C \intQT \ue^{-\frac{4}{3}} \ve|\nabla \ue|^2  + 
 \frac{1}{2} \intQT \ue^{\frac{4}{3}} \ve , 
\end{align*}
and therefore,
\begin{align}
 \intQT \ue^{\frac{4}{3}} \ve \le C_{T}+ C \intQT \ue^{-\frac{4}{3}} \ve|\nabla \ue|^2 .
  \label{L:Stro:Key:P1b}
\end{align}
Then, combining estimates \eqref{L:Stro:Key:P1} and \eqref{L:Stro:Key:P1b} gives
\begin{align*}
 & \intO \ue^{r+1} + \intQT \ue^{r} \ve|\nabla \ue|^2+ \intQT \ue^{\frac{4}{3}} \ve \\
 & \le C_{T,r,\alpha} \left( 1+ \intQT \ue^{-\frac{4}{3}} \ve|\nabla \ue|^2 + \intQT \ue^{r+\alpha+1} \ve \right) .
\end{align*}
One can find $0<\sigma \ll 1$ such that $-4/3<3-2\alpha-\sigma$ and $\ue^{3-2\alpha-\sigma} \ve|\nabla \ue|^2$ is uniformly bounded in $L^1(Q_T)$ due to Lemma \ref{L:Basic:MixUV}. Since $r< -4/3<3-2\alpha-\sigma$, we can interpolate between $\ue^{-4/3}$ and $\ue^{3-2\alpha-\sigma}$ to have that
\begin{align}
\label{L:Stro:Key:P3}
\begin{aligned}
 C(T) \intQT \ue^{-\frac{4}{3}} \ve|\nabla \ue|^2 &\le C_T \intQT \ue^{3-2\alpha-\sigma} \ve|\nabla \ue|^2 + \frac{1}{2} \intQT \ue^{r} \ve|\nabla \ue|^2 \\
 &\le C_T + \frac{1}{2} \intQT \ue^{r} \ve|\nabla \ue|^2 . 
\end{aligned}
\end{align}
Moreover, it follows from $r<-\alpha+1/3$ that $r+\alpha+1<4/3$, and therefore, 
\begin{align}
 C(T)\intQT \ue^{r+\alpha+1} \ve \le C_T + \frac{1}{2} \intQT \ue^{\frac{4}{3}} \ve .
\label{L:Stro:Key:P4}
\end{align}
Taking estimates \eqref{L:Stro:Key:P3}-\eqref{L:Stro:Key:P4} together, we get 
\begin{align*}
 \intO \ue^{r+1} + \intQT \ue^{r} \ve|\nabla \ue|^2+ \intQT \ue^{\frac{4}{3}} \ve \le C_T, 
\end{align*}
for any $-\alpha <r < \min\left( -\frac{4}{3};\,-\alpha+\frac{1}{3} \right)$. The constants are uniform for $0<t<T$, so the first term may be replaced by its time supremum.  If a target $r_*\in(-\alpha,3-2\alpha)$ already belongs to this initial interval, there is nothing to prove.  Otherwise choose
\begin{align*}
 r_-\in(-\alpha,\min\{-4/3,-\alpha+1/3\}),
 \qquad r_-<r_*,\qquad
 r_+\in(\max\{r_*,-1\},3-2\alpha).
\end{align*}
Lemma \ref{L:Basic:MixUV} supplies the gradient estimate at $r_+$, and $s^{r_*}\le s^{r_-}+s^{r_+}$ extends the gradient bound to $r_*$.  It remains to check the moment.  Since $r_*+1<1$, the mass bound controls it
when $r_*+1\ge0$; if $r_*+1<0$, then $s^{r_*+1}\le1+s^{r_-+1}$, so the already proved negative moment controls
it.  The $\intQT u_\varepsilon^{4/3}\ve $ estimate is independent of the target exponent.  This proves all three terms in \eqref{L:Stro:Key:State1} on the full stated interval. 
\end{proof}

\subsection{Combination of the evolution of $\ue^{p} \ve^q$, $\ve^{-r}|\nabla\ve|^{r+1}$ and $\ue^{s}$}
\label{Sec:Combina}

Another earlier key in our analysis is to estimate the term of the general form
$$\intO \ve^{-r} |\nabla \ve|^{r+1} , \quad \text{for } 0 \le r\le 6.$$ For this purpose, we suggest considering a gradient-like energy for the system, which is a suitable linear combination of the evolutions of $\ue \log \ue$ and $|\nabla \ve|^4/\ve^3$. 

\begin{lemma}
\label{L:St:upvq} Let $3/2 < \alpha <2$. For  $p,q \in \mathbb{R}$, it holds that
 \begin{align}
 \begin{split}
  \frac{d}{dt} \intO \ue^p \ve^q& = p(1-p)\intO \ue^{p-1} \ve^{q+1} |\nabla \ue|^2 + p q\intO \ue^{p-1+\alpha} \ve^{q} |\nabla \ve|^2 + p \ell \intO \ue^{p} \ve^{q+1} \\
  & - p (1-p)\intO \ue^{p-2+ \alpha} \ve^{q+1} \nabla \ue \cdot \nabla \ve - p q\intO \ue^{p} \ve^{q} \nabla \ve \cdot \nabla \ue \\
  & - pq \intO \ue^{p-1} \ve^{q-1} \nabla \ve \cdot \nabla \ue -  q(q-1) \intO \ue^p \ve^{q-2} |\nabla \ve|^2 - q \intO \ue^{p+1} \ve^q.
 \end{split}
 \label{L:St:upvq:S}
\end{align}
\end{lemma}

\begin{remark} We note that the above lemma holds for any $\alpha \in \mathbb{R}$, where it will be made rigorous if, after integrating over time, all terms on the right-hand side exist finitely. 
\end{remark}

\begin{proof}[Proof of Lemma \ref{L:St:upvq}] By the product rule of differentiation, we will calculate the integration of the quantities $\ve^q \pa_t \ue^p$ and $\ue^{p} \pa_t \ve^q$. First, using the equation for $\ue$, we have 
\begin{align*}
 \intO \ve^q \pa_t \ue^p 
= \,&\, p \intO \ue^{p-1} \ve^q \left( \nabla (\ue \ve \nabla \ue ) - \nabla (\ue^{\alpha} \ve \nabla \ve ) + \ell \ue\ve \right). 
\end{align*}
Performing integration by parts, it is straightforward to check that 
\begin{align*}
 \begin{split}
  \intO \ve^q \pa_t \ue^p 
  =\,&\, -p \intO \ue\ve((p-1)\ue^{p-2} \ve^q \nabla \ue + q\ue^{p-1} \ve^{q-1} \nabla \ve) \cdot \nabla \ue + p \ell \intO \ue^{p} \ve^{q+1} \\
  \,&\,+ p \intO \ue^{\alpha} \ve ((p-1)\ue^{p-2} \ve^q \nabla \ue + q\ue^{p-1} \ve^{q-1} \nabla \ve) \cdot \nabla \ve \\
  =\,&\, - p(p-1)\intO \ue^{p-1} \ve^{q+1} |\nabla \ue|^2 + p q\intO \ue^{p-1+\alpha} \ve^{q} |\nabla \ve|^2 + p \ell \intO \ue^{p} \ve^{q+1} \\
  & + p (p-1)\intO \ue^{p-2+ \alpha} \ve^{q+1} \nabla \ue \cdot \nabla \ve - p q\intO \ue^{p} \ve^{q} \nabla \ve \cdot \nabla \ue . 
 \end{split}
\end{align*} 
Now, using the equation for $\ve$, we have $\ue^{p} \pa_t \ve^q = q \ue^p \ve^{q-1} \left( \Delta \ve - \ue\ve\right)$. Therefore, performing the integration by parts again, we get 
\begin{align*}
 \begin{split}
  \intO \ue^{p} \pa_t \ve^q 
  &=-q \intO \nabla \ve \cdot \left( p \ue^{p-1} \ve^{q-1} \nabla \ue + (q-1) \ue^p \ve^{q-2} \nabla \ve \right) - q \intO \ue^{p+1} \ve^q\\
  & = -pq \intO \ue^{p-1} \ve^{q-1} \nabla \ve \cdot \nabla \ue - q(q-1) \intO \ue^p \ve^{q-2} |\nabla \ve|^2 - q \intO \ue^{p+1} \ve^q .
 \end{split}
\end{align*}
The expression \eqref{L:St:upvq:S} is obtained by taking the above computations together. 
\end{proof}

\begin{lemma}
\label{L:St:Comb} Given $\frac{3}{2}<\alpha<2$ and $0<p<1$. 
 \begin{itemize}
  \item[a)] For any real number $\rho\gg1$, it holds that 
\begin{align}
\begin{aligned}
 & \intQT \ue^{p-1} \ve^{2} |\nabla \ue|^2 + \intQT \ue^{p-1+\alpha} \ve|\nabla \ve|^2 + \rho \sup_{0<t<T} \intO \frac{|\nabla \ve(t) |^{4}}{\ve^{3}(t)} \\
  &  + \rho \intQT \ve|\nabla \ue|^2 + 2 \rho \intQT \frac{|\nabla \ve|^{2}}{\ve}|D^{2}\log \ve|^{2} 
 + \rho \intQT \ue \frac{|\nabla \ve|^4}{\ve^3} \\
  & \le C_{T,p,\alpha} \left( C_\rho + \intQT \ue^{p+1} \ve + \intQT  \ue^{p-1} |\nabla\ve \cdot \nabla \ue| \right. \\
  & \hspace{2.3cm} \left. \,+ \rho \intQT \ue^{2\alpha -2} \ve|\nabla \ve|^2 + \intQT \ue^{p} \ve|\nabla\ve \cdot \nabla \ue| \right).
\end{aligned}
 \label{L:St:Comb:S1}
\end{align} 
  \item[b)] Assume furthermore that $\frac{11}{6}\le\alpha<2$ and $ \alpha-1 < p < 1$. Then, for any real numbers $\eta\ll1$ and  $\rho\gg1$ can be independently chosen, it holds 
  \begin{align}
\begin{aligned}
& 
 \intQT \ue^{p+1-\alpha} \ve|\nabla \ue|^2 + \intQT \ue^{p-1} \ve^{2} |\nabla \ue|^2 + \intQT \ue^{p-1+\alpha} \ve|\nabla \ve|^2 \\
& + \rho \sup_{0<t<T} \intO \frac{|\nabla \ve(t) |^{4}}{\ve^{3}(t)} + 2 \rho \intQT \frac{|\nabla \ve|^{2}}{\ve}|D^{2}\log \ve|^{2} 
 + \rho \intQT \ue \frac{|\nabla \ve|^{4}}{\ve^{3}} \\ 
& \le C_{T,p,\alpha} \left( C_\rho + \intQT \ue^{p+1} \ve + C_\eta \intQT  \ue^{p-1} |\nabla\ve \cdot \nabla \ue| \right. \\
  & \hspace{2.3cm} \left. + \rho \intQT \ue^{2\alpha -2} \ve|\nabla \ve|^2 + \eta \intQT \ue^{p} \ve|\nabla\ve \cdot \nabla \ue| \right).
\end{aligned}
 \label{L:St:Comb:S2}
\end{align}
 \end{itemize}
\end{lemma}

\begin{remark}
  A sufficiently large constant $\rho\gg1$ will be properly chosen in Lemma \ref{L:St:B:Key-1}, where we need to absorb the integral on $Q_T$ of $\ue \ve^{-3}|\nabla \ve|^4$ into the left-hand side of \eqref{L:St:Comb:S1}. 
\end{remark}

\begin{proof}[Proof of Lemma \ref{L:St:Comb}] a) To prove this part, we will combine the Lemmas \ref{L:St:upvq} and \ref{L:st:vNablav}, together with the boundedness obtained in Section \ref{Sec:St:m43}. Let us restrict $0<p<1$ in Lemma \ref{L:St:upvq}, and so that the first two terms on the right-hand side of \eqref{L:St:upvq:S} are nonnegative. Moving the remaining terms to the left-hand side, this estimate gives 
\begin{align*}
 & p(1-p)\intO \ue^{p-1} \ve^{q+1} |\nabla \ue|^2 + p q\intO \ue^{p-1+\alpha} \ve^{q} |\nabla \ve|^2 \\
  & = \frac{d}{dt} \intO \ue^p \ve^q - p \ell \intO \ue^{p} \ve^{q+1} + q \intO \ue^{p+1} \ve^q \\
  & + p (1-p)\intO \ue^{p-2+ \alpha} \ve^{q+1} \nabla \ue \cdot \nabla \ve + p q\intO \ue^{p} \ve^{q} \nabla \ve \cdot \nabla \ue \\
  & + pq \intO \ue^{p-1} \ve^{q-1} \nabla \ve \cdot \nabla \ue +  q(q-1) \intO \ue^p \ve^{q-2} |\nabla \ve|^2 . 
\end{align*} 
Integrating over time, we get 
\begin{align}
\begin{aligned}
  & p(1-p) \intQT \ue^{p-1} \ve^{q+1} |\nabla \ue|^2 + p q\intQT \ue^{p-1+\alpha} \ve^{q} |\nabla \ve|^2 \\
  & = \intO \ue^p \ve^q - \intO (u_0+\eps)^p v_0^q - p \ell \intQT \ue^{p} \ve^{q+1} + q \intQT \ue^{p+1} \ve^q \\
  & + p (1-p)\intQT \ue^{p-2+ \alpha} \ve^{q+1} \nabla \ue \cdot \nabla \ve + p q\intQT \ue^{p} \ve^{q} \nabla \ve \cdot \nabla \ue \\
  & + pq \intQT \ue^{p-1} \ve^{q-1} \nabla \ve \cdot \nabla \ue + q(q-1) \intQT \ue^p \ve^{q-2} |\nabla \ve|^2 .
\end{aligned}
 \label{L:St:Comb:P1}
\end{align}
Let us take $q=1$. Then, by the Young inequality, 
\begin{align}
 \intO \ue^p \ve^q + q \intQT \ue^{p+1} \ve^q \le C_{p} \left( 1 + \intQT \ue^{p+1} \ve \right) . 
 \label{L:St:Comb:P2}
\end{align}
Moreover, since the value of $p-2+ \alpha$ stays between $p-1$ and $p$, we can estimate $\ue^{p-2+ \alpha}$ by interpolating $\ue^{p-1}$ and $\ue^{p}$. More specifically, for any $\eta>0$,  
\begin{align}
 \intQT \ue^{p-2+ \alpha} \ve^{2} \nabla \ue \cdot \nabla \ve \le  \eta \intQT \ue^{p} \ve|\nabla\ve \cdot \nabla \ue| +  C_\eta C_{p,\alpha} \intQT \ue^{p-1} |\nabla\ve \cdot \nabla \ue| ,
 \label{L:St:Comb:P3}
\end{align}
where the constant $C_{p,\alpha}$ depends also on the $L^\infty$-norm of $\ve$. 
Therefore, taking the estimates \eqref{L:St:Comb:P1}-\eqref{L:St:Comb:P3} together, we obtain the estimate 
\begin{align}
 \begin{aligned}
& \intQT \ue^{p-1} \ve^{2} |\nabla \ue|^2 + \intQT \ue^{p-1+\alpha} \ve|\nabla \ve|^2 \\
  & \le C_{p,\alpha} \left( 1 + \intQT \ue^{p+1} \ve + \eta \intQT \ue^{p} \ve|\nabla\ve \cdot \nabla \ue| + C_\eta \intQT  \ue^{p-1} |\nabla\ve \cdot \nabla \ue| \right) .
 \label{L:St:Comb:P4}
 \end{aligned}
\end{align}

Now, using Lemma \ref{L:st:vNablav}, we have 
\begin{align}
\begin{gathered} 
 \intQT \ve|\nabla \ue|^2 + 2 \intQT \frac{|\nabla \ve|^{2}}{\ve}|D^{2}\log \ve|^{2} 
 + \intQT \ue \frac{|\nabla \ve|^{4}}{\ve^{3}} \\ 
 + \sup_{0<t<T} \intO \frac{|\nabla \ve(t) |^{4}}{\ve^{3}(t)} \leq C_T \left(1 + \intQT \ue^{2\alpha -2} \ve|\nabla \ve|^2 \right) , 
 \end{gathered} 
 \label{L:St:Comb:P5}
\end{align}
where, using the boundedness of $\ue^{\frac{4}{3}} \ve$ in Lemma \ref{L:Stro:Key}, 
\begin{align*}
 \intQT \ue \ve \log \ue \le \int_0^T \int_{ \{\ue > 1\} } \ue \ve \log \ue \le C \int_0^T \int_{ \{\ue > 1\} } \ue^{\frac{4}{3}} \ve \le C_T. 
\end{align*} 
Multiplying two sides of \eqref{L:St:Comb:P5} by a sufficiently large constant $\rho\gg1$, we get 
\begin{align}
\begin{gathered} 
 \rho \intQT \ve|\nabla \ue|^2 + 2\rho \intQT \frac{|\nabla \ve|^{2}}{\ve}|D^{2}\log \ve|^{2} 
 + \rho\intQT \ue \frac{|\nabla \ve|^{4}}{\ve^{3}} \\ 
 + \rho \sup_{0<t<T} \intO \frac{|\nabla \ve(t) |^{4}}{\ve^{3}(t)} \leq C_T \left(\rho +\rho \intQT \ue^{2\alpha -2} \ve|\nabla \ve|^2 \right) , 
 \end{gathered} 
 \label{L:St:Comb:P6}
\end{align}
Finally, we derive \eqref{L:St:Comb:S1} by combining \eqref{L:St:Comb:P4} (with $\eta=1$) and \eqref{L:St:Comb:P6}. 

\medskip

\noindent b) To prove this lemma, together with the estimates \eqref{L:St:Comb:P4}-\eqref{L:St:Comb:P5}, the support from the $L^p$-energy will be considered. First, we recall from the proof of Lemma \ref{L:Weak:Feedback} that
\begin{align*}
  \frac{d}{dt} \intO \ue^{\widetilde p}(t) 
& \le - \frac{\widetilde p(\widetilde p-1)}{2} \intO \ue^{\widetilde p-1} \ve|\nabla \ue|^2 + \frac{\widetilde p(\widetilde p-1)}{2} \intO \ue^{\widetilde p-3+2\alpha} \ve|\nabla \ve|^2 + \ell \widetilde p \intO \ue^{\widetilde p} \ve .
\end{align*}
Let us choose $\widetilde p=p+2-\alpha$. Then, 
this $\widetilde p$ is strictly larger than $1$ since $p>\alpha-1$. Moreover, it is obvious that $\widetilde p<p+1$ (since $\alpha \ge 11/6>1$). Therefore, 
\begin{align*}
 \intO \ue^{\widetilde p} \ve \le C_{p,\alpha} \left( \intO \ue \ve + \intO \ue^{p+1} \ve \right) \le C_{p,\alpha} \left( 1 + \intO \ue^{p+1} \ve \right)
\end{align*}
Integrating over time the latter energy estimate yields 
\begin{align}
\begin{aligned}
 \intQT \ue^{p+1-\alpha} \ve|\nabla \ue|^2 \le C^*_{p,\alpha,\ell} \left( 1 + \intQT \ue^{p+1} \ve + \intQT \ue^{p-1+\alpha} \ve|\nabla \ve|^2 \right) . 
\end{aligned} 
\label{L:St:B:Key0a2:P1}
\end{align} 

Now, multiplying two sides of \eqref{L:St:Comb:P4} by $2C^*_{p,\alpha,\ell}$ gives
\begin{align}
\begin{aligned}
 & 2C^*_{p,\alpha,\ell} \intQT \ue^{p-1} \ve^{2} |\nabla \ue|^2 + 2C^*_{p,\alpha,\ell} \intQT \ue^{p-1+\alpha} \ve|\nabla \ve|^2 \\
  & \le C_{\alpha,\delta,\ell} \left( 1 + \intQT \ue^{p+1} \ve + \eta \intQT \ue^{p} \ve|\nabla\ve \cdot \nabla \ue| + C_\eta \intQT  \ue^{p-1} |\nabla\ve \cdot \nabla \ue| \right) ,
\end{aligned}
\label{L:St:B:Key0a2:P3}
\end{align}
which suggests that we can absorb the last term of \eqref{L:St:B:Key0a2:P1}  into the left-hand side of \eqref{L:St:B:Key0a2:P3}. 
This part is ended by taking\eqref{L:St:B:Key0a2:P1}  (with sufficiently small $\eta\ll1$), \eqref{L:St:B:Key0a2:P3} and \eqref{L:St:Comb:P6} together.
\end{proof}

\subsection{The uniform boundedness of $\ue^m\ve$ for $m = 7/3$}
\label{Sec:73}

In the following two lemmas, we will show that the right-hand sides of \eqref{L:St:Comb:S1}-\eqref{L:St:Comb:S2} can be controlled by the $L^1(Q_T)$-norm of $\ue^{7/3}\ve$. In more detail, we present the improvement of Parts a and b of Lemma \ref{L:St:Comb} respectively in 
Lemmas \ref{L:St:B:Key0a} and \ref{L:St:B:Key0a2}.
 
\begin{lemma} \label{L:St:B:Key0a}
 Given $\frac{3}{2}<\alpha<\frac{11}{6}$ and $\frac{3}{4}<p<\frac{5}{6}$. There exists a constant $C_{T,p,\alpha}$ such that  
 \begin{align} 
\begin{aligned}
 & \intQT \ue^{p-1} \ve^{2} |\nabla \ue|^2 + \intQT \ue^{p-1+\alpha} \ve|\nabla \ve|^2 + \rho \sup_{0<t<T} \intO \frac{|\nabla \ve(t) |^{4}}{\ve^{3}(t)} \\
  &  + (\rho - \eta C_{T,p,\alpha}) \intQT \ve|\nabla \ue|^2 + 2 \rho \intQT \frac{|\nabla \ve|^{2}}{\ve}|D^{2}\log \ve|^{2} \\
& + (\rho - 2\eta C_{T,p,\alpha})  \intQT \ue \frac{|\nabla \ve|^4}{\ve^3}  \le  C_{T} + 3\eta C_{T,p,\alpha}\intQT \ue^{\frac{7}{3}}\ve ,
\end{aligned}  
\label{L:St:B:Key0a:S}
\end{align}
for real numbers  $0<\eta\ll1$ and $\rho\gg 1$, where $C_T$ also depends on $\eta,\rho,p,\alpha$. 
\end{lemma}

\begin{proof} This proof will be based on Part a of Lemma \ref{L:St:Comb}, where, for the sake of convenience, we use the following notations
\begin{align*}
 &I_1:= \intQT \ue^{p+1} \ve , \quad I_2:= \intQT  \ue^{p-1} |\nabla\ve \cdot \nabla \ue| \\
  & I_3:= \intQT \ue^{2\alpha -2} \ve|\nabla \ve|^2, \quad I_4:= \intQT \ue^{p} \ve|\nabla\ve \cdot \nabla \ue|.
\end{align*}
These terms will be estimated in the sequel. Moreover, to control the term including $\rho$, which takes place on the right-hand side of \eqref{L:St:Comb:S1}, we first choose $0<\eta\ll1$ so that the coefficients independent of $\rho$ remain positive, and then choose $\rho\gg1$ after $C_\eta$ is fixed. 

\medskip 

\noindent \underline{Estimating $I_1$}: By the clarity $0<p+1<7/3$, the Young inequality shows 
\begin{align}
 I_1 \le \eta \intQT \ue^{\frac{7}{3}}\ve + C_{p,\eta} \intQT \ve \le \eta \intQT \ue^{\frac{7}{3}}\ve + C_{p,\eta} ,
 \label{L:St:B:Key0a:P1}
\end{align}
recalling that a general constant can be different even in the same line.

\medskip 

\noindent \underline{Estimating $I_2$}: Under the condition $3/4<p<5/6$, it is obvious that $-\alpha<2p-3<3-2\alpha$. Therefore, the uniform boundedness of $\ue^{2p-3} \ve|\nabla \ue|^2$ in $L^1(Q_T)$ is ensured by Lemma \ref{L:Stro:Key}. Additionally, Lemma \ref{L:Stro:Key} also reads that $\ue\ve^{-1} |\nabla \ve|^2$ is uniformly bounded in $L^1(Q_T)$, which allows us to estimate $I_2$ as
\begin{align} 
 I_2 & \leq C_p \left( \intQT \frac{\ue}{\ve} |\nabla \ve|^2 + \intQT \ue^{2p-3} \ve|\nabla \ue|^2 \right) \leq C_{T,p}.
  \label{L:St:B:Key0a:P2}
\end{align}

\noindent \underline{Estimating $\rho I_3$}: We estimate this quantity as 
\begin{align*}
 I_3 & \leq \frac{\eta}{\rho} \intQT \ue \frac{|\nabla \ve|^4}{\ve^3} + C_{\eta,\rho} \intQT \ue^{4\alpha -5} \ve \\
 & \leq \frac{\eta}{\rho} \intQT \ue \frac{|\nabla \ve|^4}{\ve^3} + \frac{\eta}{\rho} \intQT \ue^{\frac{7}{3}}\ve + C_{\eta,\rho,\alpha}, 
\end{align*} 
where, 
observing $0<4\alpha - 5< 7/3$ due to $3/2<\alpha<11/6$, the Young inequality shows 
\begin{align*}
 C_{\eta,\rho} \intQT \ue^{4\alpha-5}\ve \le \frac{\eta}{\rho} \intQT \ue^{\frac{7}{3}}\ve + C_{\eta,\rho,\alpha} \intQT \ve \le \frac{\eta}{\rho} \intQT \ue^{\frac{7}{3}}\ve + C_{\eta,\rho,\alpha}.
\end{align*}
Subsequently, we obtain 
\begin{align*}
 \rho I_3= \rho \intQT \ue^{2\alpha -2} \ve|\nabla \ve|^2 \le \eta \intQT \ue\frac{|\nabla \ve|^4}{\ve^3} + \eta \intQT \ue^{\frac{7}{3}}\ve + C_{\eta,\rho,\alpha}. 
\end{align*}
 
\noindent \underline{Estimating $I_4$}: The condition $3/4<p<5/6$ yields that $0<4p-1<7/3$. Hence, an interpolation allows us to estimate $\ue^{4p-1}\ve$ using $\ue^{7/3}\ve$ and $\ve$, similarly to estimating for $I_1,I_3$ above. Indeed, by the Young inequality, 
\begin{align}
\begin{aligned}
 I_4 
 &\leq \eta \intQT \ve|\nabla \ue|^2 + C_{p, \eta} \intQT \ue^{2p} \ve|\nabla \ve|^2 \\
 &\leq \eta \intQT \ve|\nabla \ue|^2 + \eta \intQT \ue \frac{|\nabla \ve|^4}{\ve^3} + C_{p, \eta} \intQT \ue^{4p-1} \ve \\
 &\leq \eta \intQT \ve|\nabla \ue|^2 + \eta \intQT \ue \frac{|\nabla \ve|^4}{\ve^3} + \eta \intQT \ue^{\frac{7}{3}}\ve + C_{p, \eta} ,
\end{aligned}
 \label{L:St:Comb:P544}
\end{align}
where we note that the constant $C_{p, \eta}$ also includes the $L^\infty(Q_T)$-norm of $\ve$. 

\medskip

Finally, thanks to the above estimates for $I_i$, $1\le i\le 4$, and the fact that  
\begin{align*}
 \text{RHS of } \eqref{L:St:Comb:S1} = C_{T,p,\alpha} \left( C_\rho + I_1 + I_2 + \rho I_3 + I_4 \right), 
\end{align*}
where the same constants as \eqref{L:St:Comb:S1} have been used, we get  
\begin{align*}
 \text{RHS of } \eqref{L:St:Comb:S1} \le  C_{T} + C_{T,p,\alpha} \left( 3\eta \intQT \ue^{\frac{7}{3}}\ve + 2 \eta \intQT \ue\frac{|\nabla \ve|^4}{\ve^3} + \eta \intQT \ve|\nabla \ue|^2\right). 
\end{align*}
Here, we note that the constant $C_T$ also depends on $\eta,\rho,p,\alpha$. Therefore, we get \eqref{L:St:B:Key0a:S} by moving the last two terms of the latter estimate to the left-hand side. 
\end{proof}

\begin{remark} We point out here a weak point of Lemma \ref{L:St:B:Key0a}. This will be restricted in turn by a combination with Lemma \ref{L:St:B:Key-1}, where our purpose is to obtain the boundedness of $\ue^{7/3} \ve$ in $L^1(Q_T)$. We can choose $p$ in Lemma \ref{L:St:B:Key0a} as close as possible to $3/4$ such that $2<4p-1<7/3$. However, to control the term
\begin{align*}
 \intQT \ue^{4\alpha -5} \ve, 
\end{align*}
appeared when estimating $I_3$, 
by $\intQT \ue^{7/3} \ve$, we require the assumption $4\alpha - 5 < 7/3$, which is equivalent to $\alpha < 11/6$. 
\end{remark}

\begin{lemma} \label{L:St:B:Key0a2}
 Given $\frac{11}{6}\le \alpha<2$ and $ \alpha-1 <p<1$. There exists a constant $C_{T,p,\alpha}$ such that
 \begin{align}
\begin{aligned} 
 & 
\left(1 -  \frac{\eta}{2} C_{T,p,\alpha}\right) \intQT \ue^{p+1-\alpha} \ve|\nabla \ue|^2 + \intQT \ue^{p-1} \ve^{2} |\nabla \ue|^2 \\
& + \left(1 -  \frac{3\eta}{2} C_{T,p,\alpha}\right) \intQT \ue^{p-1+\alpha} \ve|\nabla \ve|^2 + \rho \sup_{0<t<T} \intO \frac{|\nabla \ve(t) |^{4}}{\ve^{3}(t)} \\
&  + 2 \rho \intQT \frac{|\nabla \ve|^{2}}{\ve}|D^{2}\log \ve|^{2} 
 + \rho \intQT \ue \frac{|\nabla \ve|^{4}}{\ve^{3}} \\ 
& \le  C_T + \eta C_{T,p,\alpha} \intQT \ue^{\frac{7}{3}}\ve  ,  
\end{aligned}
\label{L:St:B:Key0a2:S}
\end{align}
for real numbers  $0<\eta\ll1$ and $\rho\gg 1$, where $C_T$ also depends on $\eta,\rho,p,\alpha$. 
\end{lemma}

\begin{proof} We will prove this lemma by exploiting Part b of Lemma \ref{L:St:Comb}, where we use the same sufficiently large and small constants $\rho$ and $\eta$. Let us use the same notations $I_i$, $1\le i\le 4$, as in the proof of Lemma \ref{L:St:B:Key0a}. The estimates for $I_1$ and $I_2$ are also obtained analogously as Lemma \ref{L:St:B:Key0a}, where we note that $-\alpha<2p-3<3-2\alpha$ since $3/2-\alpha/2<p<1<3-\alpha$. 
We will present different estimates for $I_3$ and $I_4$. 

\medskip

\noindent \underline{Estimating $I_3$}: The condition $p > \alpha - 1$ ensures that $0< 2\alpha - 2 < p-1+\alpha$. Therefore, we have the interpolation 
\begin{align*}
 I_3 = \intQT \ue^{2\alpha -2} \ve|\nabla \ve|^2 & \leq \frac{\eta}{\rho} \intQT \ue^{p-1+\alpha} \ve|\nabla \ve|^2 + C_{p,\alpha,\eta,\rho} \intQT \ve|\nabla \ve|^2 \\
 & \leq \frac{\eta}{\rho} \intQT \ue^{p-1+\alpha} \ve|\nabla \ve|^2 + C_{p,\alpha,\eta,\rho} ,
\end{align*} 
thanked to the boundedness of $\nabla \ve$ in $L^2(Q_T)$ as obtained by Lemma \ref{L:BasicEst:v}. We obtain
 \begin{align}
\rho I_3  & \leq \eta \intQT \ue^{p-1+\alpha} \ve|\nabla \ve|^2 +C_{p,\alpha,\eta,\rho} . 
\label{L:St:B:Key0a2:P1:1}
\end{align} 
In the above lines, it is useful to recall that we used the same constant notation  
 $C_{p,\alpha,\eta,\rho}$, but it can take different values. 

\medskip
 
\noindent \underline{Estimating $I_4$}: We directly estimate this term as  
\begin{align*} 
 I_4 = \intQT \ue^{p} \ve|\nabla\ve \cdot \nabla \ue| &\leq \frac{1}{2} \intQT \ue^{p+1-\alpha} \ve|\nabla \ue|^2 + \frac{1}{2} \intQT u^{p-1+\alpha} \ve|\nabla \ve|^2 ,  
\end{align*}
and consequently, 
\begin{align} 
\eta I_4  &\leq \frac{\eta}{2} \intQT \ue^{p+1-\alpha} \ve|\nabla \ue|^2 + \frac{\eta}{2} \intQT u^{p-1+\alpha} \ve|\nabla \ve|^2 .
\label{L:St:B:Key0a2:P2}
\end{align}

Finally, using the fact that 
\begin{align*}
 \text{LHS of } \eqref{L:St:Comb:S2} = C_{T,p,\alpha} \left( C_\rho + I_1 + C_\eta I_2 + \rho I_3 + \eta I_4 \right), 
\end{align*}
and the estimates \eqref{L:St:B:Key0a:P1}-\eqref{L:St:B:Key0a:P2} for $I_1,I_2$ and \eqref{L:St:B:Key0a2:P1:1}-\eqref{L:St:B:Key0a2:P2} for $I_3,I_4$, we get 
\begin{align*}
 \text{RHS of } \eqref{L:St:Comb:S2} & \le C_{T,p,\alpha} \left( \frac{3\eta}{2} \intQT \ue^{p-1+\alpha} \ve|\nabla \ve|^2 
 + \frac{\eta}{2} \intQT \ue^{p+1-\alpha} \ve|\nabla \ue|^2 \right) \\
 &  + \eta C_{T,p,\alpha} \intQT \ue^{\frac{7}{3}}\ve + C_{T,p,\eta,\rho,\alpha} .
\end{align*}
Hence, it follows from \eqref{L:St:Comb:S2} that 
\begin{align*}
 & 
 \intQT \ue^{p+1-\alpha} \ve|\nabla \ue|^2 + \intQT \ue^{p-1} \ve^{2} |\nabla \ue|^2 + \intQT \ue^{p-1+\alpha} \ve|\nabla \ve|^2 \\
& + \rho \sup_{0<t<T} \intO \frac{|\nabla \ve(t) |^{4}}{\ve^{3}(t)} + 2 \rho \intQT \frac{|\nabla \ve|^{2}}{\ve}|D^{2}\log \ve|^{2} 
 + \rho \intQT \ue \frac{|\nabla \ve|^{4}}{\ve^{3}} \\ 
& \le C_{T,p,\alpha} \left( \frac{3\eta}{2} \intQT \ue^{p-1+\alpha} \ve|\nabla \ve|^2 
 + \frac{\eta}{2} \intQT \ue^{p+1-\alpha} \ve|\nabla \ue|^2 \right) \\
 &  + \eta C_{T,p,\alpha} \intQT \ue^{\frac{7}{3}}\ve + C_{T,p,\eta,\rho,\alpha} .
\end{align*}
We derive \eqref{L:St:B:Key0a2:S}  by moving the first two terms of the latter estimate to the left-hand side. 
\end{proof}

\begin{lemma} 
\label{L:St:B:Key-1} 
Given $\frac{3}{2}<\alpha<2$ and $0<p<1$. 
\begin{itemize}
 \item[a)] If $\frac{3}{2}<\alpha<\frac{11}{6} ,  \, \frac{3}{4}<p<\frac{5}{6}$, it holds that 
\begin{align}
\begin{aligned}
\intQT 
\ue^{\frac{7}{3}} \ve &+ \intQT \ve|\nabla \ue|^2 + \intQT \ue^{p-1+\alpha} \ve|\nabla \ve|^2 + \sup_{0<t<T} \intO \frac{|\nabla \ve(t) |^{4}}{\ve^{3}(t)} \\
&+ \intQT \frac{|\nabla \ve|^{2}}{\ve}|D^{2}\log \ve|^{2} + \intQT \ue \frac{|\nabla \ve|^4}{\ve^3}  \le C_{T,p,\alpha}.  
\end{aligned} 
\label{L:St:B:Key-1:State}
\end{align}

 \item[b)] If $\frac{11}{6}\le \alpha<2 , \, \alpha-1 <p<1$, it holds that
\begin{align}
  \begin{aligned} 
\intQT  \ue^{\frac{7}{3}} \ve& + 
 \intQT \ue^{p+1-\alpha} \ve|\nabla \ue|^2  +  \intQT \ue^{p-1+\alpha} \ve|\nabla \ve|^2 +  \sup_{0<t<T} \intO \frac{|\nabla \ve(t) |^{4}}{\ve^{3}(t)}  \\
&+  \intQT \frac{|\nabla \ve|^{2}}{\ve}|D^{2}\log \ve|^{2} 
 + \intQT \ue \frac{|\nabla \ve|^{4}}{\ve^{3}}  \le C_{T,\eta,p, \alpha} .  
\end{aligned} 
\label{L:St:B:Key-1:State2}
\end{align}
\end{itemize}  
\end{lemma}

\begin{remark} We observe from Part b of Lemma  \ref{L:St:B:Key-1} that the boundedness of $\ue^{p+1-\alpha} \ve|\nabla \ue|^2$ in $L^1(Q_T)$ is better than the one of $\ve|\nabla \ue|^2$ by noting $p+1-\alpha>0$. This is reasonable since its proof is based on Part b of Lemma \ref{L:St:Comb} (and a consequence obtained in Lemma  \ref{L:St:B:Key0a2}), where we use more structures of the system than Part a of Lemma \ref{L:St:Comb}. 
\end{remark}

\begin{proof}[Proof of Lemma \ref{L:St:B:Key-1}] Due to Lemmas \ref{L:St:B:Key0a} and \ref{L:St:B:Key0a2}, we split this proof into two separable cases, including  $3/2<\alpha<11/6$ and $11/6\leq \alpha<2$. Let us consider the first case. Using analogous techniques as Lemma \ref{L:Mod:Key0} with $m=-1/3$, we have 
 \begin{align*}
\intQT \ue^{\frac{7}{3}} \ve 
  &\leq C + C \intQT \ue^{-\frac{1}{3}} \ve|\nabla \ue|^2+ C \intQT \ue^{\frac{5}{3}} \ve^{-1}|\nabla \ve|^2 . 
 \end{align*}
Recalling that the integration on $Q_T$ of $\ue^{r} \ve|\nabla \ue|^2$ is bounded if $r=-1$, see Lemma \ref{L:Stro:Key}, we can estimate the first term on the right-hand side via the interpolation 
\begin{align*}
 C \intQT \ue^{-\frac{1}{3}} \ve|\nabla \ue|^2 & \le \eta \intQT \ve|\nabla \ue|^2 + C_\eta \intQT \frac{\ve}{\ue} |\nabla \ue|^2 \\
 & \le \eta \intQT \ve|\nabla \ue|^2 + C_{T,\eta} . 
\end{align*}
Now, by slightly modifying the use of the Young inequality in Lemma \ref{L:Mod:Key0}, one can see that 
\begin{align}
C \intQT \ue^{\frac{5}{3}} \ve^{-1}|\nabla \ve|^2 &\leq \eta \intQT \ue^{\frac{7}{3}} \ve + C_{\eta} \intQT \ue \frac{|\nabla \ve|^4}{\ve^3} ,
\label{L:St:B:Key-1:P1}
 \end{align}
where we notice that $C_\eta$ is inversely proportional to $\eta$.
 Therefore, taking the above estimates,
\begin{align*}
(1-\eta) \intQT 
\ue^{\frac{7}{3}} \ve \le \eta \intQT \ve|\nabla \ue|^2 + C_{\eta} \intQT \ue \frac{|\nabla \ve|^4}{\ve^3} + C_{T,\eta} . 
\end{align*}
In turn of a combination with Lemma \ref{L:St:B:Key0a}, we get 
\begin{align*}
 \begin{aligned}
 &(1-\eta - 3\eta C_{T,p,\alpha} ) \intQT 
\ue^{\frac{7}{3}} \ve  +  \intQT \ue^{p-1} \ve^{2} |\nabla \ue|^2 + \intQT \ue^{p-1+\alpha} \ve|\nabla \ve|^2 \\
  & + \rho \sup_{0<t<T} \intO \frac{|\nabla \ve(t) |^{4}}{\ve^{3}(t)}+ (\rho -  \eta - \eta C_{T,p,\alpha}) \intQT \ve|\nabla \ue|^2 \\
& + 2 \rho \intQT \frac{|\nabla \ve|^{2}}{\ve}|D^{2}\log \ve|^{2}  + (\rho - 2\eta C_{T,p,\alpha} -  C_{\eta})  \intQT \ue \frac{|\nabla \ve|^4}{\ve^3}\le C_{T,\eta} .
\end{aligned}  
\end{align*}
By letting $0<\eta\ll1$ be sufficiently small and  $\rho$ be sufficiently large such that 
$$\min\Big\{\big(1-\eta - 3\eta C_{T,p,\alpha}\big), \; \big(\rho-\eta- \eta C_{T,p,\alpha}\big), \; \big(\rho - 2\eta C_{T,p,\alpha} -  C_{\eta}\big)\Big\}>0,$$ 
we directly obtain 
\begin{align*}
 \begin{aligned}
 & \intQT \ue^{\frac{7}{3}} \ve  +  \intQT \ue^{p-1} \ve^{2} |\nabla \ue|^2 + \intQT \ue^{p-1+\alpha} \ve|\nabla \ve|^2 +  \sup_{0<t<T} \intO \frac{|\nabla \ve(t) |^{4}}{\ve^{3}(t)}  \\
  &  +  \intQT \ve|\nabla \ue|^2+  \intQT \frac{|\nabla \ve|^{2}}{\ve}|D^{2}\log \ve|^{2}  +  \intQT \ue \frac{|\nabla \ve|^4}{\ve^3}  \le C_{T,\eta,p, \alpha} ,
\end{aligned}  
\end{align*}
i.e. the estimate \eqref{L:St:B:Key-1:State} is proved. 

\medskip
Next, we will base on Lemma \ref{L:St:B:Key0a2} to prove the desired estimate in the second case, i.e., $11/6 \leq \alpha <2$, where the appearance of the integral of $\ue^{p+1-\alpha} \ve|\nabla \ue|^2$ is mainly the key point. Indeed, using the boundedness of $ \ve \ue^{-1} |\nabla \ue|^2$ in Lemma \ref{L:Stro:Key},
\begin{align*}
 C \intQT \ue^{-\frac{1}{3}} \ve|\nabla \ue|^2 & \le  \eta \intQT \ue^{p+1-\alpha} \ve|\nabla \ue|^2 + C_\eta \intQT \frac{\ve}{\ue} |\nabla \ue|^2 \\
 & \le \eta \intQT \ue^{p+1-\alpha} \ve|\nabla \ue|^2  + C_{T,\eta} . 
\end{align*}
Therefore, by following the same  techniques in estimating $  \ue^{7/3} \ve$ as the first case above, 
 \begin{align*} 
\intQT \ue^{\frac{7}{3}} \ve 
  &\leq C + C \intQT \ue^{-\frac{1}{3}} \ve|\nabla \ue|^2+ C \intQT \ue^{\frac{5}{3}} \ve^{-1}|\nabla \ve|^2 \\
  &\leq \eta \intQT \ue^{p+1-\alpha} \ve|\nabla \ue|^2 + \eta \intQT \ue^{\frac{7}{3}} \ve + C_{\eta} \intQT \ue \frac{|\nabla \ve|^4}{\ve^3}  + C_{T,\eta}. 
 \end{align*}
where we used \eqref{L:St:B:Key-1:P1} at the second line. Consequently,
\begin{align*}
(1-\eta) \intQT 
\ue^{\frac{7}{3}} \ve 
& \le \eta \intQT \ue^{p+1-\alpha} \ve|\nabla \ue|^2+ C_{\eta} \intQT \ue \frac{|\nabla \ve|^4}{\ve^3}  + C_{T,\eta}.
\end{align*}
Combining this with Lemma \ref{L:St:B:Key0a2}, we derive 
\begin{align*}
 \begin{aligned} 
 &(1-\eta -\eta C_{T,p,\alpha}) \intQT 
\ue^{\frac{7}{3}} \ve + 
\left(1 -  \frac{\eta}{2} C_{T,p,\alpha} - \eta  \right) \intQT \ue^{p+1-\alpha} \ve|\nabla \ue|^2 \\
& + \left(1 -  \frac{3\eta}{2} C_{T,p,\alpha}\right) \intQT \ue^{p-1+\alpha} \ve|\nabla \ve|^2  + \rho \sup_{0<t<T} \intO \frac{|\nabla \ve(t) |^{4}}{\ve^{3}(t)} \\
&  + 2 \rho \intQT \frac{|\nabla \ve|^{2}}{\ve}|D^{2}\log \ve|^{2} 
 + (\rho - C_{\eta}) \intQT \ue \frac{|\nabla \ve|^{4}}{\ve^{3}}  \le  C_{T,\eta} .  
\end{aligned} 
\end{align*}
By letting $0<\eta\ll1$ be sufficiently small and  $\rho$ be sufficiently large such that 
$$\min\left\{ \big(1-\eta - \eta C_{T,p,\alpha}\big), \, \left( 1 -  \frac{\eta}{2} C_{T,p,\alpha} - \eta \right) , \, \left( 1 -  \frac{3\eta}{2} C_{T,p,\alpha}\right), \, \big(\rho - C_{\eta}\big) \right\}>0,$$ 
we obtain the estimate \eqref{L:St:B:Key-1:State2}. 
\end{proof}

\subsection{Mutual feedback with singularity enhancement}
\label{Sec:Str:FB}

We will perform a feedback argument, which will serve as the bootstrap argument for the next part. As introduced in the first section, we aim to find a component-wise increasing sequence $\{(q_k,p_k,r_k)\}$ satisfying not only the inter-feedback (S1)  but also its ``inverse-type feedback" (S2) and the forward feedback (S3). 
In our bootstrap argument, $q_k$ is obtained first, and $p_k,r_k$ will be determined via $q_k$.

\begin{lemma} 
\label{L:Stro:Sequence}
Let $\frac{3}{2}<\alpha<2$ and $\{(q_k,p_k,r_k)\}_{k=0,1,\dots}$ be defined by 
$\{(q_k,p_k,r_k)\}_{k=0,1,\dots}$ defined by $q_0 > 2\alpha-4$, and for $k\ge 0$,
\begin{align}
p_{k}:= q_{k} + 5 - 2\alpha, \quad 
  r_{k} : =  p_{k}-1,    \quad q_{k+1} := \min\left\{2p_k-2;\, \frac{7}{6}p_k-1 \right\}.  
 \label{Def:5:Seq:a}
\end{align}
Then, the component sequences $\{q_k\}$ and $\{p_k\}$ are strictly increasing, and 
 $$\lim_{k\to \infty} p_k = \lim_{k\to \infty} q_k = \infty .$$
\end{lemma}
 
\begin{proof} We first note that $p_k>1$ for all $k\ge 0$. Indeed, this holds for $k=0$ since $q_0+5-2\alpha>1$. Indeed, if this holds for $k\ge 1$, then the computation
\begin{align*}
  p_{k+1} - p_k = q_{k+1} - q_k & = \min\left\{2p_k-2;\, \frac{7}{6}p_k-1 \right\} - (p_k-5+2\alpha) \\ 
  & =  \left\{\begin{array}{lll}
     p_k+3-2\alpha & \text{if }  p_k \le \dfrac{6}{5}, \vspace{0.1cm} \\
     \dfrac{1}{6}p_k + 4 - 2\alpha   & \text{if }  p_k >\dfrac{6}{5}   
    \end{array}\right.
\end{align*}
shows that $q_{k+1} - q_k>4-2\alpha$, and so $q_{k+1} > q_k$. Since $p_k>1$, we have $q_k=p_k-5+2\alpha>2\alpha-4$, and therefore, $q_{k+1} > q_k>2\alpha-4$. This obviously implies $p_{k+1}:= q_{k+1} + 5 - 2\alpha >1$. Hence, based on the proof by induction, the conclusion $p_k>1$, $\forall k\ge 0$ holds. Then, the increasing property of both $\{p_k\}$ and $\{q_k\}$ are straightforwardly checked. Moreover, it follows from $p_{k+1} - p_k>4-2\alpha$ that $\lim_{k\to \infty} p_k=\infty$, and consequently $\lim_{k\to \infty} q_k=\infty$.
\end{proof}

\begin{lemma} 
\label{L:SeqProp}
Let $\{(q_k,p_k,r_k)\}_{k=0,1,\dots}$ be the  sequence defined in Lemma \ref{L:Stro:Sequence}. Then, for all $k\ge 0$, it has the following properties. 
\begin{itemize}
    \item [$(\mathbb P_1)$] $\displaystyle -\alpha < \frac{3}{2} q_{k+1}-\frac{5}{3}p_k+1 \le r_k$, \quad and \quad $\displaystyle -\alpha <  q_{k+1}-\frac{i}{6}p_k \le r_k$, $\forall i=1,...,4$.
    \item [$(\mathbb P_2)$] $\displaystyle 0\le \frac{3}{2} q_{k+1}-\frac{5}{3}p_k+3 \le q_{k+1}-\frac{2}{3}p_k+2$.   
\end{itemize}
    
\end{lemma}

\begin{proof} Since the proof of this lemma is based on direct computation, we present here only the proof of the property $(\mathbb P_2)$, and leave the remaining part for the reader. By the definition of the bootstrap sequence $\{(q_k,p_k,r_k)\}_{k\ge 1}$,  
\begin{align*}
    & \left(\dfrac{3}{2} q_{k+1}-\frac{5}{3}p_k+3\right) - \left(q_{k+1}-\dfrac{2}{3}p_k+2\right) = \dfrac{1}{2}q_{k+1} - p_k + 1 \\
    & \hspace{1cm}  = \frac{1}{2} \min\left\{2p_k-2 ;\, \frac{7}{6}p_k-1 \right\} - p_k + 1 
    =  
    \left\{\begin{array}{lll}
     0    & \text{if }  p_k \le \dfrac{6}{5}, \vspace{0.1cm} \\
     \dfrac{1}{2} - \dfrac{5}{12}p_k   & \text{if }  p_k >\dfrac{6}{5},  
    \end{array}\right.
\end{align*}
which yields the most right-hand inequality of $(\mathbb P_2)$. Similarly, the quantity $\frac{3}{2} q_{k+1}-\frac{5}{3}p_k+3$ is equal to $4p_k/3$ if $p_k\le 6/5$ and equal to $p_k/12+3/2$ if $p_k> 6/5$, which obviously shows the most left-hand inequality of $(\mathbb P_2)$. 
\end{proof}

\begin{lemma}
\label{L:Feedback3}
Let $\frac{3}{2}<\alpha<2$ and $\{(q_k,p_k,r_k)\}_{k=0,1,\dots}$ be the sequence defined by \eqref{Def:5:Seq:a}. If 
\begin{align} 
 \intQT \ue^{q_{k}+2}  \ve  |\nabla \ve|^2  \le C_T,
 \label{L:Feed1:S0}
\end{align}
then 
\begin{align}\label{L:Feed1:S1}
 \sup_{0<t<T}\intO \ue^{p_{k}}(t) + \intQT \ue^{r_{k}} \ve  |\nabla \ue|^2  \le C_{T,k}. 
\end{align}
\end{lemma} 

\begin{proof} Recalling from the proof of Lemma \ref{L:Weak:Feedback}, one has from the $L^p$-energy functional that 
\begin{align*}
  & \frac{d}{dt} \intO \ue^{ p_{k}}(t) 
 + \frac{ p_{k}( p_{k}-1)}{2} \intO \ue^{ p_{k}-1} \ve|\nabla \ue|^2 \\
 &\le  \frac{ p_{k}( p_{k}-1)}{2} \intO \ue^{ p_{k}-3+2\alpha} \ve|\nabla \ve|^2 + \ell  p_{k} \intO \ue^{ p_{k}} \ve ,
\end{align*}
where, by the definition of $(q_k,p_k,r_k)$,  
\begin{align*}
 \intQT \ue^{ p_{k}-3+2\alpha} \ve|\nabla \ve|^2 =  \intQT \ue^{q_{k}+2} \ve|\nabla \ve|^2 \le C_T. 
\end{align*}
Hence, this part is proved by employing the uniform boundedness of $\ve$ in  $L^\infty(Q_T)$ and the Gr\"onwall inequality after integrating over time, as well as noticing that $r_{k}=p_{k}-1$.
\end{proof}

\begin{lemma} \label{L:Feedback1} Let $\frac{3}{2}<\alpha<2$ and $\{(q_k,p_k,r_k)\}_{k=0,1,\dots}$ be the sequence defined by \eqref{Def:5:Seq:a}.
  If 
\begin{align}\label{L:Feed1:C2}
 \sup_{0<t<T}\intO \ue^{p_k}(t) + \intQT \ue^{r_k} \ve  |\nabla \ue|^2  \le C_T, 
\end{align}
then 
\begin{align}\label{L:Feed1:S2}
 \intQT \ue^{q_{k+1}-\frac{2}{3}p_k+2}  \frac{|\nabla \ve|^2}{\ve}    \le C_{T,k} .
\end{align}  
\end{lemma}

\begin{proof}  
Thanks to Lemma \ref{L:St:B:Key-1}, the quantity$\ve^{-3}|\nabla\ve|^{4}$ is uniformly bounded in $L^\infty(0,T;L^1(\Omega))$. Hence, using the H\"older inequality, for any $\eta>0$ we have 
 \begin{align*}
 \begin{aligned}
 \intQT \ue^{q_{k+1}-\frac{2}{3}p_k+2}  \frac{|\nabla \ve|^2}{\ve} & = \intQT  \left( \ue^{\frac{3}{2} q_{k+1}-p_k+3} \ve \right)^{\frac{2}{3}}  \left(\frac{|\nabla \ve|^6}{\ve^{5}}\right)^{\frac{1}{3}}    \\
 & \leq \eta \intQT  \ue^{\frac{3}{2} q_{k+1}-p_k+3} \ve  + C_\eta \intQT \frac{|\nabla \ve|^6}{\ve^{5}} .
 \end{aligned}
\end{align*}
Due to the  boundedness of $\ue$ in $L^\infty(0,T;L^{p_k}(\Omega))$, we now apply  Lemma \ref{L:SOBOineqn} to estimate the term including $\ue^{3q_{k+1}/2-p_k+3} \ve$ as follows 
\begin{align*}
 \intO  \ue^{\frac{3}{2} q_{k+1}-p_k+3} \ve  & = \intO  \ue^{\frac{3}{2} q_{k+1}-\frac{5}{3}p_k+3}  \ve \ue^{\frac{2}{3}p_k}   \le \|\ue \|_{L^\infty(0,T;L^{p_k}(\Omega))}^{\frac{2}{3}p_k}  \|\ue^{\frac{3}{2} q_{k+1}-\frac{5}{3}p_k+3}  \ve\|_{L^3(\Omega)}  \\
& \le C_T \left(1 + \intO \ue^{\frac{3}{2} q_{k+1}-\frac{5}{3}p_k+1} \ve |\nabla \ue|^2 + \intO \ue^{\frac{3}{2} q_{k+1}-\frac{5}{3}p_k+3} \frac{|\nabla \ve|^2}{\ve} \right) .
\end{align*}
Combining the above estimates gives \begin{align*} 
  \begin{split}
& \intQT \ue^{q_{k+1}-\frac{2}{3}p_k+2}  \frac{|\nabla \ve|^2}{\ve}  \\
& \hspace{2cm} \le  C_{T,\eta} + \eta C_T \intQT \ue^{\frac{3}{2} q_{k+1}-\frac{5}{3}p_k+1} \ve |\nabla \ue|^2 + \eta C_T \intQT \ue^{\frac{3}{2} q_{k+1}-\frac{5}{3}p_k+3} \frac{|\nabla \ve|^2}{\ve}  .
\end{split} 
  \end{align*}
Here, thanks to the property $(\mathbb{P}_1)$ in Lemma \ref{L:SeqProp}, we note that $-\alpha <\frac{3}{2} q_{k+1}-\frac{5}{3}p_k+1 \le r_k$, and therefore, we can use the integrability of $\ue^r\ve|\nabla\ue|^2$ on $Q_T$ for all $-\alpha <r\le r_k$ (due to the estimate \eqref{L:Stro:Key:State1} in Lemma \ref{L:Stro:Key} and the assumption \eqref{L:Feed1:C2}) to conclude the boundedness of the term including $ \ue^{3 q_{k+1}/2-5p_k/3+1} \ve |\nabla \ue|^2$ above. To estimate the last term on the right-hand side, we note that its coefficient can be sufficiently small by choosing $\eta$ very small. Moreover, thanks to the property $(\mathbb{P}_2)$ in Lemma \ref{L:SeqProp}, we have the interpolation 
\begin{align*}
    \eta C_T \intQT \ue^{\frac{3}{2} q_{k+1}-\frac{5}{3}p_k+3} \frac{|\nabla \ve|^2}{\ve}
    & \le C_T \intQT \left(1 + \ue^{q_{k+1}-\frac{2}{3}p_k+2} \right) \frac{|\nabla \ve|^2}{\ve}  \\
    & \le \eta C_T  \left(\intQT \sqrt{\ve} \left( \frac{|\nabla \ve|^4}{\ve^3} \right)^{\frac{1}{2}} + \intQT \ue^{q_{k+1}-\frac{2}{3}p_k+2} \frac{|\nabla \ve|^2}{\ve} \right) \\
    & \le   C_{T,\eta} + \eta C_T \intQT \ue^{q_{k+1}-\frac{2}{3}p_k+2} \frac{|\nabla \ve|^2}{\ve} .
\end{align*}
Hence, we finally obtain
\begin{align*} 
  \begin{split}
 \intQT \ue^{q_{k+1}-\frac{2}{3}p_k+2}  \frac{|\nabla \ve|^2}{\ve} \le  C_{T,\eta} + \eta C_T \intQT \ue^{q_{k+1}-\frac{2}{3}p_k+2} \frac{|\nabla \ve|^2}{\ve} 
\end{split} 
  \end{align*}
which shows \eqref{L:Feed1:S2} by choosing $\eta$ sufficiently small. 
\end{proof}

\begin{lemma}\label{L:Feedback2} Let $\frac{3}{2}<\alpha<2$ and $\{(q_k,p_k,r_k)\}_{k=0,1,\dots}$ be the sequence defined by \eqref{Def:5:Seq:a}. 
If \eqref{L:Feed1:S0}
 holds, i.e., 
\begin{align*} 
\intQT  u^{q_k+2} v |\nabla v|^2 \le C_T ,
\end{align*}
then the estimate \eqref{L:Feed1:S2} obtained in Lemma \ref{L:Feedback1}, i.e., 
\begin{align*}
\intQT \ue^{q_{k+1}-\frac{2}{3}p_k+2}  \frac{|\nabla \ve|^2}{\ve}    \le C_{T,k}, 
\end{align*}
consequently shows that  
\begin{align}\label{L:Feedback2:S}
 \intQT \ue^{q_{k+1}+2}  \ve  |\nabla \ve|^2  \le C_{T,k} .
\end{align}

\end{lemma}

\begin{proof} By Lemma \ref{L:Feedback3}, we  note that \eqref{L:Feed1:S0} implies  
\eqref{L:Feed1:C2}, which will be utilised in this proof.  
To prove  \eqref{L:Feedback2:S}, we first observe for all real number $a\in \mathbb{R}$ and all integer $n\ge 1$  that
\begin{align}
\begin{aligned}
& \intQT \ue^{q_{k+1}+2}  \ve^{\frac{a}{2}}  |\nabla \ve|^2  \\
& \hspace{1cm} \le C_T \left(1 + \sum_{i=1}^n \intQT \ue^{q_{k+1}-\frac{i}{6}p_k} \ve^{\frac{a}{2}+\frac{4-i}{2}}  |\nabla \ue|^2 + \intQT \ue^{q_{k+1}+2-\frac{n}{6}p_k} \frac{|\nabla \ve|^2}{\ve^{\frac{n}{2}-\frac{a}{2}}} \right).   
\end{aligned}
\label{L:Feedback2:P1}
\end{align}
Let us prove this with $n=1$, which will be iterated for any $n\ge 1$. Indeed, by writing $\ue^{q_{k+1}+2}  \ve^{a/2}  |\nabla \ve|^2$ as the product of $\ue^{q_{k+1}+2}  \ve^{a/2+3/2}$ and $|\nabla \ve|^2/\ve^{3/2}$, and then using the boundedness of $\ve^{-3} |\nabla \ve|^4$ in $L^\infty((0,T);L^1(\Omega))$, we have   
\begin{align*}
  \intQT \ue^{q_{k+1}+2}  \ve^{\frac{a}{2}}  |\nabla \ve|^2 &\le \left( \sup_{0<t<T} \intO \frac{|\nabla \ve(t)|^4}{\ve^3(t)} \right)^{\frac{1}{2}} \int_0^T  \|\ue^{q_{k+1}+2} \ve^{\frac{a}{2}+\frac{3}{2}}\|_{L^2(\Omega)} \\
  & \le C_T \int_0^T  \|\ue^{q_{k+1}+2} \ve^{\frac{a}{2}+\frac{3}{2}}\|_{L^2(\Omega)} = C_T \int_0^T \left( \intO \ue^{2q_{k+1}+4-\frac{1}{3}p_k} \ve^{a+3} \ue^{\frac{1}{3}p_k} \right)^{\frac{1}{2}} \\
  & \le C_T \|\ue\|_{L^\infty((0,T);L^{p_k}(\Omega))}^{\frac{1}{6}p_k} \int_0^T  \|\ue^{2q_{k+1}+4-\frac{1}{3}p_k} \ve^{a+3}\|_{L^{\frac{3}{2}}(\Omega)} ^{\frac{1}{2}} \\
  & = C_T \|\ue\|_{L^\infty((0,T);L^{p_k}(\Omega))}^{\frac{1}{6}p_k} \int_0^T  \|\ue^{q_{k+1}+2-\frac{1}{6}p_k} \ve^{\frac{a}{2}+\frac{3}{2}}\|_{L^3(\Omega)},  
\end{align*}
thanks to \eqref{L:Feed1:C2}. 
An application of   Lemma \ref{L:SOBOineqn} gives
\begin{align*}
    \|\ue^{q_{k+1}+2-\frac{1}{6}p_k} \ve^{\frac{a}{2}+\frac{3}{2}}\|_{L^3(\Omega)} \le C_T \left( 1 + \intO \ue^{q_{k+1}-\frac{1}{6}p_k} \ve^{\frac{a}{2}+\frac{3}{2}} |\nabla \ue|^{2} + \intO \ue^{q_{k+1}+2-\frac{1}{6}p_k} \frac{|\nabla \ve|^{2}}{\ve^{\frac{1}{2}-\frac{a}{2}}}  \right) ,
\end{align*}
which subsequently shows \eqref{L:Feedback2:P1} for $n=1$. Repeating the above techniques for the term including $\ue^{q_{k+1}+2-p_k/6} \ve^{a/2-1/2} |\nabla \ve|^2 $ and so on, we get
\begin{align}
    \begin{aligned}
& \intQT \ue^{q_{k+1}+2}  \ve^{\frac{a}{2}}  |\nabla \ve|^2  \\
& \hspace{1cm} \le C_T \left(1 +  \intQT \ue^{q_{k+1}-\frac{1}{6}p_k} \ve^{\frac{a}{2}+\frac{3}{2}}  |\nabla \ue|^2 + \intQT \ue^{q_{k+1}+2-\frac{1}{6}p_k} \frac{|\nabla \ve|^2}{\ve^{\frac{1}{2}-\frac{a}{2}}} \right) \\
& \hspace{1cm} \le C_T \left(1 + \sum_{i=1}^2 \intQT \ue^{q_{k+1}-\frac{i}{6}p_k} \ve^{\frac{a}{2}+\frac{4-i}{2}}  |\nabla \ue|^2 + \intQT \ue^{q_{k+1}+2-\frac{2}{6}p_k} \frac{|\nabla \ve|^2}{\ve^{\frac{2}{2}-\frac{a}{2}}} \right) \\
& \hspace{1cm} ... \\
& \hspace{1cm} \le C_T \left(1 + \sum_{i=1}^n \intQT \ue^{q_{k+1}-\frac{i}{6}p_k} \ve^{\frac{a}{2}+\frac{4-i}{2}}  |\nabla \ue|^2 + \intQT \ue^{q_{k+1}+2-\frac{n}{6}p_k} \frac{|\nabla \ve|^2}{\ve^{\frac{n}{2}-\frac{a}{2}}} \right),
\end{aligned}
\label{L:Feedback2:P2}
\end{align}
i.e., the estimate \eqref{L:Feedback2:P1} is proved. 

\medskip

We now use the above observation in combination with \eqref{L:Feed1:S2}  to show \eqref{L:Feedback2:S}. By setting $a=2$ and $n=4$, the estimate \eqref{L:Feedback2:P1} corresponds to 
\begin{align}
\begin{aligned}
& \intQT \ue^{q_{k+1}+2}  \ve   |\nabla \ve|^2  \\
& \hspace{1cm} \le C_T \left(1 + \sum_{i=1}^4 \intQT \ue^{q_{k+1}-\frac{i}{6}p_k} \ve^{1+\frac{4-i}{2}}  |\nabla \ue|^2 + \intQT \ue^{q_{k+1}+2-\frac{2}{3}p_k} \frac{|\nabla \ve|^2}{\ve} \right) \\
& \hspace{1cm} \le C_T \left(1 + C_T\sum_{i=1}^4 \intQT \ue^{q_{k+1}-\frac{i}{6}p_k} \ve   |\nabla \ue|^2 + \intQT \ue^{q_{k+1}+2-\frac{2}{3}p_k} \frac{|\nabla \ve|^2}{\ve} \right),   
\end{aligned}
\label{L:Feedback2:P3}
\end{align}
where the last term on the right-hand side is bounded using  \eqref{L:Feed1:S2}. On the other hand, the summation in \eqref{L:Feedback2:P3} is bounded by combining the property  $(\mathbb{P}_1)$ in Lemma \ref{L:SeqProp} and the integrability of $\ue^r\ve|\nabla\ue|^2$ on $Q_T$ for all $-\alpha <r\le r_k$, where we recall the estimate \eqref{L:Stro:Key:State1} in Lemma \ref{L:Stro:Key} and the assumption \eqref{L:Feed1:C2}. 
\end{proof}

\subsection{Bootstrap based on the singularity-enhancing feedback}
\label{Sec:Str:BA}

\begin{lemma}
\label{L:St:B:Key0c}
Given $\frac{3}{2}<\alpha<2$,
there exists $q_0>2\alpha-4$ such that
 \begin{align*}
 \intQT \ue^{q_0+2} \ve |\nabla \ve|^2 \le C_T.
 \end{align*} 
\end{lemma}

\begin{proof} Thanks to Lemma \ref{L:St:B:Key-1}, for any real number $p$ satisfying   $$\max\left\{\frac{3}{4};\alpha-1\right\}\mathbf{1}_{\alpha<\frac{11}{6}}+ (\alpha-1)\mathbf{1}_{\alpha\ge\frac{11}{6}}< p < \frac{5}{6} \mathbf{1}_{\alpha<\frac{11}{6}}+ \mathbf{1}_{\alpha\ge\frac{11}{6}},$$
we have the integrability of $\ue^{(p-3+\alpha)+2}\ve |\nabla \ve|^2$ on $Q_T$. Since $p-3+\alpha>2\alpha-4$, the desired estimate of this lemma is claimed.
\end{proof}

\begin{lemma}
\label{L:Stro:Boot} 
Let $3/2<\alpha<2$. For any $1\le p<\infty$, it holds that 
 \begin{align} 
  \sup_{0<t<T} \intO \ue^{p}(t)  \leq C_{T,p},
  \label{Coro:Str:LInf:S1}
 \end{align}
where $C_{T,p}$ may tend to infinity as $p\to \infty$. Consequently, 
\begin{align}
 0< C_{T,v_0} \le \|v_\eps\|_{L^\infty(Q_T)} \quad \text{and} \quad \|\ue\|_{L^\infty(Q_T)} \le C_T ,
 \label{Coro:Str:LInf:S2}
\end{align}
and $T_{\max,\varepsilon} = \infty$ for all $\varepsilon>0$, i.e. the global solvability of the regularised problem is claimed.
\end{lemma}
\begin{proof} Let us consider the sequence $\{(q_k,p_k,r_k)\}_{k=0,1,\dots}$ be the sequence defined by \eqref{Def:5:Seq:a}. In the following, we will write a general constant $C_{T,k}>0$ to indicate the $k$-th step of our bootstrap argument. 
By Lemmas \ref{L:St:B:Key0c}, there exists $q_0>2\alpha-4$ such that 
\begin{align} 
 \intQT \ue^{q_0+2}  \ve  |\nabla \ve|^2 \le C_{T,0}. 
 \label{L:Stro:Boot:P1}
\end{align} 
Hence, using Lemma \ref{L:Feedback3}, we get
\begin{align*}
 \sup_{0<t<T}\intO \ue^{p_0}(t) + \intQT \ue^{r_0} \ve  |\nabla \ue|^2 \le C_{T,0}.
\end{align*}
Thanks to Lemma \ref{L:St:B:Key0c} with the boundedness of $\ue$ in $L^\infty(0,T;L^{p_0}(\Omega))$,  Lemma \ref{L:Feedback1} can be applied to see that  
\begin{align*}
 \intQT \ue^{q_1-\frac{2}{3}p_0+2} \frac{|\nabla \ve|^2}{\ve} \le C_{T,0}. 
\end{align*}
Consequently, Lemma \ref{L:Feedback2} implies  
\begin{align*}
 \intQT \ue^{q_{1}+2}  \ve |\nabla \ve|^2 \le C_{T,1}. 
\end{align*}
Iterating the above argument, we obtain  
\begin{align*}
 \sup_{0<t<T}\intO \ue^{p_k}(t)  + \intQT \ue^{r_k} \ve  |\nabla \ue|^2 \le C_{T,k}, \quad \forall k\ge 0. 
\end{align*}
By Lemma 
\ref{L:Stro:Sequence}, the sequence $\{p_k\}$ tends to infinity. Then, there exist some values of $k$ that are larger than $p$. The estimate \eqref{Coro:Str:LInf:S2} is derived using the same arguments as Lemma \ref{L:Weak:Global}.
\end{proof}

\subsection{Global existence of a weak solution} 
 
 \begin{proof}[\underline{Proof of Theorem \ref{MainTheo} with $\frac{3}{2}<\alpha < 2 $}] One can be proved similarly to the case $1<\alpha \le \frac{3}{2} $. 
\end{proof}  
 
\vspace*{0.5cm}
 
\noindent \textbf{\Large Competing Interests} The authors declare that they have no competing interests.


\end{document}